\documentclass[11pt]{amsart}\usepackage{mathtools,amssymb,latexsym,graphics,enumerate}
\usepackage[mathscr]{eucal}
\usepackage{amsmath,amsfonts,amsthm,amssymb,bbm,mathtools}
\usepackage{color}
\usepackage{cite,soul,cancel}
\usepackage[left=1in,right=1in,top=1in,bottom=1in]{geometry}
\usepackage[normalem]{ulem}
\usepackage{nicefrac}

\usepackage{graphicx}
\usepackage{subcaption}
\usepackage{float}

\numberwithin{equation}{section}

\renewcommand{\geq}{\geqslant}
\renewcommand{\ge}{\geqslant}
\renewcommand{\leq}{\leqslant}
\renewcommand{\le}{\leqslant}

\newcommand{\rr}{\mathbb{R}}

\newcommand{\lan}{\langle}
\newcommand{\ran}{\rangle}
\newcommand{\be}{\begin{eqnarray*}}
\newcommand{\bel}{\begin{eqnarray}}
\newcommand{\ee}{\end{eqnarray*}}
\newcommand{\eel}{\end{eqnarray}}
\newcommand{\ba}{\begin{aligned}}
\newcommand{\ea}{\end{aligned}}
\newcommand{\de}{\Delta}
\newcommand{\al}{\alpha}
\newcommand{\na}{\nabla}
\newcommand{\ep}{\epsilon}
\newcommand{\f}{\frac}

\newcommand{\bp}{{\mathbf{p}}}
\newcommand{\bq}{{\mathbf{q}}}
\newcommand{\bx}{{\mathbf{x}}}
\newcommand{\dx}{\textnormal{d}\bx}
\newcommand{\dt}{\textnormal{d}t}
\newcommand{\ddt}{\frac{\textnormal{d}}{\dt}}
\newcommand{\ds}{\textnormal{d}s}
\newcommand{\by}{{\mathbf{y}}}
\newcommand{\dS}{{\textnormal{d}S}}
\newcommand{\dpp}{{\textnormal{d}\mathbf{p}}}
\newcommand{\dy}{\textnormal{d}\by}
\newcommand{\D}{{\mathbb{D}}}
\newcommand{\bk}{{\mathbf{k}}}
\newcommand{\bl}{{\boldsymbol{\ell}}}

\newcommand{\mf}{\mathfrak}

\newcommand{\pa}{\partial}
\newcommand{\wh}{\widehat}
\newcommand{\wt}{\widetilde}
\newcommand{\lf}{\left}
\newcommand{\rg}{\right}
\newcommand{\te}{\theta}

\newtheorem{theorem}{Theorem}
\newtheorem{cor}{Corollary}

\newtheorem{lem}{Lemma}

\newtheorem{remark}{Remark}

\numberwithin{remark}{section}
\numberwithin{lem}{section}
\numberwithin{theorem}{section}
\numberwithin{cor}{section}
\numberwithin{pro}{section}

\newcommand\Torus{{\mathbb T}}

\newcommand{\dss}{\displaystyle}
\newcommand{\vv}{\mathbb{V}}
\newcommand{\uu}{{\mathbb{U}}}
\newcommand{\ww}{{\mathbb{W}}}
\newcommand{\zz}{Z}

\newcommand{\mr}{\mathrm}
\newcommand\step[1]{\noindent {\bf Step \##1}}

\newcommand{\pw}[2]{ {#1}^{(#2)}} 
\newcommand{\A}{\mathcal{A}}
\newcommand{\dd}{\mathbbm{d}}
\newcommand{\msc}{\mathscr}
\newcommand{\br}[1]{{\left\langle #1\right\rangle}}

\newcommand{\myh}[1]{}
\newcommand{\mys}[1]{{#1}}

\newcommand{\n}{\ensuremath{\nonumber}}

\newcommand{\cG}{\mathcal{G}}

\definecolor{mygreen}{rgb}{0.1,0.75,0.2}

\providecommand{\bbs}[1]{\left(#1\right)}

\newcommand{\pt}{\partial}

\providecommand{\bbs}[1]{\left(#1\right)}
\newcommand{\aaa}[1]{\begin{equation}
begin{aligned} #1 \end{aligned}
\end{equation}}

\newcommand{\ud}{\,\mathrm{d}}
\newcommand{\8}{\infty}

\newcommand{\hs}{\mathcal{H}}

\newcommand{\bR}{\mathbb{R}}

\newcommand{\bv}{\mathbf{v}}

\newcommand{\bb}{{\mathbf b}}

\renewcommand{\vec}[1]{\ensuremath{\boldsymbol{#1}}}

\mathtoolsset{showonlyrefs}
\allowdisplaybreaks

\usepackage[dvipsnames]{xcolor}
\usepackage[colorlinks=true, pdfstartview=FitV, linkcolor=RoyalBlue,citecolor=ForestGreen, urlcolor=blue]{hyperref}

\usepackage{varwidth}

\date{\today}

\title{Sampling and Optimization {meet Enhanced} Flows}
\author{Yuan Gao} \address{Department of Mathematics, Purdue University, West Lafayette, IN, USA}
\email{gao662@purdue.edu}
\author{Siming He}\address{Department of Mathematics, University of South Carolina, Columbia, SC, USA}
\email{siming@mailbox.sc.edu}
\author{Eitan Tadmor}
\address{Department of Mathematics and IPST, University of Maryland, College Park, MD, USA}
\email{tadmor@umd.edu}

\thanks{\textbf{Acknowledgment.} The research of YG is supported by NSF CAREER Award DMS-2440651. SH is supported by NSF grant DMS-2406293 and would like to thank Wuchen Li for helpful discussions. The research of ET was supported by ONR grant N00014-2412659 and NSF grant DMS-2508407.}

\begin{document}

\begin{abstract}
It is well known that the computational realization of Gibbs probability measures, $e^{-\mathbb{U}(\mathbf{x})}/Z$, plays a central role in sampling and optimization. In this paper, we introduce two types of dynamics that exhibit rapid convergence towards these Gibbs measures. The mechanism driving this rapid convergence is the enhanced dissipation associated with these transport-diffusion dynamics. Motivated by these enhanced dynamics, we design numerical algorithms for sampling from the target Gibbs measure. Finally, we provide the corresponding particle systems that may yield other effective numerical samplers. 
\end{abstract}

\keywords{ Langevin Dynamics, Enhanced dissipation, Mixing}
\subjclass[2020]{Primary 35B40, 35Q84;
  Secondary 35Q35, 37A25, 65C05}
\maketitle

{\small 
\tableofcontents
}

\vspace*{-1cm}
\section{Introduction}
A classical problem in sampling theory and optimization is to design dynamical systems whose solutions converge rapidly towards a target 
\begin{align}\label{Trgt}
0\leq \Pi(\bx)= \frac{ e^{-{\uu}(\bx)}}{\zz},\quad  \bx \in \Torus^d=(\rr/2\pi \mathbb Z)^d.
\end{align} 
The normalization constant $ \zz:=\int_{\Torus^d} e^{-{\uu}(\bx)}\dx$ ensures that $\Pi$ is a probability density. 
{The efficient numerical realization of general densities $\Pi$ is fundamental in two distinct but intertwined areas (see, e.g., \cite{trillos2023optimization}). In \emph{sampling theory}, specifically in Bayesian inference, sampling from a fixed probability distribution $\Pi$ is crucial for approximating statistically relevant quantities--such as expectations and variances--of high-dimensional observables; see, e.g., \cite{RobertCasella04, trillos2023optimization}. On the other hand, in the context of \emph{optimization}, approximating the target $\Pi={e^{-\uu}}/{Z}$ helps characterize the global optimum of the associated target function $\uu$. 

Although many methods have been developed, significant challenges remain in the vast literature on this problem. The first numerical challenge arises when one attempts to compute the value of the normalization constant $Z$--particularly in the high-dimensional setting, $d\gg 1$. In spite of the straightforward theoretical representation of the constant $Z$, numerical implementation of the corresponding high-dimensional quadrature is challenging; see, e.g., \cite{Haber70,SloanWozniakowski97Intractability,Novak16Complexity}.   
We mention {direct estimators} based on lattice methods~\cite{SloanJoe94}, bridge and path sampling~\cite{GelmanMeng98} and annealed importance sampling~\cite{Neal01}. 
Another widely applied technique, which {circumvents direct} computation of the normalization constant $\zz$, is the {design of} \emph{Langevin sampling dynamics} whose long-time behavior converges toward the target density $\Pi$; see, e.g., \cite{Pavliotis14}. 
We recall the gradient flow associated with \emph{overdamped Langevin sampling dynamics}, which is given by
\cite{Pavliotis14,Altschuler22,Bhattacharya78,
Parisi81,RobertsTweedie96,GirolamiCalderhead11},
\begin{align}\label{MCMC}
\pa_t \rho=\nu\de_\bx\rho+\nu\na_\bx\cdot(\rho\na_\bx\uu),\quad \rho(t=0,\bx)=\rho_0(\bx),\quad \bx\in \Torus^d.
\end{align}
Here, the solution $\rho\geq 0$ is a probability density, and the parameter $\nu^{-1}\geq 1$ encodes the intrinsic time scale of the Langevin dynamics.  One readily checks that \(Z^{-1} e^{-\mathbb U}\) is a stationary solution of \eqref{MCMC}, 
and hence the target measure \(\Pi\) can be approximated by the solution of \eqref{MCMC} as \(t\to\infty\).
However, the convergence \(\rho(t) \stackrel{t\rightarrow \infty}{\longrightarrow}\Pi\) can be exceedingly slow; see, e.g., \cite{christie2025speeding}. In this context, one may consider the following scenario, in which the function $\uu$ has multiple ``valleys'' surrounding local minima, and the initial density is concentrated in one of them. Then, it can take the Langevin dynamics \eqref{MCMC} an extremely long time to escape these ``valleys'' around local minima before successfully exploring other parts of the landscape of $\uu$. However, one expects that the judicious introduction of an external drift can drive the density away from these locally trapped states and enable it to explore the landscape effectively. In conclusion, we aim to design new dynamics to address the following question: 

{\centering
\bf Question:} \emph{By incorporating external drift, is it possible to design dynamics whose solution converges to a given smooth density $\Pi(\bx)=e^{-\uu(\bx)}/\zz$ within a short time scale?}}

Initial progress toward this goal was made in \cite{christie2025speeding}, where the authors considered the following modified Langevin dynamics (in PDE form):
\begin{align} \label{drift_defect}&\pa_t \rho+\na_\bx \cdot (\vv \rho)=\nu \Delta_\bx \rho+\nu \na_\bx\cdot (\rho\na_\bx\uu).  
\end{align}
Here the external time-dependent drift $\vv$ is chosen such that the \textit{drift defect}, 
\begin{align}\label{D}
\mathbb{D}:=-e^{\uu}\na_\bx\cdot(e^{-\uu}\vv),
\end{align}
is identically zero, i.e., $\mathbb D\equiv 0$. The vanishing of the {drift defect} ensures that the dynamics preserves the invariant measure $\Pi=e^{-\uu}/Z$. This can be verified by rewriting the dynamics in terms of the ratio $h=\rho/\Pi$ and observing that $h\equiv 1$ is a stationary solution. In \cite{christie2025speeding}, the authors identify delicate, rapidly alternating random drifts $\vv$ that exhibit strong mixing properties \cite{BedrossianBlumenthalPunshonSmith19,BlumenthalCotiZelatiGvalani23} while preserving a vanishing drift defect $\mathbb D$. As a consequence, rapid convergence of the dynamics towards $\Pi$ is established theoretically. However, due to the randomness and rapid alternation involved in the construction, it is numerically challenging to simulate the dynamics for long times. Finally, the choice of $\vv$ is constrained by the requirement that $\D\equiv0$, which also complicates the numerical implementation.  

Motivated by the aforementioned goal and the early development of \cite{christie2025speeding}, we propose two types of enhanced dynamics that exhibit rapid convergence towards the target distribution $\Pi$ while reducing the complexity of the drift $\vv$. We start by introducing a \textit{first-order model} incorporating alternating shear flows, then we present a \textit{second-order kinetic model} with a fixed drift. Both models are supplemented with their \textit{simplified versions} and the quadrature-free \textit{mass-searching dynamics} for numerical implementation.

\subsection{The First-Order Model}
Given the target distribution $\Pi=Z^{-1}e^{-\uu}$ and a \emph{divergence-free} ambient drift $\vv$, we recall the drift defect $\D$ \eqref{D} associated with $\uu$ and $\vv$,
\begin{align*}
\mathbb{D}=-e^{\uu}\na_\bx\cdot(e^{-\uu}\vv)= \vv\cdot\na_\bx\uu,
\end{align*}
and introduce the first-order  \emph{mass-preserving enhanced sampling} model,
\begin{equation}\label{EQ:1st_Ord}\begin{split}&\displaystyle\pa_t \rho+\overbrace{\vphantom{\int} \na_\bx \cdot(\vv\rho)}^{\text{Transport}}=\overbrace{\vphantom{\int}\nu\de_\bx \rho+\nu\na_\bx\cdot(\rho \na_\bx{\uu})}^{\text{Gradient Flow}}  +\overbrace{\rho\int \lf(\mathbb{D}(\by)-\mathbb{D}(\bx)\rg) \rho(\by) \dy}^{\mathbb D\text{-alignment }\mathcal{Q}_{\uu,\vv}[\rho]},\\
&\na_\bx\cdot\vv=0,\quad\rho(t=0,\bx)=\rho_{0}(\bx),\quad (t,\bx)\in \rr_+\times \Torus^d.   \end{split}
\end{equation}
The model describes the time evolution of the probability density $\rho\geq0$. 
Here, the viscosity parameter $\nu>0$ characterizes the strength of the classical Langevin dynamics (``gradient flow''). We focus on the parameter regime $0<\nu\ll 1$, emphasizing the dominant roles of transport and alignment. Next, we observe that if the drift $\vv$ is absent, the model \eqref{EQ:1st_Ord} reduces to the classical equation \eqref{MCMC}. Thus, a suitably chosen divergence-free vector field $\vv$ can be viewed as an external control to enhance the convergence of the Langevin dynamics. The key question to address is the \emph{convergence rate}: below we measure the \emph{relative error}, $ \lf\|\frac{\rho(t,\cdot) }{\Pi(\cdot)}-1 \rg\|_{L^2}$, 
proving that our enhanced model with properly chosen control $\vv$ leads to a rapid
exponential convergence in time. 
Finally, in order to preserve the invariant measure $\Pi$ of the dynamics, the transport term $\na_\bx\cdot (\vv\rho)$ on the left-hand side of \eqref{EQ:1st_Ord} is coupled {on the right} with the  
\emph{defect alignment} term ($\mathbb D$-alignment), 
which can be reformulated as follows 
\[
 \mathcal{Q}_{\uu,\vv}[\rho](t,\bx)
 =-\rho(t,\bx)\vv(t,\bx)\cdot \na_\bx {\uu}(\bx) +\rho(t,\bx)\int \rho(t,\by)\, \vv(t,\by)\cdot \na_\by {\uu}(\by) \dy
.
\] 
Furthermore, by replacing the divergence-free condition $\na_\bx\cdot\vv\equiv0$ by the vanishing-defect constraint $\mathbb D\equiv 0$, one recovers the model in \cite{christie2025speeding}. 
The commutator structure in the $\mathbb{D}$-alignment $\mathcal{Q}_{\uu,\vv}[\rho]$ appears in various alignment/flocking dynamics
\cite{ShvydkoyTadmor17,griffin2019consensus,Grindrod88,MotschTadmor14,HeTadmor192}.
 The role of such an antisymmetric alignment 
 mechanism {is} to drive the density $\rho$ towards {an} ``environmental average'', {which in our case will be the desired distribution $\Pi$}. Moreover, the 
antisymmetric form of the alignment term guarantees mass conservation, $\|\rho(\cdot, t)\|_{L^1_\bx}=\|\rho_{0}\|_{L_\bx^1}=1$.

Our first main theorem illustrates that it is possible to design the dynamics \eqref{EQ:1st_Ord} such that its solutions $\rho(t,\cdot)$ exhibit fast convergence to the target distribution $\Pi$. 
\begin{theorem}\label{thm_1st_order}
Consider the enhanced dynamics \eqref{EQ:1st_Ord} associated with the target probability density $0<\Pi\in C^2(\Torus^d)$ with $2\leq d\in 2 \mathbb{N}$. Assume that the initial data is smooth and normalized: $0\leq \rho_{0}\in C^\infty(\Torus^d)$ and $\|\rho_0\|_{L^1}=1$. There exist time-dependent velocity fields $\vv\in C_{t,\bx}^\infty$, constants $C_\dagger=C_\dagger(\rho_{0},Z, \uu)$, $\delta=\mathcal{O}(1)>0$ independent of $d$ and $\nu$, and a threshold $\nu_0(\rho_{0}, \uu)$ such that the following estimate holds for all $\nu\in (0,\nu_0]$: 
\begin{align}\label{eq:relative-estimate}
\lf\|\f {\rho(t)}{\Pi}-1\rg\|_{L^2}\leq& C_\dagger  e^{-\delta\nu^{1/2}t},\quad \forall t\geq 0.
\end{align} 
Moreover, the velocity fields $\vv$ have concrete expressions. 
\end{theorem}
{For the classical Langevin dynamics \eqref{MCMC}, the expected convergence time scale is $\mathcal{O}(\nu^{-1})$. By contrast, the convergence time scale for the model \eqref{EQ:1st_Ord} is $\mathcal{O}(\nu^{-\frac12})$, which is significantly shorter than the classical time scale in the transport-dominated regime $\nu\ll 1$. This is the so-called \emph{enhanced-dissipation/relaxation-enhancement phenomenon} in the literature \cite{BCZ15,ConstantinEtAl08}, and we will refer to $\nu^{\f12}$ as the  \emph{enhanced sampling rate}. In the work \cite{christie2025speeding}, by incorporating fast-alternating randomized drift, the authors can achieve an even faster theoretical enhanced sampling rate $|\log\nu|^{-1}$. However, our model is simpler and more flexible in the sense that a large family of time-dependent vector fields $\vv$ ensures the fast convergence stated in \eqref{eq:relative-estimate}. In this paper, we focus on the time-periodic alternating sine-shear flows of period $\nu^{-1/2}$, see \eqref{u_ell}. Our argument extends to other vector fields $\vv$ exhibiting the enhanced dissipation outlined in \eqref{ED_V}, e.g., the randomized Pierrehumbert flows analyzed in  \cite{BlumenthalCotiZelatiGvalani23,christie2025speeding} serve as natural alternatives.}

{Motivated by Theorem \ref{thm_1st_order}, one seeks effective discretizations of the dynamics  \eqref{EQ:1st_Ord} which would lead to fast sampling/optimization algorithms. One such discretization, based on a particle system, is proposed in
Appendix \ref{sec:App-D}.} 
Another approach, which will be our main focus, is to simulate the PDE dynamics directly. We emphasize that even though the $\mathbb D$-alignment term on the right of \eqref{EQ:1st_Ord} involves integration
$\int \rho\, \vv\cdot \na {\uu} \dy$, \emph{no quadrature is needed} to access the solution  $\rho(t)$ of the PDE. Indeed, it can be verified, see Lemma \ref{lem:Connections}, that  
$ \rho(t,\cdot)=\frac{\widehat \rho(t,\cdot)}{\|\widehat \rho(t,\cdot)\|_{L^1}}$, where $\widehat \rho(t,\cdot)$ satisfies the following \emph{simplified first-order dynamics},
\begin{align}\label{EQ:1st_Ord_sim}\quad&\begin{cases}&\displaystyle\pa_t \widehat\rho+ \vphantom{\int}\vv \cdot\na_\bx \widehat\rho  = {\vphantom{\int}\nu\de_\bx\widehat \rho+\nu\na_\bx\cdot(\widehat \rho \na_\bx{\uu})}    -\widehat\rho\vv\cdot \na_\bx {\uu} ,\\
&\na_\bx\cdot\vv(t,\bx)=0,\quad\widehat \rho(t=0,\bx)=\rho_{0}(\bx),\quad (t,\bx)\in \rr_+\times \Torus^d.   \end{cases}
\end{align}
Furthermore,  a separate quadrature-free \emph{mass-searching dynamics} will be introduced to compute the total mass $\|\widehat \rho\|_{L^1}$ at any fixed time $t$ (Section \ref{sec:numerics}). 
Consequently, the simulation of the dynamics \eqref{EQ:1st_Ord} can be decomposed into two steps: 
  \begin{enumerate}
  \item Simulate  the simplified dynamics \eqref{EQ:1st_Ord_sim} to obtain $\widehat\rho(t)$. 
  \item Use the ``mass-searching dynamics'' to determine its mass $\|\widehat\rho\|_{L^1}$ at a fixed time $t$.
   \end{enumerate} 
  By forming the quotient
   $\rho(t,\cdot)=\frac{\widehat \rho(t,\cdot)}{\|\widehat \rho(t,\cdot)\|_{L^1}}$, we end up with a   quadrature-free scheme to simulate $\rho$,  which in turn yields a rapidly convergent algorithm for sampling $\Pi$.
\subsection{The Second-Order Model}
One potential drawback of the first-order model is that the external drift controls $\vv$ are chosen manually in our context and depend on time. Numerical realizations of such vector fields can be challenging, see, e.g., \cite{christie2025speeding}. 
We next develop an alternative second-order sampling/optimization dynamics with a simpler, static external drift control. {
To set up the dynamics to sample $\Pi=Z^{-1}e^{-\uu}$, we consider  the kinetic state space $(\bx,\bp)$  with position $\bx$ and momentum  $\bp$. The assigned drift $\bv$ is determined by the  momentum, and the {drift-defect} $\mathbb{D}$ \eqref{D} is nontrivial unless the target $\Pi$ is constant, since \[
\mathbb{D}(\bx,\bp)=-e^{\uu(\bx)}\na_\bx\cdot(e^{-\uu(\bx)}\bv(\bp))=\bv(\bp)\cdot\na\uu(\bx).\] 
With these ingredients introduced, we consider the following second-order kinetic \emph{mass-preserving enhanced sampling} dynamics,}
\begin{align}
\label{EQ:2nd_Ord}\begin{cases}
&\displaystyle\pa_t f+\overbrace{\bv\cdot{\na_\bx} f
}^{\text{Transport}}= \overbrace{\nu\de_\bp f}^{\text{Diffusion}}+\overbrace{\kappa\de_\bx f+ \kappa\na_\bx\cdot\lf(f\na_\bx{\uu}\rg)}^{\text{Langevin}}+\overbrace{\mathcal{Q}_{\uu,\bv}[f]}^{\mathbb D\text{-alignment}},\phantom{\int}\\
&\displaystyle\mathcal{Q}_{\uu,\bv}[f](t,\bx,\bp)=\iint \lf(\D(\by,\bq)-\D(\bx,\bp)\rg)  f(t,\bx,\bp)f(t,\by,\bq) \textnormal{d}\by \textnormal{d}{\bq},\\
&\displaystyle f(t=0,\bx,\bp)=f_{0}(\bx,\bp),\quad (\bx,\bp)\in \Torus^d\times \mathbb{T}^{d}.  \phantom{\int} \end{cases}
\end{align}
In this model, the probability density $f\geq0$ describes the probability of finding agents at position $\bx$ moving with velocity $\bv(\bp):=(\sin(p_1),\sin(p_2),\ldots, \sin(p_d))\in [-1,1]^{d}$. The model is similar to \emph{underdamped Langevin dynamics} (see, e.g., \cite{CaoLuWang23,AlbrittonArmstrongMourratNovack24,
Baudoin17,CattiauxGuillinMonmarche19,Wu01,
Talay02,MattinglyStuartHigham02,NierHelffer05,
DolbeaultMouhotSchmeiser15,
HerauNier04,GrothausStilgenbauer13} and the references therein). 
We assume initial normalization, which will be preserved by the dynamics,
\begin{align}
\|f_{0}\|_{L_{\bx,\bp}^1}=\|f(t)\|_{L_{\bx,\bp}^1}=1,\quad \text{for all }t\geq 0 . \label{2nd_mss}
\end{align} 
The Langevin structure $\kappa\de_\bx f+\kappa\na_\bx\cdot(f\na_\bx \uu)$ helps stabilize the numerical simulation. 
The $\mathbb{D}$-alignment term $\mathcal{Q}_{\uu,\bv}[f]$ is essential to preserve the invariant measure $\Pi\otimes(|\Torus|^{-d}\rm  Lebesgue)$, and can be expressed as follows
\begin{align}
\mathcal{Q}_{\uu,\bv}[f](t,\bx,\bp)
=f(t,\bx,\bp)\Bigl( - \bv(\bp)\cdot\na_\bx {\uu}(\bx)+\iint f(t,\by,\bq)\, \bv(\bq)\cdot \na_\by {\uu}(\by)\; \dy \mr{d}{\bq} \Bigr).
\end{align}As in the model \eqref{EQ:1st_Ord}, the $\mathbb{D}$-alignment ensures mass conservation \eqref{2nd_mss}. 
In a broad sense, the dynamics \eqref{EQ:2nd_Ord} exhibit features similar to those found in second-order models in flocking dynamics (see, e.g., \cite{Vicsek95,CS07,MotschTadmor11, HaTadmor08,FrouvelleLiu12, Shvydkoy21,Shvydkoy2021,Shvydkoy24,GuHe24} and the references therein).

The motivation behind the second-order enhanced dynamics is that, by extending the problem's phase space, one can design a simple drift to speed up convergence. Once the solution is sufficiently close to the target, a hydrodynamic averaging operation yields an approximation of the target measure. Let us consider the following hydrodynamic density:
\begin{align}\label{hydrodnsty}
[f](t,\bx):=\int_{\mathbb{T}^{d}} f(t,\bx,\bp)\dpp.
\end{align}
Our second main theorem states that the hydrodynamic density converges to the target density $\Pi$ at a fast rate. 
\begin{theorem}\label{thm_2nd_order}
Consider equation $\eqref{EQ:2nd_Ord}$ with smooth, normalized initial data $0\le f_{0}\in C^\infty(\Torus^d_\bx\times\Torus^d_\bp)$   satisfying $\|f_0\|_{L^1_{\bx,\bp}}=1$, and with a positive target density $0<\Pi\in C^\infty(\Torus^d)$ ($1\leq d\in\mathbb N $).  There exists a threshold $\nu_0(
d,\uu, f_{0})$ such that for any $0<\nu\leq \nu_0$,  and any $\kappa\in(0,\nu]$, the following estimate holds
\begin{align}
\mathrm{Sample: }\;&\lf\|[f](t)-{\Pi} \rg\|_{L_{\bx}^1}\leq 
\mathcal{C}_1(d,f_0,\uu)\|{\Pi}\|_{L_{\bx}^\infty}\|f_{0} e^\uu\|_{L^2}e^{-\delta\nu^{1/2} t},\quad \forall t\geq 0;\phantom{\int}
\label{Sample_2nd}\\
\mathrm{Approximation: }\;&\lf\|[f](t)-{\Pi} \rg\|_{L_{\bx}^\infty}\leq 
\mathcal{C}_2(d,f_0,\uu,Z)e^{-{\delta_2}\nu^{1/2} t},\quad \forall t\geq 0. \phantom{\int}\label{App_2nd}
\end{align} The parameter $\delta$ {is independent of the dimension $d$ and diffusion coefficient  $\nu$}, but the parameter $\delta_2$ depends on $d$. 
Moreover, for an arbitrary $\varsigma>0$ and any test function $\varphi\in H^{\sigma}_\bx$, the following mixing estimate holds for all $t\in (0,\delta^{-1}\nu^{-1/2}]$:
\begin{align}\label{Sample_2ndMix}
\lf|\int \varphi ( \bx )  \Pi ( \bx )  \dx - \int \varphi(\bx) [f](t,\bx) \dx\rg|\leq \frac{\mathcal{C}_3(d,\varsigma,\Pi, f_0)}{t^{1/2}}\|\varphi\|_{H_\bx^{\sigma}}, \qquad \sigma=\max\{1,d/2+\varsigma\}.
\end{align} 
\end{theorem}\vspace{0.1cm}
\begin{remark}
The integral $\int \varphi[f]\dx$ is the hydrodynamic observable that captures probabilistic information, e.g., means and variances. The mixing estimate \eqref{Sample_2ndMix} implies that the hydrodynamic observable converges to the target with a rate that is independent of the parameter $\nu$. 
\end{remark}
\begin{remark}
Other second-order models give similar results. We mention a
Vicsek-type model, which is adopted in our numerical test in Section \ref{sec:Numerical_Result}:
\begin{align}\label{sphere_model}
\displaystyle\begin{cases}
&\pa_t f+{\bp}\cdot{\na_\bx} f
= \nu\de_\bp f+\kappa\de_\bx f+ \kappa\na_\bx\cdot\lf(f\na_\bx{\uu}\rg)+\mathcal{Q}_{\uu,\bp}[f],\\
& \mathcal{Q}_{\uu,\bp}[f](t,\bx,\bp)=\iint\bigl( {\bq}\cdot \na_\by {\uu} - \bp\cdot\na_\bx {\uu}\bigr)\;f(t,\by,\bq)f(t,\bx,\bp)\,  \dy \dS_{\bq} ,\\
& f(t=0,\bx,\bp)=f_{0}(\bx,\bp),\quad (\bx,\bp)\in \Torus^d\times \mathbb{S}^{d-1}.   \end{cases}
\end{align}
For this model, one can consider the corresponding hydrodynamic density $[f]=\int _{\mathbb{S}^{d-1}}f \dS_\bp$ and apply ideas from this paper to derive the sampling and approximation estimates. 
\end{remark}

Similar to the first-order case, we consider the following second-order \emph{simplified enhanced sampling} dynamics:

\begin{align}
\label{EQ:2nd_Ord_sim}&\begin{cases}
&\displaystyle\pa_t \widehat f+{\bv}\cdot{\na_\bx} \widehat f
= \nu\de_\bp  \widehat f+\kappa\de_\bx \widehat f+ \kappa\na_\bx\cdot\lf( \widehat f\na_\bx{\uu}\rg)
 - \widehat f \bv\cdot\na_\bx {\uu},\\
&\quad  \widehat f(t=0,\bx,\bp)= f_{0}(\bx,\bp),\quad (\bx,\bp)\in \Torus^d\times \mathbb{T}^{d}.   \end{cases}
\end{align} The solutions of \eqref{EQ:2nd_Ord} and  \eqref{EQ:2nd_Ord_sim} are related by the formula: $ f= \frac{\widehat f}{\|\widehat f\|_{L^1}}. $ Hence, one can use a similar scheme as in the first-order enhanced model (simplified dynamics \eqref{EQ:2nd_Ord_sim} + ``mass-searching dynamics'') to simulate the PDE dynamics \eqref{EQ:2nd_Ord}. 
\subsection{Analysis of the Models} 
To motivate the mathematical tools applied in our analysis, we would like to consider the simplest possible target function, 
\begin{align}
\Pi_0(\bx)\equiv \frac{1}{|\Torus^d|}=\text{constant},\quad \bx \in\Torus^d, \quad d\geq 2.
\end{align} 
{We want to address the following  key question.}

\smallskip
{\centering
\fbox{ \begin{varwidth}{\linewidth}{\bf Question ${\mathscr A}$:} Fix arbitrary initial data that is sufficiently smooth.  Is it possible to design a dynamical system whose solution converges to the constant state $\Pi_0$ within a short time?\end{varwidth}}
}

\smallskip
We propose two 
{solutions} to address the above problem, which lead, respectively, to the first- and second-order models we design.

\noindent{\bf 
{Approach} \#1: Introducing time-dependent flows. }Incorporate the time-dependent divergence-free vector fields $\vv(t,\bx)$, and keep track of the passive scalar solution, 
\begin{align}\label{PS_intro_1}
\pa_t \rho+\na_\bx \cdot(\vv\rho)=\nu\de \rho,\quad \na_\bx\cdot\vv\equiv 0,\quad \rho(t=0)=\rho_0.
\end{align} 
Define the average $\dss\overline{\rho}:=\frac{1}{|\Torus^d|}\int\rho \dx$. The standard $L^2$-energy estimate shows that the norm $\lf\|\rho-\overline{\rho}\rg\|_{L^2}(t)$ decays to zero at a rate of $\mathcal{O}(\nu)$. It is a classical problem in the fluid mechanics community to identify suitable fluid vector fields $\vv$ that improve the decay rate of the solution. These flows are commonly referred to as \emph{relaxation enhancing} flows. We refer interested readers to the classical paper \cite{ConstantinEtAl08} for a precise definition. This phenomenon of flow-induced acceleration of dissipation is also called the \emph{enhanced dissipation}. Later research also shows that the \emph{enhanced dissipation} has a direct connection to the \emph{mixing} phenomenon (see, e.g., \cite{ElgindiCotiZelatiDelgadino18,FengIyer19} and the expository article \cite{CotiZelatiCrippaIyerMazzucato23}). Specifically, smooth, exponentially mixing vector fields $\vv$ yield a solution $\rho$ of \eqref{PS_intro_1} with rapid decay 
\begin{align}
\|\rho(t)-\overline{\rho}\|_{L^2}\leq C\|\rho_0-\overline{\rho_0}\|_{L^2}e^{-\delta|\log\nu|^{-2}t}.\label{Exp_ED}
\end{align} We observe that if $0<\nu\ll1$, the decay rate here is much faster than $\nu$. 
These exponentially mixing vector fields (such as the Pierrehumbert flow \cite{Pierrehumbert94} and the stochastic Navier-Stokes solutions considered in  \cite{BedrossianBlumenthalPunshonSmith19}) have recently been identified using probabilistic methods (see, e.g., \cite{BlumenthalCotiZelatiGvalani23,ElgindiLissMattingly23
}).  
 One can design fast sampling algorithms by incorporating these random, exponentially mixing flows into the sampling dynamics. This is essentially the idea presented in the pioneering paper \cite{christie2025speeding}.

However, it has also been pointed out in \cite{christie2025speeding} that applying exponentially mixing flows can lead to numerical challenges in sampling because the exponentially mixing flows constructed in the literature are random and vary drastically over time. For example, to generate the exponential decay in \eqref{Exp_ED}, the random vector fields $\vv$ will need to switch orientation over time intervals of length $\mathcal{O}(1)$. When $\nu\ll1$, these intervals are much shorter than the classical decay time scale of the heat equation, which is $\mathcal{O}(\nu^{-1})$. 

Our proposal to resolve the problem above is to consider a slowly varying, \emph{deterministic} alternating shear flow $\vv$. Since the construction is straightforward for even dimensions, we assume that $d=2\dd$ for some $\dd\in\mathbb{N}\backslash\{0\}$. The building blocks for this alternating shear are the following 
\begin{align}\label{u_ell}
\mathbf{u}_{1}(\{x_{\dd+j}\}_{j=1}^{\dd})= \sum_{j=1}^{\dd}\mathcal{U}(x_{\dd+j}){\bf e}_j ,\quad\mathbf{u}_{2}(\{x_j\}_{j=1}^{\dd})= \sum_{j=1}^{\dd}\mathcal{U}(x_j){\bf e}_{\dd+j} .
\end{align} Here, the set $\{\vec{e}_j\}_{j=1}^d$ is the collection of the canonical basis vectors of $\rr^d$ and $\mathcal{U}\in C^\infty$. When we consider the passive scalar equation subject to one of these shear flows, 
\begin{align}
\pa_t \rho+\mathbf{u}_1\cdot\na_{\bx}\rho =\nu\de \rho,\quad 0<\nu\leq1. 
\end{align}
There will be two distinct components of the solution.  The first is the stream-line average in the flow direction $\dss\lan \rho\ran_{1}:=|\Torus|^{-\dd}\int \rho(t,\bx)\mr dx_1\cdots\mr dx_\dd$ and the other is the fluctuation $\rho-\lan \rho\ran_{1}$. In two dimensions $(d=2)$, it is well-established that the fluctuation experiences \emph{enhanced dissipation} under suitable structural assumptions on the shear profile $\mathcal{U}$. The study of this shear flow-induced enhanced dissipation (and related hypoelliptic smoothing) can be traced back to the nineteenth-century work of Lord Kelvin \cite{Kelvin87}, 
and to A. Kolmogorov \cite{Kolmogoroff34} and L. H\"ormander \cite{Hormander67}.   In recent years, the sharp enhanced dissipation estimate has been obtained for general shear flows. Assuming that $d=2 $, it is shown in the works \cite{BCZ15,Wei18} that if the shear flow profile $\mathcal{U}$ has only finitely many non-degenerate critical points, the enhanced dissipation holds. In particular, if $\mathcal{U}=\sin$, the following decay holds for all $0<\nu\leq 1$,
\begin{align}\label{ED_sine}
\|\rho(t)-\lan{\rho}\ran_{\ell}(t)\|_{L^2}\leq C_{\sin}\|\rho_0-\lan{\rho_0}\ran_{\ell}\|_{L^2}e^{-\delta_{\sin}\nu^{1/2}t}, \quad \mathcal{U}(\cdot)=\sin(\cdot),\quad \forall t\geq 0.
\end{align}
Here, the constants $ C_{\sin}\geq 1,\,  \delta_{\sin}\in(0,1)$ are independent of the diffusion coefficient $\nu$. The $\nu^{1/2}$-rate has been proved to be sharp in \cite{CotiZelatiDrivas19}. We refer interested readers to the related works \cite{AlbrittonBeekieNovack21,
CobleHe23,
GardnerLissMattingly24,FengMazzucatoNobili23,
CotiZelatiGallay21}.

Finally, by alternately applying the shear flows $\mathbf{u}_1$ and $\mathbf{u}_2$ \eqref{u_ell} (with profile $\mathcal{U}=\sin$) on time intervals of length $\mathcal{O}(\nu^{-1/2})$, one establishes enhanced dissipation on $\mathbb{T}^2$:
\begin{align}\label{ED_intro_22}
\|\rho(t)-\overline{\rho}\|_{L^2}\leq C\|\rho_0-\overline{\rho_0}\|_{L^2}e^{-\delta\nu^{1/2}t}.
\end{align} This result was first derived in \cite{HeKiselev21}. In this paper, we derive the higher-dimensional analog. 

\begin{theorem}[Enhanced Dissipation]\label{thm:lnrED_AltS}
Consider the passive scalar equation in even dimension,
\begin{align}\label{eq:PS}
\pa_t \eta+\vv(t)\cdot \na_\bx \eta=\nu\de_\bx \eta,\quad \eta(t=0)=\eta_0\in L^2(\Torus^d),\quad d=2\dd. 
\end{align}
Let the time-dependent flow $\vv$ be the linear combination of $\eqref{u_ell}_{\mathcal{U}=\sin}$, 
$
\vv(t,\bx)=\sum_{i=1}^2 \Phi^{(i)}(t)\mathbf{u}_i(\bx).$ 
There exists a family of coefficient functions $\{\Phi^{(i)}\}_{i=1}^2$ that vary on a time scale $\mathcal{O}(\nu^{-1/2})$ such that the following enhanced dissipation estimate holds for any $0<\nu\leq 1$:  
\begin{align}\label{ED_V}\ \
\lf\|\eta(s+t,\cdot)-\overline{\eta}(s+t)\rg\|_{L_\bx^2(\Torus^d)}\leq C_{\rm ED}\lf\|\eta(s,\cdot)-\overline{\eta}(s)\rg\|_{L_\bx^2(\Torus^d)}e^{-\delta_{\rm ED} \nu^{1/2}t},\quad\forall s,t\in\rr_+. 
\end{align}
Here, $\delta_{\rm ED}\in(0,1)$ and $C_{\rm ED}\geq 1$ are constants that are independent of the solution and ${d},\nu$.  
\end{theorem}
\begin{proof}
The proof of Theorem \ref{thm:lnrED_AltS} is detailed in Section \ref{sec:ED_alt_sf}.
\end{proof}

\begin{remark}
The schematic picture of the coefficient functions $\{\Phi^{(i)}\}_{i=1}^2$ is shown in Figure \ref{fig:Cutoff} and the construction is detailed in Section \ref{sec:ED_alt_sf}.
\begin{figure}[h]
\includegraphics[scale=1]{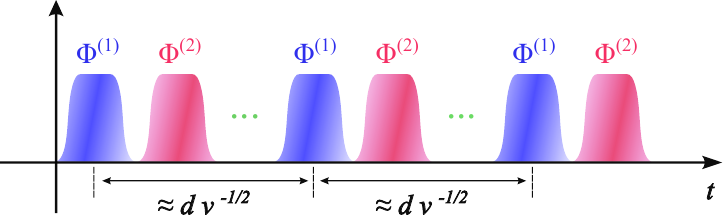} 
\caption{The coefficient functions $\Phi^{(i)}$}
\label{fig:Cutoff}
\end{figure}
\end{remark}
\begin{remark}
The parameter $\delta$ in Theorem \ref{thm_1st_order} is closely connected to the parameters in Theorem \ref{thm:lnrED_AltS}:
\begin{align}
 \delta:=\frac{\delta_{\rm ED}}{2+\log C_{\rm ED}}<\delta_{\rm ED}.\label{chc_del}
 \end{align}
\end{remark}

\begin{remark}
We remark that our construction is robust in the sense that replacing the shear profile $\mathcal{U}$ by another function with finitely many nondegenerate critical points leaves the result unchanged. 
\end{remark}
\begin{remark}
One can extend the result to odd dimensions $d=2\dd+1$ by a small modification of the flow. We will detail the construction in Section \ref{sec:ED_alt_sf}. 
\end{remark}
In conclusion, we propose replacing the random exponentially mixing flow in the previous literature by slowly varying deterministic shear flows. Even though the enhanced dissipation rate \eqref{ED_intro_22} is slower than the rate in \eqref{Exp_ED}, it has the advantage that the flow fields do not change too drastically and hence have the potential to make the numerics simpler.

\noindent{\bf 
{Approach} \#2: Extend the dimension of the phase space and use simpler flows.} 

The second approach is to introduce the momentum variable $\bp\in \mathbb{T}^{d}$, and to consider dynamics
\begin{align}\label{PS_bp} 
\pa_t f+\bv(\bp)\cdot\na_\bx f=\nu\de_\bp f,\quad f(t=0)=f_0,\quad (\bx,\bp)\in \mathbb{T}_\bx^d\times \mathbb{T}_\bp^{d},\quad d=1,2,3,\cdots. 
\end{align} 
Here, the transport vector field is $\bv(\bp):=\sum_{j=1}^d \sin(p_j){\bf e}_{\bx,j}$, and $\{{\bf e}_{\bx,j}\}_{j=1}^d$ are the canonical basis vectors for $\mathbb{T}_\bx^d$. 
Consider the average (in the shearing direction) $\lan f\ran_\bx$ and the remainder $\wt f$, 
 \begin{align}\label{P_avg_tild}
& \lan f \ran_\bx(t,\bp):=\frac{1}{|\Torus|^d}\int_{\Torus^d} f(t,\bx,\bp)\dx,\quad \wt f(t,\bx,\bp)=\wt{\mathbb{P}}_\bx f(t,\bx,\bp):=f(t,\bx,\bp)-\lan f\ran_\bx(t,\bp).
 \end{align}
Similar to the estimate \eqref{ED_sine}, one can derive the estimate (Corollary \ref{cor:ED_full}, Section \ref{sec:ED_alt_sf})
\begin{align}\label{ED_bp_intro}
\|\wt{f}(t)\|_{L^2_{\bx,\bp}}\leq C\|\wt{f}_0\|_{L_{\bx,\bp}^2}e^{-{\delta}\nu^{1/2}t},\quad \forall t\geq0.
\end{align}
A direct consequence of the estimate \eqref{ED_bp_intro} is that the hydrodynamic density \eqref{hydrodnsty} converges to the equilibrium with enhanced dissipation rate $\nu^{1/2}$: 
$
\|[f](t)-\overline{[f]}\|_{L_\bx^2}\leq C\|\wt{f}_0\|_{L_{\bx,\bp}^2}e^{-\delta\nu^{1/2}t},\quad \forall t\geq0.
$

In conclusion, if the target density $\Pi_0\equiv \frac{1}{|\Torus|^d}$, both the first-order dynamics \eqref{PS_intro_1} and the second-order dynamics \eqref{PS_bp} provide fast-converging solutions towards $\Pi_0$. This settles question ${\mathscr A}$.  
{Next, we turn} to address the following key {question}.  

\smallskip
{\centering
\fbox{ \begin{varwidth}{\linewidth}{\bf Question ${\mathscr B}$:} Fix arbitrary initial data that is sufficiently smooth.  Is it possible to design a dynamical system whose solution converges to any given target density $\Pi(\bx)=\zz^{-1}e^{-\uu(\bx)}\in L^1_+\cap C^2(\Torus^d)$ within a short time?\end{varwidth}}
}

\smallskip
As it turns out, we have already introduced all the essential ingredients to resolve question ${\mathscr B}$. The missing piece to address is how to introduce nontrivial density information into the dynamics \eqref{PS_intro_1} and \eqref{PS_bp}. 

\noindent
{\bf a) First-order model:} First, we rewrite the dynamics \eqref{PS_intro_1} as follows:
\begin{align} 
\pa_t \rho-\nu\na_\bx\cdot\lf(\rho\na_\bx \log\lf(\frac{\rho}{\Pi_0}\rg)\rg)=-\Pi_0\na_\bx \cdot\lf(\frac{\rho\vv}{\Pi_0}\rg),\quad \na_\bx\cdot\vv\equiv 0,\quad \rho(t=0)=\rho_0.
\end{align}
Naively, we could change the target density to $\Pi(\bx)=e^{-\uu(\bx)}/\zz$. However, a direct replacement would break the mass-conservation structure in general. Motivated by the designs in \cite{LuLuNolen19,TanLu23}, we incorporate a ``mass correction'' term in the dynamics to ensure conservation of the total mass: 
\begin{align}\label{EQ:1st_Ord2}
\begin{cases}&\dss\pa_t \rho-\nu \na_\bx\cdot \lf( \rho \na_\bx\log\lf(\frac{\rho}{\Pi}\rg)\rg)= -\Pi\ \na_\bx\cdot\lf (\frac{\rho\vv}{\Pi}\rg)- \rho \int \rho  {\vv}\cdot \na_\by \log\Pi\; \dy,\\
&\dss\na_\bx\cdot\vv(t,\bx)=0,\quad\rho(t=0,\bx)=\rho_{0}(\bx),\quad (t,\bx)\in \rr_+\times \Torus^d. \end{cases}
\end{align}
Recalling the definition of drift defect  \eqref{drift_defect} and applying the divergence-free condition of $\vv$, we  observe that the modified transport term naturally leads to the drift defect:
\begin{align}
-\Pi\ \na_\bx\cdot\lf (\frac{\rho\vv}{\Pi}\rg)
&=-\na_\bx\cdot(\rho \vv)+{\rho}e^{\uu}\na_\bx\cdot(e^{-\uu}\vv)=-\na_\bx\cdot(\rho \vv)-{\rho}\D.
\end{align}
Hence, the equation \eqref{EQ:1st_Ord2}, when rewritten properly, recovers the $\D$-alignment $\mathcal Q_{\uu,\vv}[\rho]$, yielding the enhanced model \eqref{EQ:1st_Ord}. 

To conclude the discussion of the first-order model, we highlight that the equation \eqref{EQ:1st_Ord} is equivalent to the following system satisfied by the quotient $h\overset{\text{def}}{=}\rho/\Pi$: 
\begin{align}\begin{cases}
&\displaystyle\pa_t h+\vv\cdot\na_\bx h=\nu\de_\bx h-\nu\na_\bx{\uu}\cdot\na_\bx h-h\int h\vv\cdot \na_\bx \Pi \dx,\\
&\displaystyle\hspace{2.3cm}\na_\bx\cdot\vv(t,\bx)=0,\quad h(t=0,\bx)=\frac{\rho_{0}(\bx)}{\Pi(\bx)}.\end{cases}\label{EQ:1stord_h}
\end{align}

\noindent
{\bf b) Second-order model} 
A similar argument yields that the passive scalar dynamics \eqref{PS_bp} can be generalized to the following equation, which is equivalent to the equation \eqref{EQ:2nd_Ord}:
\begin{align}
\begin{cases}&\dss\pa_t f+\Pi\ \na_\bx \cdot\lf(\frac{\bv(\bp) f}{\Pi}\rg)=\nu\de_\bp f+ \kappa\na_\bx\cdot \lf(f\na\log\lf(\frac{f}{\Pi}\rg)\rg)- {f}\iint f\, \bv(\bp)\cdot \na_\bx \log\Pi\; \dx \textnormal{d}\bp\\
&\dss f(t=0,\bx,\bp)= f_{0} (\bx,\bp),   \quad  \bx\in \Torus^d,\, \bp\in \mathbb{T}^{d}. \end{cases}
\end{align} 
Now we can rewrite it with $h\overset{\text{def}}{=}f/\Pi$ as follows, 
\begin{align}
 \begin{cases}&\displaystyle\pa_t h+\bv(\bp)\cdot\na_\bx h=\nu\de_{\bp}
h+\kappa\de_\bx h-\kappa\na_\bx{\uu}\cdot\na_\bx h-{h}\iint h\, \bv(\bp)\cdot \na_\bx \Pi\; \dx \textnormal{d}\bp,\\
&\displaystyle\hspace{2.3cm}h(t=0,\bx,\bp)=\frac{f_{0}(\bx,\bp)}{\Pi(\bx)},   \quad  \bx\in \Torus^d,\, \bp\in \mathbb{T}^{d}.\end{cases}\label{EQ:2ndord_h}
\end{align}

In Sections \ref{sec:First_Order}--\ref{sec:Second_Order}, we develop enhanced dissipation estimates, which are in the spirit of \eqref{ED_V} and \eqref{ED_bp_intro} for the formulations presented in this section.

\subsection{The Mass-searching Dynamics}\label{sec:numerics}
As discussed in the previous sections, for an accurate simulation of the full enhanced dynamics \eqref{EQ:1st_Ord} and \eqref{EQ:2nd_Ord}, a mass-searching procedure is required to recover the exact solutions. We propose a dynamics to compute or approximate the mass of a given function $e^{-\ww}$. This dynamics can be applied to find the normalizing constant $Z$ in $\Pi(\bx)=e^{-\uu}/Z$ \eqref{Trgt} or the mass of the solutions $\wh\rho(t,\bx),\,[\wh f](t,\bx)\geq 0$ to the simplified dynamics \eqref{EQ:1st_Ord_sim}\myh{ and \eqref{EQ:2nd_Ord_sim}}.

\myh{\subsubsection{{\bf  First-order Model}}} 
To compute the mass of a given function $F=e^{-\ww}\in C_\bx^3$,  we propose the following first-order dynamics to identify $M:=\|F\|_{L_\bx^1}$:
\begin{align}\label{EQ:1st_Ord_mass}
&\begin{cases}&\displaystyle\vphantom{\int}\pa_t \omega+ \vv \cdot\na_\bx \omega   -\omega\;\vv\cdot \na_\bx {\ww} =\nu e^{\ww}\de_\bx(\omega e^{-\ww})  ,\\
&\na_\bx\cdot\vv(t,\bx)=0,\quad\omega(t=0,\bx)\equiv {|\Torus|^d},\quad (t,\bx)\in \rr_+\times \Torus^d.   \end{cases}
\end{align}
Here, the vector field $\vv$ is defined in $\eqref{u_ell}_{\mathcal{U}=\sin}$ and the parameter $\nu$ is chosen as in Theorem \ref{thm:lnrED_AltS}. The diffusion operator is reformulated as follows:\begin{align}
\nu e^{\ww}\de_\bx(\omega e^{-\ww}) =  \nu\de_\bx\omega-2\nu \na_\bx \omega\cdot  \na_\bx{\ww}+ \nu\omega(-\de_\bx \ww+|\na_\bx \ww|^2),\quad \na_\bx \ww=-F^{-1}{\na_\bx F} .
\end{align} 
Hence, the gradient $\na\ww$ is singular if $F$ is close to zero. To avoid this, we assume that 
\begin{align}\label{F_low_bnd}
\min_{\bx} F(\bx)\geq 1. 
\end{align} For functions $\widetilde {F}$ that fail the condition \eqref{F_low_bnd}, we can define the shifted function $F:=\widetilde {F}+1$, apply the dynamics \eqref{EQ:1st_Ord_mass} to determine $\|F\|_{L^1}$, and then subtract the mass introduced by the shift to obtain the true mass, i.e., $\|\widetilde{F} \|_{L^1_\bx}=\|F\|_{L^1_\bx}-|\Torus|^d$. 

We emphasize that no quadrature is needed to simulate the dynamics \eqref{EQ:1st_Ord_mass}, and the solutions converge to the mass $\|e^{-\ww}\|_{L^1_\bx}$ at an enhanced rate. This is the content of the next theorem. 

\begin{theorem}\label{thm:mass_1}
Consider solutions to \eqref{EQ:1st_Ord_mass} subject to the vector field $\vv$ in $\eqref{u_ell}_{\mathcal{U}=\sin}$. Assume that the parameter $\nu$ is chosen as in Theorem \ref{thm:lnrED_AltS} and the target function $F$ satisfies the condition \eqref{F_low_bnd}. Then the following estimate holds
\begin{align}
\|\omega(t) e^{-\ww}-M\|_{L^2}\leq&C \||\Torus|^de^{-\ww}- M\|_{L^2}e^{-\delta_{\rm ED}\nu^{1/2}t},\quad \forall t\ge0.\label{mass_converge}
\end{align}
Here, the mass $M:=\|F\|_{L_\bx^1}$. 
\end{theorem} 
\myh{
\subsubsection{{\bf Second-order Model}} 
Consider the target function $F=e^{-\ww}\in C_\bx^3$ satisfying the condition \eqref{F_low_bnd}. We introduce a second-order PDE model that recovers $M$ as $t\rightarrow \infty$:
\begin{align}\label{EQ:2nd_Ord_mass}
\begin{cases}&\pa_t f+\underbrace{e^\ww \na_\bx\cdot(\bv f e^{-\ww})}_{=\bv\cdot \na_\bx f-\bv\cdot\na_\bx \ww f}=\nu\de_\bp f,\\
&\quad
f(t=0,\bx,\bp)=|\Torus|^d,\quad (t,\bx,\bp)\in \rr_+\times\Torus^d\times\Torus^d. 
\end{cases} %
\end{align}
The drift velocity has the form $\bv(\bp)=(\sin(p_1),\sin(p_2),\ldots, \sin(p_d))\in [-1,1]^{d}$. 
As a simple consequence of the enhanced dissipation \eqref{ED_bp_intro}, we have the following convergence result.

\begin{theorem}\label{thm:mass_2}
Consider solutions to \eqref{EQ:2nd_Ord_mass}. Assume that the parameter $\nu\in(0,1]$ and the target function $F$ satisfies the condition \eqref{F_low_bnd}. There exists a universal constant $\delta\in(0,1)$ such that the following estimate holds
\begin{align}
\|f(t) e^{-\ww}-M\|_{L_{\bx,\bp}^2}\leq&C \||\Torus|^de^{-\ww}- M\|_{L_{\bx,\bp}^2}e^{-\delta\nu^{1/2}t},\quad \forall t\ge0.\label{mass_converge_2}
\end{align}
Here, the mass $M:=\|F\|_{L_\bx^1}$. 
\end{theorem}
}

The remainder of the paper is organized as follows: In Section \ref{sec:First_Order}, we present the analysis of the first-order model \eqref{EQ:1st_Ord} and the proof of Theorem \ref{thm_1st_order}; in Section \ref{sec:Second_Order}, we present the proof of Theorem \ref{thm_2nd_order}, which characterizes the behavior of the second-order model \eqref{EQ:2nd_Ord}; in Section \ref{sec:Numerical_Result}, we present a numerical experiment to illustrate the fast convergence of our dynamics (see, e.g., Figure \ref{fig_error}). In the appendix, we present several technical theorems needed in the main text and the particle dynamics associated with the systems \eqref{EQ:1st_Ord} and \eqref{EQ:2nd_Ord}.

\noindent
{\bf Notation: }Throughout the paper, the symbol $C$ denotes a generic constant that may vary from line to line. The notation $T_j$ ($j=1, 2, \dots$) refers to intermediate terms appearing within a proof of a theorem or lemma. The $T_j$'s are local to the specific argument and can be redefined when we start the proof of a new theorem or lemma.

\section{The First-Order Enhanced Dynamics}
\label{sec:First_Order}

\subsection{Preliminaries}

In this section, we present some preliminary facts about the system \eqref{EQ:1stord_h}. Since $\rho_{0}$ and $\Pi$ are regular, a standard argument yields that the equation \eqref{EQ:1stord_h} is locally well-posed. Hence, we only derive a priori estimates in the forthcoming text.

An important aspect of the proposed model is that it preserves the total mass of the density $\rho$. To see this, we compute the time derivative of the mass. Thanks to the maximum principle, the solution $\rho$ stays positive as long as it stays regular. Hence, the mass $\|\rho(t)\|_{L^1}$ of the solution is the integral $\dss\int \rho(t,\bx)\dx$. We recall \eqref{EQ:1st_Ord2} and compute the time evolution of the mass as follows:
\begin{align}
\frac{d}{dt}\int \rho \dx=&\int_{\Torus^d}\na_\bx\cdot\lf(-\vv  \rho+\nu\na_\bx \rho+\nu \rho \na_\bx{\uu}\rg) \dx+\int \rho\lf(-\vv\cdot \na_\bx {\uu} +\int \rho\, \vv\cdot \na_\bx {\uu} \dx\rg)\dx.
\end{align} 
Thanks to the divergence theorem, the first term is 0. For the second term, we rewrite it as follows:
\begin{align}
\frac{d}{dt}\int \rho \dx=\lf(\int \rho \vv\cdot \na_\bx{\uu} \dx\rg) \lf(-1+\int \rho \dx\rg).
 \end{align}
Since the first factor on the right-hand side is continuous in time, $\int \rho(t,\bx) \dx\equiv 1$ is the unique solution to the above ODE with the prescribed initial mass. As a result, we conclude that the total mass is preserved by the dynamics: 
$
\|\rho(t,\cdot)\|_{L^1_\bx(\Torus^d)}\equiv \|\rho_{0}\|_{L^1_\bx(\Torus^d)}=1,\, \forall t\in\rr_+. 
$

\subsection{Proof of Theorem \ref{thm_1st_order}} In this section, we consider the first-order model \eqref{EQ:1st_Ord} and derive the result \eqref{eq:relative-estimate}. The proof is organized in steps. 

\vspace{5mm}
\step{1: Key estimate. } 
To distill the key estimate to establish, we observe that the following relations for all $t\geq 0$ imply estimate \eqref{eq:relative-estimate}:
\begin{align}\label{Est1st_nmd}
\lf\|\wt{\lf(\frac{\rho}{\Pi}\rg)}(t)\rg\|_{L_{\bx}^2}\leq& C_1\lf\|\wt{\lf(\frac{\rho_{0}}{\Pi}\rg)}\rg\|_{L_\bx^2} e^{-\delta\nu^{1/2}t},\quad\text{(Key Estimate)}\\ \label{Est1st_0md}
\lf|\overline{\lf(\frac{\rho}{\Pi}\rg)}(t)-1\rg|\leq& C_2e^{-\delta\nu^{1/2}t} .
\end{align}Here, $\dss\overline{F}:=\frac{1}{|\Torus|^d}\int_{\Torus^d } F \dx$ is the spatial average and $\wt F:=F-\overline{F}$ is the fluctuation. The constants $C_1$ (defined in \eqref{C_1}), $C_2$ (defined in \eqref{C2}) depend only on the initial density $\rho_{0}$, and the potential $\uu$.  We recall that the parameter $\delta$ \eqref{chc_del} can be expressed in terms of  $C_{\rm ED},\;\delta_{\rm ED}$ in \eqref{ED_V}. To see the implication of the result \eqref{eq:relative-estimate}, we apply the triangle inequality,
\begin{align*}
\lf\|\frac{\rho}{\Pi}(t)-1\rg\|_{L_{\bx}^2}&\leq \lf(\lf\|\wt{\lf(\frac{\rho}{\Pi}\rg)}(t)\rg\|_{L_\bx^2}+ |\Torus|^{d/2}\Big|\overline{\lf(\frac{\rho}{\Pi}\rg)}(t)-1\Big|\rg)\leq \lf(C_1\lf\|\wt{\lf(\frac{\rho_0}{\Pi}\rg)}\rg\|_{L_\bx^2}+|\Torus|^{d/2}C_2\rg)e^{-\delta\nu^{1/2}t}.
\end{align*}
Defining $C_\dagger:=C_1\lf\|\wt{\lf(\frac{\rho_0}{\Pi}\rg)}\rg\|_{L_\bx^2}+|\Torus|^{d/2}C_2$ yields the result. 
 
Next, we observe that the estimate \eqref{Est1st_0md} is a simple consequence of \eqref{Est1st_nmd}. Thanks to the mass conservation, the following equality holds 
\begin{align}
&0=\int \rho \dx-\int\Pi \dx=\int \lf(\frac{\rho}{\Pi}-1\rg)\Pi \dx=\int \lf[\wt {\left(\frac{\rho}{\Pi}\rg)} +\overline{\lf(\frac{\rho}{\Pi}\rg)}-1\right]\Pi \dx
\;\Longrightarrow\;\overline{\lf(\frac{\rho}{\Pi}\rg)}-1=-\int  \wt {\lf(\frac{\rho}{\Pi}\rg)} \Pi \dx.
\end{align}
Hence, 
\begin{align}\begin{split}&\lf|\overline{\lf(\frac{\rho}{\Pi}\rg)}(t)-1\rg|=\lf|\int\wt{\lf(\frac{\rho}{\Pi}\rg)}\Pi \dx\rg|\leq \|\Pi\|_{L^2}\lf\|\wt{\lf(\frac{\rho}{\Pi}\rg)}(t)\rg\|_{L^2}\\
&\leq C_1\|\Pi\|_{L^2}\lf\|\wt{\lf(\frac{\rho_{0}}{\Pi}\rg)}\rg\|_{L^2}e^{-\delta\nu^{1/2}t}=C_1\|e^{-\uu}\|_{L^2}\lf\|\wt{\lf({\rho_{0}}e^\uu\rg)}\rg\|_{L^2}e^{-\delta\nu^{1/2}t}=: C_2(C_1,\rho_0,\uu)e^{-\delta\nu^{1/2}t}.\end{split}\label{C2} 
\end{align} 
This is \eqref{Est1st_0md}. Hence, we conclude that the \emph{key estimate} is the nonlinear enhanced dissipation \eqref{Est1st_nmd}, whose proof stands as the focus of the remainder of the section.

\vspace{5mm}
\step{2: Bootstrap.} To develop the key estimate \eqref{Est1st_nmd}, we recall the quantity $h=\rho\Pi^{-1}$ and the equation \eqref{EQ:1stord_h} that it satisfies,
\begin{align}\label{EQ:1stord_h_1}
\begin{split}\pa_t h+\vv\cdot\na_\bx h=&\nu\de_\bx h-\nu\na_\bx{\uu}\cdot\na_\bx h-Qh,
\hspace{.5cm}Q(t):=\int_{\Torus^d} h \vv\cdot\na_\bx \Pi \dx,\quad h\geq 0.
\end{split}
\end{align}
The challenge in analyzing this system stems from the reaction term $Qh$. Unlike other terms on the right-hand side of \eqref{EQ:1stord_h_1}, it does not have a small parameter $\nu$ in front. Therefore, to control this term, we need to estimate the time integral of the quantity $Q$. This observation motivates our bootstrap setup. To prove the key estimate \eqref{Est1st_nmd}, we apply a bootstrap argument. Assume that $[0,T_\ast)$ is the maximal interval on $[0,\infty)$ such that the following two \emph{hypotheses} are satisfied:\begin{subequations}

\noindent a) Control on $Q(t)$: for all $ t\in[0,T_\ast)$ there holds
\begin{align} \label{Hypothesis_a}
\lf|\int_0^t Q(s)ds\rg|\leq 2\mathfrak{B}_1(\rho_0,\uu):=2\lf(1+ \lf|\log\lf(\frac{2  \|e^{-\uu}\|_{L_\bx^\infty}^{-1}}{3\|\rho_0 e^\uu\|_{L^1_{\bx}}}\rg)\rg|+\lf|\log\lf(\frac{  2\|e^{\uu}\|_{L_\bx^\infty}}{\|\rho_0 e^\uu\|_{L^1_{\bx}}}\right)\rg|\rg).
\end{align}

\noindent
b) Enhanced Dissipation
\begin{align}\label{Hypothesis_b}
\|\wt h(t)\|_{L^2}\leq 2\mathfrak{B}_2(\mf B_1)\|\wt h_{0}\|_{L^2}\exp\lf\{-{\delta}\nu^{1/2}t\rg\},\quad\mf B_2= e^{2\mf B_1+2}. 
\end{align}
Here, the parameter $\delta$ is defined in \eqref{chc_del}. \end{subequations}

To establish the estimate \eqref{Est1st_nmd}, it is sufficient to prove the following stronger \emph{conclusion}: If the threshold $0<\nu_0(\rho_0,\uu,\delta^{-1})$ is small enough, then for all $0<\nu\leq \nu_0$, the following estimates hold\begin{subequations}

\noindent  a) Control on $Q(t)$
 \begin{align}\label{Con_a}
\lf|\int_0^t Q(s)ds\rg|\leq \mathfrak{B}_1(\rho_0,\uu), \quad t\in [0,T_\ast);
\end{align}

\noindent
b) Enhanced Dissipation
\begin{align}\label{Con_b}
\|\wt h(t)\|_{L^2}\leq \mathfrak{B}_2(\rho_0,\uu)\|\wt h_{0}\|_{L^2}\exp\lf\{-{\delta}\nu^{1/2}t\rg\}, \quad t\in [0,T_\ast). 
\end{align}
\end{subequations}
\begin{remark}[The choice of $\nu_0$]
The $\nu_0$ satisfies all constraints  
 \eqref{chc_nu_0_1st}, \eqref{chc_nu01st2} and \eqref{chc_nu01st3}.
\end{remark}
Once the estimates \eqref{Con_a} and \eqref{Con_b} are established, we can see that $T_\ast=\infty$ and the following estimate holds
\begin{align}
&\lf\|\wt{\lf(\frac{\rho}{\Pi}\rg)}(t)\rg\|_{L_\bx^2}=\|\wt{h}(t)\|_{L_\bx^2}\leq e^{2\mf B_1+2}\|\wt{h_0}\|_{L_\bx^2}e^{-\delta\nu^{1/2}t}, \quad \forall t\geq 0.
\end{align}
This is \eqref{Est1st_nmd} with\begin{align}
&C_1(\rho_0,\uu):=\exp\lf\{ 4+ 2\lf|\log\lf(\frac{2  \|e^{-\uu}\|_{L_\bx^\infty}^{-1}}{3\|\rho_0 e^\uu\|_{L^1_{\bx}}}\rg)\rg|+2\lf|\log\lf(\frac{  2\|e^{\uu}\|_{L_\bx^\infty}}{\|\rho_0 e^\uu\|_{L^1_{\bx}}}\right)\rg|\rg\} .\label{C_1}
\end{align}  This will conclude the proof of Theorem \ref{thm_1st_order}. 

\step{3: Proof of conclusion \eqref{Con_a}.} 
To capture the quantity $Q$, we observe that it varies in time and is constant in space. Hence, it mainly causes fluctuations in the total mass of the solution $h$. Therefore, one can apply an a priori estimate of the $L^1$-norm of $h$ to recover $Q$. We compute, 
$\displaystyle \frac{d}{dt}\| h\|_{L_{\bx}^1}=\nu\int_{\Torus^d} \wt h\de_\bx {\uu} \dx-\|h\|_{L_{\bx}^1} Q$, which implies
\begin{equation}\label{h_L1_T12}
\begin{split}
\| h(t)\|_{L_{\bx}^1}=&\|h_0\|_{L_{\bx}^1} \exp\lf\{-\int_0^t Q(s)ds\rg\}+\nu\int_0^t \exp\lf\{-\int_s^t Q(\tau)d\tau\rg\}\int_{\Torus^d}\wt h(s,\bx)\de_\bx {\uu}(\bx)\dx ds\\
=:&T_1+T_2,\quad \forall t\in[0,T_\ast). 
\end{split}
\end{equation}
We first observe that thanks to the bootstrap assumptions \eqref{Hypothesis_a} and \eqref{Hypothesis_b}, the second term $T_2$ can be controlled as follows:
\begin{align}
|T_2|\leq&\nu\int_0^t \exp\{4\mf B_1 \} \|\wt h(s)\|_{L_\bx^2}\|\de_\bx{\uu}\|_{L_\bx^2}ds 
\leq  \nu\int_0^t 2\mf B_2  \exp\{4\mf B_1 \} \|\wt h(0)\|_{L_\bx^2}e^{-\delta\nu^{1/2}s}\|\de_\bx\uu\|_{L_\bx^2}ds\\
\leq &\nu^{1/2}\frac{1}{\delta}2\mf B_2  \exp\{4\mf B_1 \}  {\zz}\lf\|\wt{\lf(\rho_0e^{\uu}\rg)}\rg\|_{L_\bx^2}\|\de_\bx \uu\|_{L_\bx^2}.
\end{align}
Here, $\zz$ is the normalizing factor. 
Hence, for any 
\begin{align}\label{chc_nu_0_1st}
0<\nu\leq\nu_0(\rho_0,{\uu},\delta^{-1})\leq\Bigl( \frac{\delta\|\rho_0e^{\uu}\|_{L^1_\bx}}{4(1+\lf\|\wt{(\rho_0e^{\uu})}\rg\|_{L_\bx^2})(1+\|\de_\bx \uu\|_{L_\bx^2})\mf B_2 \exp\{6\mf B_1\}}\Bigr)^2,
\end{align}
 the following relation holds
\begin{align}
|T_2|\leq \frac{\zz}{2}\|\rho_0e^{\uu}\|_{L_\bx^1}\exp\{-2\mf B_1\}=\frac{1}{2}\|h_0\|_{L_\bx^1}\exp\{-2\mf B_1\}\leq \frac{1}{2}\|h_0\|_{L_\bx^1}\exp\lf\{-\int_0^t Q(s)ds\rg\}. 
\end{align} 
As a consequence of this estimate, we obtain the following 
\begin{align}
\|h(t)\|_{L_\bx^1}\in \lf[\frac{1}{2}\|h_0\|_{L_\bx^1}e^{-\int_0^t Q(s)ds},\frac{3}{2}\|h_0\|_{L_\bx^1}e^{-\int_0^t Q(s)ds}\rg].\label{h_L1_est}
\end{align}
On the other hand, we also have the following observation
\begin{align}
\|h\|_{L_{\bx}^1}=&\|\rho/\Pi\|_{L_{\bx}^1}\leq \|\rho\|_{L_{\bx}^1}\|\Pi^{-1}\|_{L_\bx^\infty} \ \ \text{and} \ \ 
\|\rho\|_{L_{\bx}^1}=\|h\Pi\|_{L_{\bx}^1}\leq \|h\|_{L_{\bx}^1}\|\Pi\|_{L_\bx^\infty}.
\end{align}
Hence,
$
\|\Pi\|_{L_\bx^\infty}^{-1}\leq \|h(t)\|_{L_{\bx}^1}\leq \|\Pi^{-1}\|_{L_\bx^\infty},\, \forall t\in[0,\infty). 
$ This further yields the bound
\begin{align}
\frac{2}{3}\|\Pi\|_{L_\bx^\infty}^{-1}\leq \|h_0\|_{L_\bx^1}\exp\lf\{-\int_0^t Q(s)ds\rg\}\leq  2\|\Pi^{-1}\|_{L_\bx^\infty}.
\end{align}
Rewriting this bound gives 
\begin{align}
  -\int_0^t Q(s)ds \in\lf[\log\lf(\frac{2  \|\Pi\|_{L_\bx^\infty}^{-1}}{3\|\rho_0/\Pi\|_{L^1_{\bx}}}\rg),\log\lf(\frac{  2\|\Pi^{-1}\|_{L_\bx^\infty}}{\|\rho_0/\Pi\|_{L^1_{\bx}}}\rg)\rg] .
\end{align}
Hence, we obtain the following estimate which concludes the proof of \eqref{Con_a},
\begin{align}
\lf|\int_0^t Q(s)ds\right|\leq \lf|\log\lf(\frac{2  \|e^{-\uu}\|_{L_\bx^\infty}^{-1}}{3\|\rho_0 e^\uu\|_{L^1_{\bx}}}\rg)\rg|+\lf|\log\lf(\frac{  2\|e^{\uu}\|_{L_\bx^\infty}}{\|\rho_0 e^\uu\|_{L^1_{\bx}}}\right)\rg|\le\mf B_1({\uu},\rho_0),\quad\forall t\in[0,T_\ast).
\end{align}

\step{4:} The proof of conclusion \eqref{Con_b} proceeds in two steps. 

\step{4a: Reformulation. }
We define the augmented quantity
$\displaystyle H:=h \exp\Big\{\int_0^t Q(s)ds\Big\}$. 
Thanks to hypothesis \eqref{Hypothesis_a}, $H$ and $h$ are comparable:
\begin{align}
H e^{-2\mf B_1}\leq H e^{-\int_0^t Q(s)ds}= h\leq He^{2\mf B_1}.\label{Comparable}
\end{align}
The equation \eqref{EQ:1stord_h_1} can be rephrased as follows
\begin{align}\label{EQ:1stord_h2}
\pa_t H+\vv\cdot\na_\bx H=\nu\de_\bx H-\nu\na_\bx{\uu}\cdot\na_\bx H,\quad H\geq 0.
\end{align} 
We can further decompose the solution,
\begin{align}\label{H_avg}\pa_t \overline{H}=&\frac{\nu}{|\Torus|^{d}}\int \wt{H} \de_\bx\uu \dx,\quad {\overline H(t=0)=\overline{\lf(\frac{\rho_0}{\Pi}\rg)}};\\
\label{H_nq}
\pa_t \wt H+  \vv\cdot\na_\bx \wt H=&\nu\de \wt H-\nu\wt {\mathbb{P}}_\bx{\lf (\na_\bx\wt \uu\cdot \na_\bx \wt H\rg)},\quad \wt H(t=0)=\wt{\lf(\frac{\rho_0}{\Pi}\rg)}.
\end{align}
Here, we recall the definition of $\wt{\mathbb{P}}_\bx$ \eqref{P_avg_tild}.

\step{4b: Enhanced Dissipation Estimate.}
We recall the linear enhanced dissipation \eqref{ED_V} and the definition of $\delta$ \eqref{chc_del} to obtain that 
\begin{align} \delta=\frac{\delta_{\rm ED}}{2+\log C_{\rm ED}}<\delta_{\rm ED}  \;\leadsto\;
C_{\rm ED}\exp\{-\delta_{\rm ED}\delta^{-1}\}= e^{-2}    .
\end{align}
Now  we consider a decomposition of the time horizon $[0,\infty)$ as 
\begin{align}
\rr_+=\bigcup_{i=0}^\infty [T_i,T_{i+1}),\quad T_i:=i\delta^{-1}\nu^{-1/2};\quad \bigcup_{i=0}^{i_\ast}[T_i,T_{i+1})\subset[0,T_\ast)\subset \bigcup_{i=0}^{i_\ast+1}[T_i,T_{i+1}).\label{Time_disc}
\end{align} Let $\wt\eta$ solve the reference equation
\begin{align}
\begin{cases}
&\pa_t \wt \eta^{(i)} +\vv\cdot\na_\bx \wt\eta^{(i)}=\nu\de_\bx \wt\eta^{(i)},\\
&\wt\eta^{(i)}(T_i,\bx)=\wt H(T_i,\bx),\quad \forall t\in[T_i,T_{i+1}].
\end{cases}
\end{align} 
\myh{\begin{figure}[H]\includegraphics[scale=1]{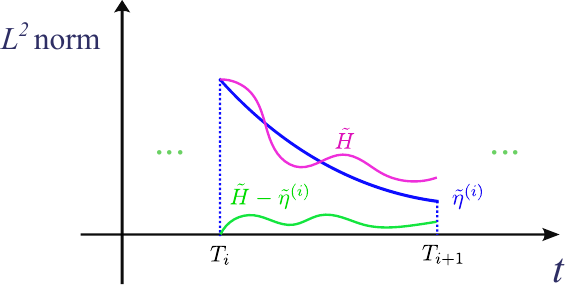} 
\caption{Relation between $\wt\eta^{(i)}$ and $\wt H$}
\end{figure}}
We compute the time evolution of the squared $L^2$-norm of the solution $\wt H$. Since $\int \overline H \wt g \dx\equiv 0$ for any function $g\in L^1$, we have \begin{align} 
\frac{1}{2}\frac{d}{dt}\|\wt H\|_{L^2}^2
=&-\nu \int| \na_\bx \wt H|^2 \dx+\int \na_\bx\cdot\vv\lf(\frac{\wt H^2}{2}\rg)\dx+ \nu\int  \de_\bx \wt\uu \lf(\frac{\wt H^2}{2}\rg)  \dx\\
=&-\nu\|\na_\bx \wt H\|_{L_\bx^2}^2+\frac{\nu}{2}\int (\de_\bx \uu) \wt H ^2\dx
=:-\mathfrak{D}+\mathfrak{S}. \label{DS1st}
\end{align}
Here, we have the diffusion contribution $\mf D$ and the source term $\mf S$. Note that the term $\mf S$ can be bounded as follows
\begin{align}
\mf S\leq \frac{\nu}{2}\|(\de_\bx \uu)_+\|_{L^\infty_\bx}\|\wt H\|_{L^2_\bx}^2.
\end{align}
Hence, the squared $L^2$-norm of $\wt H$ satisfies the differential inequality,
\begin{align}
&\frac{d}{dt}\|\wt H\|_{L_\bx^2}^2\leq \nu \|(\de_\bx \uu)_+\|_{L^\infty_\bx}\|\wt H\|_{L_\bx^2}^2\;\leadsto \; \\
&\quad\|\wt H(T_i+\tau)\|_{L_\bx^2}^2\leq \|\wt H(T_i)\|_{L_\bx^2}^2\exp\{\nu \|(\de_\bx \uu)_+\|_{L^\infty_\bx}\tau\},\quad \forall \tau \in [0,\delta^{-1}\nu^{-1/2}].
\end{align}
If we choose $\nu_0$ such that
\begin{align}
\nu_0^{1/2}\delta^{-1}\|\de_\bx \uu\|_{L^\infty_\bx}\leq 2,\label{chc_nu01st2}
\end{align}
then for any $\nu\in(0,\nu_0]$, we obtain that  
\begin{align}
\|\wt H(T_i+\tau)\|_{L_\bx^2}\leq e\|\wt H(T_i)\|_{L_\bx^2},\quad \forall \tau \in [0,\delta^{-1}\nu^{-1/2}]. 
\label{Reg_est_1st}
\end{align}
Next, we derive the deviation estimate. Consider the time evolution of the difference $\wt H-\wt\eta^{(i)}$. Standard $L^2$-energy estimate yields that
\begin{align} 
\label{Diff_1st} \begin{split}
\frac{1}{2}&\frac{d}{dt}\|\wt H-\wt \eta^{(i)}\|_{L^2}^2\\
&=-\nu \|\na (\wt H-\wt \eta^{(i)})\|_{L^2}^2+  \nu\int \na_\bx ( \wt H -\wt \eta^{(i)} ) \cdot \na_\bx\wt\uu \wt H  \dx+ \nu\int  ( \wt H -\wt \eta^{(i)} ) \de_\bx\wt\uu\wt H  \dx\\
&=:-\mf D_\ast+\mf S_{\ast;1}+\mf S_{\ast;2}. \end{split}
\end{align} 
To estimate the $\mf S_{\ast;1}$-term in \eqref{Diff_1st}, we apply the H\"older inequality, Young's inequality and the estimate \eqref{Reg_est_1st} to obtain the following 
\begin{align}
\mf S_{\ast;1}&\leq \nu\|\na_\bx (\wt H-\wt \eta^{(i)})\|_{L^2_\bx}\|\na_\bx \uu\|_{L^\infty_\bx}\|\wt H\|_{L_\bx^2}
\leq\frac{1}{4}\mf D_\ast+e^2\nu\|\na_\bx \uu\|_{L^\infty_\bx}^2\|\wt H(T_i)\|_{L_\bx^2}^2.
\end{align}
Similarly, we apply the Poincar\'e inequality $\|\wt H -\wt\eta^{(i)}\|_{L^2_\bx}\leq \|\na_\bx(\wt H -\wt\eta^{(i)})\|_{L^2_\bx}$ and a similar argument to obtain that 
\begin{align}
\mf S_{\ast;2}&\leq \nu\| \wt H-\wt \eta^{(i)}\|_{L^2_\bx}\|\de_\bx \uu\|_{L^\infty_\bx}\|\wt H\|_{L_\bx^2}
\leq\frac{1}{4}\mf D_\ast+e^2\nu\|\de_\bx \uu\|_{L^\infty_\bx}^2\|\wt H(T_i)\|_{L_\bx^2}^2.
\end{align}
Combining the estimates above yields that
\begin{align}
\frac{1}{2}\frac{d}{dt}\|\wt H-\wt \eta^{(i)}\|_{L^2}^2\leq -\frac{1}{2}\mf D_\ast+e^2\nu \lf(\|\na_\bx \uu\|_{L^\infty_\bx}^2+\|\de_\bx \uu\|_{L^\infty_\bx}^2\rg) \|\wt H(T_i) \|_{L^2}^2.
\end{align}
Integrating in time yields 
\begin{align}\|\wt H-\wt \eta^{(i)}\|_{L^2}(T_i+\tau)\leq& \frac{\sqrt{2}\,e\nu^{1/4}}{\delta^{1/2}} \lf(\|\na_\bx \uu\|_{L^\infty_\bx}^2+\|\de_\bx \uu\|_{L^\infty_\bx}^2\rg)^{1/2} \|\wt H(T_i) \|_{L^2}. 
\end{align}
As long as $\nu_0$ is chosen small enough, i.e.,
\begin{align}
0<\nu\leq\nu_0(\rho_0,\uu,\delta^{-1})\leq \lf(\frac{\delta}{8e^4(1+\|\na_\bx \uu\|_{L^\infty_\bx}^2+\|\de_\bx \uu\|_{L^\infty_\bx}^2 )}\rg)^2, \label{chc_nu01st3}
\end{align}
we have that
$
\|\wt H-\wt \eta^{(i)}\|_{L^2}(T_i+\tau)\leq  \frac{1}{2e}\|\wt H(T_i)\|_{L^2},\, \forall \tau\in[0,\delta^{-1}\nu^{-1/2}].$ 
As a consequence of the above discussion, if $\nu\leq\nu_0$ is chosen small such that all conditions 
\eqref{chc_nu_0_1st}, \eqref{chc_nu01st2} and \eqref{chc_nu01st3} are fulfilled, 
\begin{align} 
\|\wt H(T_i+\tau)\|_{L^2}^2\leq& {e^2}\|\wt H(T_i)\|_{L^2}^2,\quad
\|\wt H-\wt \eta^{(i)}\|_{L^2}(T_i+\tau)\leq \frac{1}{2e}\|\wt H(T_i)\|_{L^2},
\quad\forall \tau\in[0,\delta^{-1}\nu^{-1/2}].
\end{align}
Thanks to the choice of $\delta$ \eqref{chc_del}, we have that 
$
\|\wt \eta^{(i)}(T_i+\delta^{-1}\nu^{-1/2})\|_{L^2}\leq \frac{1}{2e}\|\wt H(T_i)\|_{L^2}.
$ 
Combining the estimates above, we have that 
\begin{align}
\|\wt H(T_{i+1})\|_{L^2}\leq 
\|\wt H(T_{i+1})-\wt \eta^{(i)}(T_{i+1})\|_{L^2}+\|\wt \eta^{(i)}(T_{i+1})\|_{L^2}\leq \frac{1}{e}\|\wt H(T_i)\|_{L^2}.\label{disc_ED_1st}
\end{align}
The remaining argument for deriving the enhanced dissipation estimate \eqref{Con_b} is standard. If $t\in \delta^{-1}\nu^{-1/2}\mathbb{N}$, then we have that by \eqref{disc_ED_1st},
\begin{align}
\|\wt H(t)\|_{L^2}\leq \|\wt H(0)\|_{L^2}\exp\lf\{- \frac{ t}{\delta^{-1}\nu^{-1/2}}\rg\}=\|\wt h_{0}\|_{L^2}\exp\lf\{- {\delta \nu^{ 1/2}}{ t}\rg\}.
\end{align} On the other hand, if $t\notin\delta^{-1}\nu^{-1/2}\mathbb{N}$, we choose the largest integer $N$ such that $\delta^{-1}\nu^{-1/2}N\leq t.$ Hence, we have the relation $t\in\delta^{-1}\nu^{-1/2}[N,N+1]$. Then, we combine the estimates \eqref{Reg_est_1st} and \eqref{disc_ED_1st} to obtain that
\begin{align}
\|\wt H(t)\|_{L^2}\leq& e\|\wt H (T_N)\|_{L^2}\leq e^2\|\wt h_{0}\|_{L^2}e^{-(N+1)}\leq e^2\|\wt h_{0}\|_{L^2}e^{ -\delta \nu^{1/2}t },\quad \forall  t \in[0,T_\ast).\label{stndrdEDarg}
\end{align} 
We recall the relation \eqref{Comparable},
\begin{align}
\|\wt h(t)\|_{L^2}\leq&e^{2\mf B_1}\|\wt H(t)\|_{L^2}\leq  e^{2\mf B_1+2}\|\wt h_{0}\|_{L^2}e^{ -\delta \nu^{1/2}t },\quad \forall  t \in[0,T_\ast).
\end{align} 
This proves the nonlinear enhanced dissipation  \eqref{Con_b} and concludes the proof of  Theorem \ref{thm_1st_order}.
\section{The Second-Order Enhanced Dynamics}
\label{sec:Second_Order}
In this section, we consider the second-order model \eqref{EQ:2nd_Ord} and prove Theorem \ref{thm_2nd_order}. Part of the analysis is similar to that for the first-order model, and we provide a sketch of the argument in Section \ref{sec:2nd_Prelim}. In Section \ref{sec:Proof_Thm2nd}, we provide the proof of Theorem \ref{thm_2nd_order}. 

\subsection{Preliminaries}\label{sec:2nd_Prelim}

We recall the model \eqref{EQ:2nd_Ord},
\begin{align}\begin{cases}&\displaystyle\pa_t f+\bv\cdot{\na_\bx} f
=\nu\de_\bp f+\kappa\de_\bx f+ \kappa\na_\bx\cdot\lf(f\na_\bx{\uu}\rg) +f\lf( - \bv\cdot\na_\bx {\uu}+\iint f\, \bv\cdot \na_\bx {\uu}\; \dx \textnormal{d}\bp \rg)\\
&\displaystyle\bv(\bp):=\sum_{j=1}^d\sin(p_j){\bf e}_{\bx,j},\quad \Pi=\zz^{-1}e^{-\uu},\quad f(t=0,\bx,\bp)= f_{0} (\bx,\bp),   \quad  \bx\in \Torus^d,\, \bp\in \mathbb{T}^d.\end{cases}\label{EQ:2nd_2d}
\end{align}Assume the following initial normalization:
\begin{align}\label{nrm_2nd}
\|f_0\|_{L_{\bx,\bp}^1}=1,\quad f_0\geq 0.
\end{align}
Standard arguments yield that the solution is locally well-posed, and we will focus on the a priori estimates in the forthcoming steps.  The model preserves mass. Thanks to the parabolic maximum principle, $f\geq 0$, and we can compute the time derivative of the mass:
\begin{align}\label{mss_prs}
&\ddt\|f\|_{L_{\bx,\bp}^1}=\iint f {\bv\cdot\na_\bx\log\Pi} \dx\textnormal{d}\bp-   \|f\|_{L_{\bx,\bp}^1}  \iint f \bv\cdot\na_\bx\log\Pi  \dx \textnormal{d}\bp,
\end{align}
Since $\|f_0\|_{L^1}=1$, the scalar ODE \eqref{mss_prs} implies \begin{align}\label{nrm_t_2nd}
\|f(t)\|_{L_{\bx,\bp}^1}\equiv 1=\|\Pi\|_{L^1_\bx}.
\end{align}
Now we consider the equation for 
$\displaystyle h(t,\bx,\bp):=\frac{f(t,\bx,\bp)}{\Pi(\bx)}$,
which solves the following equation:
\begin{align}
 \begin{cases}&\displaystyle\pa_t h+\bv\cdot\na_\bx h=\nu\de_{\bp}
h+\kappa\de_\bx h-\kappa\na_\bx{\uu}\cdot\na_\bx h-{h}\iint h\, \bv \cdot \na_\bx \Pi\; \dx \textnormal{d}\bp,\\
&\displaystyle\bv(\bp):=\sum_{j=1}^d\sin(p_j){\bf e}_{\bx,j},\quad \Pi=\zz^{-1}e^{-\uu}, \quad h(t=0,\bx,\bp)=\frac{f_{0}(\bx,\bp)}{\Pi(\bx)}.\end{cases}\label{EQ:2nd_2dh}
\end{align}
As in the first-order case, we define the renormalization factor
\begin{align}\label{G}
G(t):=\iint h(t,\bx,\bp)\, \bv(\bp)\cdot\na_\bx \Pi(\bx)\; \dx \textnormal{d}\bp, 
\end{align}
and define the renormalized solution
\begin{align}\label{Rnrm_sln}
H(t,\bx,\bp):= h(t,\bx,\bp)\exp\lf\{\int_0^t G(s)\ds\rg\},\quad \forall t\geq 0.
\end{align}
The equation satisfied by the renormalized solution is as follows:
\begin{align}
\begin{cases}&\displaystyle\pa_t H+\bv\cdot\na_\bx H=\nu\de_{\bp}
H+\kappa\de_\bx H-\kappa\na_\bx{\uu}\cdot\na_\bx H,\\
&\displaystyle H(t=0,\bx,\bp)=\frac{f_{0}(\bx,\bp)}{\Pi(\bx)}.\end{cases}\label{EQ:2nd_H}
\end{align}
To analyze this equation, we further decompose the solution $H$ into the average $\lan H\ran_\bx$ and the fluctuation $\wt H=H-\lan H \ran_\bx$:
\begin{align}\label{h_eq_avg}\pa_t \lan H\ran_\bx=&\nu\de_\bp \lan H\ran_\bx+\kappa\lan \wt H\de_\bx \uu \ran_\bx,\quad \lan H\ran_\bx \geq 0;\\
\label{h_eq_nq}
\pa_t \wt H=&\nu\de_\bp \wt H{+\kappa\de_\bx \wt H}-\kappa\wt{\mathbb{P}}_\bx(\na_\bx\uu\cdot \na_\bx \wt H)- \bv\cdot\na_\bx \wt H.
\end{align}\scalebox{0.0001}{\eqref{h_eq_avg}}
Here, we recall the definition of $\wt {\mathbb{P}}_\bx$ in \eqref{P_avg_tild}. 

Finally, we fix the parameter $\delta$ in Theorem \ref{thm_2nd_order} as follows
\begin{align}\label{chc_del_2nd}\delta = \frac{\delta_{0}}{5}\qquad \text{(Theorem \ref{thm_2nd_order})} .
\end{align}
 The constant $\delta_0$ is defined in Theorem \ref{thm:ED} and is independent of the diffusion coefficient $\nu$, the dimension $d$, and the solution. 
 
The proof of the following nonlinear enhanced dissipation estimates is analogous to that of \eqref{eq:relative-estimate}.
\begin{theorem}\label{thm:Nlnr_ED_2nd} Under the assumptions of Theorem \ref{thm_2nd_order}, the following claims hold: \begin{itemize}
\item There exists a threshold $\nu_0(f_0,\delta^{-1},\|\uu\|_{C_\bx^2})$ such that for all $0<\nu\leq \nu_0$, the nonlinear estimates hold:
\begin{align}
\label{G_est}
& \lf|\int_0^t   G(s)\ds\rg| \leq1+ \lf|\log\lf(\frac{2  \|e^{-\uu}\|_{L_\bx^\infty}^{-1}}{3\|f_0 e^\uu\|_{L^1_{\bx,\bp}}}\rg)\rg|+\lf|\log\lf(\frac{  2\|e^{\uu}\|_{L_\bx^\infty}}{\|f_0 e^\uu\|_{L^1_{\bx,\bp}}}\right)\rg|,\\ 
\label{ED_2nd} &\|\wt h(t)\|_{L_{\bx,\bp}^2}\leq  C_1^*\|\wt h_{0}\|_{L_{\bx,\bp}^2}e^{ -\delta \nu^{1/2}t },\\
&\hspace{0.4cm}C_1^*:=\exp\lf\{ 4+ 2\lf|\log\lf(\frac{2  \|e^{-\uu}\|_{L_\bx^\infty}^{-1}}{3\|f_0 e^\uu\|_{L^1_{\bx,\bp}}}\rg)\rg|+2\lf|\log\lf(\frac{  2\|e^{\uu}\|_{L_\bx^\infty}}{\|f_0 e^\uu\|_{L^1_{\bx,\bp}}}\right)\rg|\rg\},\quad \forall  t \in[0,\infty) .
\end{align}
\item Consider the passive scalar  solution $\wt\zeta$, 
\begin{align} \label{pssv}
&\pa_t \wt\zeta +\bv\cdot\na_\bx \wt\zeta=\nu\de_\bp \wt\zeta+\kappa\de_\bx \wt\zeta,\quad \wt \zeta(0,\bx,\bp)=\wt H(0,\bx,\bp).
\end{align} 
Then the following estimate holds for all $\nu\leq 1$,
\begin{align}
\|\wt H-\wt\zeta\|_{L_{\bx,\bp}^2}(t)\leq C\left(\delta^{-1}, \|\uu\|_{C^2_\bx}\right)\nu^{\frac14}\|\wt H(0)\|_{L_{\bx,\bp}^2},\quad\forall t\in[0,\delta^{-1}\nu^{-1/2}]. \label{pssv_app}
\end{align}
\end{itemize}
\end{theorem}

\begin{proof}  We apply a bootstrap argument and decompose the proof into three steps.   

\smallskip
\step{1: Setup.}
Assume that $[0,T_\ast)$ is the maximal interval on $[0,\infty)$ such that the following two hypotheses are satisfied for all $t\in[0,T_\ast)$:\begin{subequations}

\noindent a) Control on $G(t)$: 
\begin{equation}\label{Hyp_a-2} 
 \lf|\int_0^t   G(s)\ds\rg| \leq 2\lf(1+ \lf|\log\lf(\frac{2  \|e^{-\uu}\|_{L_\bx^\infty}^{-1}}{3\|f_0 e^\uu\|_{L^1_{\bx,\bp}}}\rg)\rg|+\lf|\log\lf(\frac{  2\|e^{\uu}\|_{L_\bx^\infty}}{\|f_0 e^\uu\|_{L^1_{\bx,\bp}}}\right)\rg|\rg)=:2\mathfrak{C}_1(f_0,\uu) .  
\end{equation}

\noindent
b) Enhanced Dissipation: 
\begin{align}\label{Hyp_b-2}
\|\wt h(t)\|_{L^2}\leq 2\mathfrak{C}_2\|\wt h_{0}\|_{L^2}e^{-{\delta}\nu^{1/2}t},\quad{\mathfrak C}_2:= e^{2{\mathfrak C}_1+2}. 
\end{align}
Here, the parameter $\delta$ is chosen in  \eqref{chc_del_2nd}. Moreover, the constants $\mathfrak{C}_1,\, \mathfrak{C}_2$ will only be used in this proof. 
 \end{subequations}\scalebox{0.001}{\eqref{Hyp_b-2}.}

The \emph{goal for the proof} is to show that on the same time interval, the estimates \eqref{Hyp_a-2} and \eqref{Hyp_b-2} can be improved provided the diffusion coefficient $0<\nu<\nu_0(f_0,\uu,\delta^{-1})$ is chosen small enough. The concrete improved  conclusions are stated below:\begin{subequations}
 \begin{align}\label{Con_a-2}
\lf|\int_0^t G(s)\ds\rg|\leq& \mathfrak{C}_1(f_0,\uu);\\ \label{Con_b-2}
\|\wt h(t)\|_{L^2}\leq& \mathfrak{C}_2(f_0,\uu)\|\wt h_{0}\|_{L^2}e^{-{\delta}\nu^{1/2}t}. 
\end{align}
\end{subequations}
The proof of \eqref{Con_a-2} will be presented in Step \# 2, and the proof of \eqref{Con_b-2} will be presented in Step \# 3. Finally, we develop \eqref{pssv_app} in Step \# 4. 

\smallskip
\step{2: Proof of \eqref{Con_a-2}.}
We can compute the time evolution of the function $h$ for all $t\in[0,\infty)$:
\begin{align}
&\ddt\| h\|_{L_{\bx,\bp}^1}= \kappa\iint \wt{h}\de_\bx \uu\,\dx\mathrm d\bp-\|h\|_{L_{\bx,\bp}^1} G(t)\;\leadsto\;\\
&\|h(t)\|_{L_{\bx,\bp}^1}=\|h_0\|_{L_{\bx,\bp}^1} e^{-\int_0^t G(s)\ds}    +\kappa\int_0^t\iint e^{-\int_s^t G(\tau)\textnormal{d}\tau} \wt{h}(s,\bx,\bp)\de_\bx \uu (\bx) \dx \mathrm{d} \bp\ds. \label{h_L1_1}
\end{align}
The structure of this equation mirrors that of \eqref{h_L1_T12}. In particular, the second term only involves the $\wt{h}$ remainder of the solution. Hence, thanks to the bootstrap hypothesis \eqref{Hyp_b-2} and the choice of $\kappa\leq \nu$, we have that the time integral of this term is small. The remaining argument of this step is similar to the proof of \eqref{Con_a}, and hence we omit it for the sake of brevity.

\smallskip
\step{3: Proof of \eqref{Con_b-2}.} The main difference in the argument comes from the analysis of the enhanced dissipation phenomenon. Here, the linear enhanced dissipation effect \eqref{Hypo_est_ndeg_full} is weaker because it only acts in the $\bx$-direction. 

First of all, we observe that the renormalized solution $H$ and the original solution $h$ are  comparable under the hypothesis \eqref{Hyp_a-2}:
\begin{align}
H(t,\bx,\bp) e^{-2{\mathfrak C}_1}\leq H(t,\bx,\bp) e^{-\int_0^t G(s)\ds}= h(t,\bx,\bp)\leq H(t,\bx,\bp)e^{2{\mathfrak C}_1},\quad\forall t\in[0,T_\ast). \label{Comparable_2nd}
\end{align}
Hence, a suitable estimate on $\wt H$ leads to an estimate of $\wt h$ as stated in \eqref{Con_b-2}.  Therefore, we will mainly focus on $\wt H$.

Now the \emph{goal} is to show that the $\wt H$ has enhanced dissipation of the form
\begin{align}
 \|\wt H(t)\|_{L^2}\leq e^2\|\wt{h}(0)\|_{L^2}e^{-\delta \nu^{1/2}t}.\label{nzm_con}
\end{align} Here, the $\delta$ is chosen as in \eqref{chc_del_2nd}. As a result of the comparison \eqref{Comparable_2nd}, the estimate \eqref{nzm_con} implies the conclusion  \eqref{Con_b-2}.  

To prove \eqref{nzm_con}, we implement the same time-discretization as in \eqref{Time_disc}. On each time interval $[T_i,T_{i+1}], \, T_i=i\delta^{-1}\nu^{-1/2}$, let $\wt\eta^{(i)}$ solve the reference equation
\begin{align}
\begin{cases}
&\pa_t \wt\eta^{(i)} +\bv\cdot\na_\bx \wt\eta^{(i)}=\nu\de_\bp \wt\eta^{(i)}+\kappa\de_\bx \wt\eta^{(i)},\\
&\wt\eta^{(i)}(T_i,\bx,\bp)=\wt H(T_i,\bx,\bp),\quad \forall t\in[T_i,T_{i+1}].
\end{cases}
\end{align}  
We carry out energy estimates similar to those used above, 
\begin{align} \label{DTR}
\frac{1}{2}\ddt\|\wt H\|_{L^2}^2=&-(\nu \|\na_\bp\wt H\|_{L^2}^2+\kappa\|\na_\bx\wt H\|_{L^2}^2)-\kappa\iint \wt H  \na_\bx \uu\cdot \na_\bx \wt H  \dx \mathrm{d}\bp
=: -\mathcal{D}+\mathcal{S}; \\ 
\label{DTD} \qquad \frac{1}{2}\ddt\|\wt H-\wt \eta^{(i)}\|_{L^2}^2=&- (\nu\|\na_\bp (\wt H-\wt \eta^{(i)})\|_{L^2}^2+\kappa\|\na_\bx(\wt H-\wt \eta^{(i)})\|_{L^2}^2)\\
&\quad-\kappa \iint  (\wt H-\wt\eta^{(i)})  \na_\bx\uu\cdot \na_\bx \wt H \dx \textnormal{d}\bp
=:-\mathcal{D}_\ast+\mathcal{S}_\ast. 
\end{align} 
Now we estimate the terms in \eqref{DTD}
\begin{align}\label{T}
|\mathcal{S}_\ast|\leq &\frac{1}{4}\kappa\|\na_\bx (\wt H-\wt\eta^{(i)})\|_{L^2}^2+C\kappa(\|\na_\bx\uu \|_{L^\infty}^2+\|\de_\bx\uu \|_{L^\infty}^2)\|\wt H\|_{L^2}^2\\
\leq&\frac{1}{4}\kappa\|\na_\bx (\wt H-\wt\eta^{(i)})\|_{L^2}^2+C(\|\uu\|_{C^2_\bx}){\nu} \|\wt H \|_{L^2}^2,
\end{align}
Hence,
\begin{align}
\ddt\|\wt H-\wt\eta^{(i)}\|_{L^2}^2\leq -\frac{1}{2}\bigl(\nu\|\na_\bp(\wt H-\wt \eta^{(i)})\|_{L^2}^2+\kappa\|\na_\bx(\wt H-\wt \eta^{(i)})\|_{L^2}^2\bigr)+C(\|\uu\|_{C^2}){\nu} \|\wt H \|_{L^2}^2.
\end{align}
Through similar estimates on the equation \eqref{DTR}, we have 
$
\ddt\|\wt H\|_{L^2}^2\leq C(\|\uu\|_{C^2 }) {\nu}\|\wt H\|_{L^2}^2 .
$ 
As a consequence of the above differential inequalities, by the Gr\"onwall inequality, for all $t\in[T_i,T_{i+1}=T_i+\delta^{-1}\nu^{-1/2}]$,
\begin{align}\begin{split}
&\|\wt H(t)\|_{L^2}^2\leq  e^{C(\|\uu\|_{C^2})\delta^{-1}\nu^{\frac{1}{2}}}\|\wt H(T_i)\|_{L^2}^2,\\
&\|\wt H-\wt \eta^{(i)}\|_{L^2}^2(t)\leq  \nu^{1/2} C(\|\uu\|_{C^2})\delta^{-1}e^{C(\|\uu\|_{C^2})\delta^{-1}\nu^{\frac{1}{2}}}\|\wt H(T_i)\|_{L^2}^2.\end{split}\label{dif_est_2nd}
\end{align}Hence, if $\nu_0=\nu_0(\delta^{-1}, \|\uu\|_{C^2})$ is chosen small enough, 
\begin{align}\label{chc_nu_1}
\begin{split}&\|\wt H(t)\|_{L^2}\leq e\|\wt H(T_i)\|_{L^2},\\
&\|\wt H-\wt \eta^{(i)}\|_{L^2}(t)\leq \frac{1}{4e}\|\wt H(T_i)\|_{L^2},\end{split}\quad\forall t\in[T_i,T_i+\delta^{-1}\nu^{-1/2}].
\end{align}
Thanks to the choice of $\delta=\frac{\delta_0}{5}$, we have that by the linear enhanced dissipation \eqref{Hypo_est_ndeg_full},
\begin{align*}
\|\wt\eta^{(i)}(T_i+\delta^{-1}\nu^{-1/2})\|_{L^2}\leq e^{\frac{1}{2}(1-\delta_0\nu^{\frac{1}{2}} 5\delta_0^{-1}\nu^{-\frac{1}{2}} )}\|\wt H(T_i)\|_{L^2}\leq \frac{1}{e^2}\|\wt H(T_i)\|_{L^2}.
\end{align*}
Combining the estimates above, we have that 
\begin{align}
\|\wt H(T_{i+1})\|_{L^2}\leq 
\|\wt H(T_{i+1})-\wt\eta^{(i)}(T_{i+1})\|_{L^2}+\|\wt\eta^{(i)}(T_{i+1})\|_{L^2}\leq \frac{1}{e}\|\wt H(T_i)\|_{L^2}.\label{disc_ED}
\end{align}
The remaining argument to derive the enhanced dissipation estimate \eqref{nzm_con} is identical to the argument in the first-order case \eqref{stndrdEDarg}. Hence, we omit the details. 

\smallskip
\step{4: Proof of \eqref{pssv_app}.} The estimate is a consequence of \eqref{dif_est_2nd}. We choose $\wt \zeta=\wt\eta^{(0)}$, and observe that $\nu\leq 1$ is sufficient to derive \eqref{dif_est_2nd}, which implies \eqref{pssv_app}. This concludes the proof.
\end{proof}

\subsection{Proof of Theorem \ref{thm_2nd_order}}\label{sec:Proof_Thm2nd}
We divide the proof into steps. 

\step{1: Convergence of the hydrodynamic density $[f]$.} 
In this step, we establish the convergence  (\ref{Sample_2nd}):  
$
\int_{\mathbb{T}^d} f(t,\bx,\bp)\textnormal{d}\bp\overset{t\rightarrow \infty}{\longrightarrow} \Pi(\bx). 
$ 
To this end, the key is to establish the following convergence
$
|\Torus|^{-d}\|h\|_{L_{\bx,\bp}^1}\overset{t\rightarrow \infty}{\longrightarrow} 1.
$ 
To see this convergence, we observe that by mass conservation \eqref{nrm_t_2nd} ($\|f(t)\|_{L_{\bx,\bp}^1}\equiv 1=\|\Pi\|_{L^1_\bx}$), we have
\begin{align}\label{ED_mass}
0&=\iint f \dx \textnormal{d}\bp-\int \Pi \dx = \iint\lf(h-\frac{1}{|\mathbb{T}|^d}\rg) \Pi \dx \textnormal{d}\bp=\iint\lf(\lan h\ran_\bx+\wt{h}-\frac{1}{|\mathbb{T}|^d}\rg) \Pi \dx \textnormal{d}\bp.
\end{align}
This, together with the enhanced dissipation \eqref{ED_2nd}, leads to the relation 
\begin{align}  \hspace{.23cm}
\begin{split}\Bigl||\Torus|^{-d}\|h&\|_{L^1_{\bx,\bp}}- 1 \Bigr|=\lf|\|\lan h\ran_\bx\|_{L^1_\bp}- 1\rg |=\lf|\int \lf(\lan h\ran_\bx-\frac{1}{|\mathbb T|^d}\rg) \textnormal{d}\bp \rg|=\lf|\iint \wt{h} \Pi \dx \textnormal{d}\bp \rg|\\
&\leq  |\mathbb T|^{d/2}\|\wt{h}\|_{L_{\bx,\bp}^2}\|\Pi\|_{L_\bx^2}\leq |\mathbb T|^{d/2}C_1^*(f_0,\uu)\|\wt{h}_0\|_{L_{\bx,\bp}^2}\|\Pi\|_{L_\bx^2}e^{-\delta\nu^{1/2}t}\phantom{\inf}\\
&=|\mathbb T|^{d/2}C_1^*(f_0,\uu)\|\wt{\mathbb{P}}_\bx(f_{0}e^\uu)\|_{L_{\bx,\bp}^2}\|e^{-\uu}\|_{L_\bx^2}e^{-\delta\nu^{1/2}t}.\phantom{\inf}\end{split}\label{h_bound_1}
\end{align} Here, we recall the definition of $\wt{\mathbb{P}}_\bx$ in \eqref{P_avg_tild}. 
Now we make the decomposition:
\begin{align*}
&\lf\|{\int f(t,\bx,\bp) \textnormal{d}\bp}-\Pi\rg\|_{L^1_{\bx}}=\lf\|{\int f(t,\bx,\bp) \textnormal{d}\bp}-\frac{\int h(t,\bx,\bp)\Pi \mathrm d\bp}{|\Torus|^{-d}\| h\|_{L_{\bx,\bp}^1} }+\frac{\int h(t,\bx,\bp)\Pi \mathrm d\bp}{|\Torus|^{-d}\| h\|_{L_{\bx,\bp}^1} }-\Pi\rg\|_{L^1_\bx} \\
&\leq \lf\|{\int f(t,\bx,\bp) \textnormal{d}\bp}-\frac{\int h(t,\bx,\bp)\Pi \mr d\bp}{|\Torus|^{-d}\| h\|_{L_{\bx,\bp}^1} }\rg\|_{L^1_{\bx}}+\lf\|\frac{\int h(t,\bx,\bp)\Pi \mathrm d\bp}{\int \lan h\ran \textnormal{d}\bp }-\Pi\rg\|_{L_\bx^1}
=:T_1+T_2.
 \end{align*}
Here we have used the fact that $h\geq 0$ and $\|\lan h\ran_\bx\|_{L^1_\bp}=|\Torus|^{-d}\|h\|_{L_{\bx,\bp}^1}$ (recall that $\lan h\ran_\bx:=|\Torus|^{-d}\int h(t,\bx,\bp)\dx$ \eqref{P_avg_tild}). We further note that 
$
\int f(t,\bx,\bp)\textnormal{d}\bp =\int h (t,\bx,\bp)\textnormal{d}\bp \Pi.
$ 
Thanks to the estimate $\|h\|_{L^1_{\bx,\bp}}\geq \|f\|_{L^1_{\bx,\bp}}\|{\Pi}\|_{L_{\bx}^\infty}^{-1}=\|{\Pi}\|_{L_{\bx}^\infty}^{-1}$ and the relation \eqref{h_bound_1}, we have that,
\begin{align}\label{T_1_sam}
T_1=\lf\|\int f\mathrm{d}\bp\lf(1-\frac{|\Torus|^d}{\|h\|_{L_{\bx,\bp}^1}}\rg)\rg\|_{L_\bx^1}= \lf|  \frac{\|h\|_{L^1_{\bx,\bp}}-|\Torus|^{{d}}}{\|h\|_{L^1_{\bx,\bp}}}\rg|\leq  C(f_0,\uu)\|{\Pi}\|_{L_{\bx}^\infty}\|\wt{\mathbb{P}}_\bx(f_{0}e^\uu)\|_{L_{\bx,\bp}^2}e^{-\delta\nu^{1/2}t}. 
\end{align}
On the other hand, we invoke the enhanced dissipation to obtain that
 \begin{align*}
T_2=&\lf\|\frac{\int \Pi(\bx)\lf(\lan h\ran (t,\bp)+\wt{h}(t,\bx,\bp)\rg) \textnormal{d}\bp}{\int \lan h\ran(t,\bp) \textnormal{d}\bp } -\Pi(\bx)\rg\|_{L_\bx^1}\\
\leq&\frac{\|\Pi \wt{h}(t)\|_{L_{\bx,\bp}^1}}{\|\lan h\ran_\bx(t)\|_{L^1_\bp}}\leq \frac{|\Torus|^{d/2}\|\Pi\|_{L_\bx^2}\|\wt{h}(t)\|_{L_{\bx,\bp}^2}}{\|h(t)\|_{L_{\bx,\bp}^1}|\Torus|^{-d}}
\leq  \|{\Pi}\|_{L_{\bx}^\infty}|\Torus|^{3d/2} C(f_0,\uu)\|e^{-\uu}\|_{L^2}\|\wt{\mathbb{P}}_\bx(f_{0}e^\uu) \|_{L^2}e^{-\delta\nu^{1/2} t}.
\end{align*} Here, we recall the notation $\wt{\mathbb{P}}_\bx$ from \eqref{P_avg_tild}. 
Hence, we have obtained the bound \eqref{Sample_2nd}. 

\smallskip
\step{2: Approximation.} In this step, we discuss $L_\bx^\infty$-convergence of the dynamics to the target $\Pi$. To this end, we consider higher-regularity estimates for the solution $h$. We adopt the $d$-dimensional multi-index notation $D_\bx^\al=\prod_{j=1}^d\pa_{x_j}^{\al_j}$, $|\al|=\sum_{j=1}^d\al_j$, $\al'\leq\al$ if and only if $\al_j'\leq\al_j$ for every $j$, and $\al'<\al$ if and only if $\al'\leq\al$ and $\al'\neq\al$. We then consider the renormalized solution $H$ \eqref{EQ:2nd_H} and derive higher-regularity estimates. 

First of all, we observe that the $\bx$-derivatives of the solution are straightforward. For any fixed $M\in \mathbb{N}$, we have that 
\begin{align*}
\frac{1}{2}\ddt\sum_{|\al|\leq M}\|D_\bx^\al H\|_{L^2}^2&=-\kappa\sum_{|\al|\leq M}\|\na_{\bx}D_\bx^\al H\|_{L^2}^2-\nu\sum_{|\al|\leq M}\|
{\na_{\bp}}D_\bx^\al H\|_{L^2}^2\\
&\quad-\sum_{|\al|\leq M}\sum_{\al'\leq \al}C_{\al,\al'}\kappa\iint \na_\bx D^{\al'}_\bx\uu\cdot \na_\bx D_\bx^{\al-\al'}H  D_\bx^{\al}H d\bx \textnormal{d}\bp\\
& \leq C_M\kappa \|\uu\|_{W^{M+1,\infty}_\bx}^2\sum_{|\al|\leq M}\|D_\bx^{\al}H \|_{L^2}^2.
\end{align*} 
As a consequence, we have that $\|H(t)\|_{H_\bx^M}\leq \|H(0)\|_{H_\bx ^M} e^{C_\ast(M) \kappa\|\uu\|_{W^{M+1,\infty}_\bx}^2 t}. $ 

Next, we observe that $f=h \Pi =H \Pi e^{\int_0^t G(s)\ds}$ \eqref{Rnrm_sln} and the integral factor is bounded by \eqref{G_est}. 
As a consequence,
$
\|f (t)\|_{H_{\bx}^M}\leq C(M,f_0,\uu,Z) \exp\{C_\ast(M)\kappa \|\uu\|_{W_\bx^{M+1,\infty}}^2 t\},\quad \kappa\leq \nu.
$

Since the $L^1$-difference between $\dss \int f(t,\bx,\bp)\textnormal{d}\bp $ and $\Pi$ decays to zero at an enhanced rate $\mathcal O(\nu^{1/2})$, an application of the Nash interpolation inequality (with $M>d/2$) yields that the $\int f(t,\bx,\bp)\textnormal{d}\bp $ converges to $\Pi$ in the space $L_\bx^\infty$ as $t\rightarrow\infty$ \eqref{Sample_2nd}, 
\begin{align*}
\lf\|\int_{\mathbb{T}^d} f \textnormal{d}\bp-\Pi \rg\|_{L^\infty_\bx}\leq& C(d,M)\lf\|\int_{\mathbb{T}^d} f \textnormal{d}\bp-\Pi\rg\|_{L_\bx^1}^{\gamma(d,M)}\lf\|\int_{\mathbb{T}^d} (f -|\mathbb{T}|^{-d}\Pi)\textnormal{d}\bp \rg\|_{H_\bx^M}^{1-\gamma(d,M)}\\
\leq& C(d,f_0,\uu,Z)e^{-\delta\gamma \nu^{1/2}t}\sum_{|\al|\leq M}\lf(\int\lf(\int_{\mathbb{T}^d} D_\bx^\al(f -|\mathbb{T}|^{-d}\Pi)\textnormal{d}\bp \rg)^2d\bx\rg)^{(1-\gamma)/2}\\
\leq& C(d,f_0,\uu,Z)e^{-\delta\gamma \nu^{1/2}t}(\|f(t)\|_{H^M_{\bx}L^2_\bp}+\|\Pi\|_{H^M_\bx})^{1-\gamma}\\
\leq& C(d,f_0,\uu,Z)\exp\lf\{C_\ast(M)(1-\gamma)\nu \|\uu\|_{W_\bx^{M+1,\infty}}^2 t-\delta\gamma \nu^{1/2}t\rg\}\\
\leq&C(d,f_0,\uu,Z)e^{-\delta\gamma \nu^{1/2}t/2}.
\end{align*} Here, we have chosen $\nu$ to be small enough in the last line. This completes the proof of the approximation estimate \eqref{App_2nd}.

\smallskip
\step{3: Nonlinear Mixing.}
Now we focus on the initial time layer $t\in[0,\delta^{-1}\nu^{-1/2} ]$ to derive the nonlinear mixing estimate \eqref{Sample_2ndMix}. We start by proving the following. 

\begin{lem}
Assume the conditions in Theorem \ref{thm_2nd_order}. There exists a constant $C_{\rm ID}=C_{\rm ID}(\varsigma,f_0,\uu,\delta^{-1},d)$ such that the following estimate holds for all $F\in H_{\bx,\bp}^{1\vee\{d/2+\varsigma\}}$:
\begin{align}  \label{nonlin_ID}
\quad \lf|\iint\wt{h}(t)F \dx \mathrm{d}\bp\rg|\leq \frac{C_{\rm ID}(\varsigma,f_0,\uu,\delta^{-1},d)   }{ t^{1/2} }\|\wt h_{ 0} \|_{H_{\bx,\bp}^{1\vee\{d/2+\varsigma\}}}\|F\|_{H_{\bx,\bp}^{1\vee\{d/2+\varsigma\}}},\ \  \forall t\in(0, \delta^{-1}\nu^{-1/2}].  
\end{align}
\end{lem}
\begin{proof}
Throughout this proof, we define $\sigma:=1\vee\{d/2+\varsigma\}.$ We organize the proof in three steps. 

\noindent{\bf Step \# 1: Setup and error estimate. } Instead of considering the coupling between $\wt h$ and $F$ directly, we consider the coupling
$\lf|\iint\wt{H}(t)F \dx \mathrm{d}\bp\rg|. 
$ 
Here $\wt H$ is the solution to \eqref{h_eq_nq} and is the renormalized solution associated with $\wt h$ through the formula \eqref{Rnrm_sln}. We further consider the passive scalar solution $ \zeta$ \eqref{pssv} initialized with the same initial data $\wt H_0$ on the time interval $[0,\delta^{-1}\nu^{-1/2}]$. Combining all the considerations above, we have that 
\begin{align}
\begin{split}\lf|\iint\wt{h}(t)F \dx \mathrm{d}\bp\rg|&=e^{-\int_0^t G(s)ds}\lf|\iint\wt{H}(t)F \dx \mathrm{d}\bp\rg|\\
&\hspace{-2cm}\leq e^{-\int_0^t G(s)ds}\lf|\iint(\wt{H}(t)-\zeta(t))F \dx \mathrm{d}\bp\rg|+e^{-\int_0^t G(s)ds}\lf|\iint \zeta (t)F \dx \mathrm{d}\bp\rg|=:T_1+T_2.\end{split}\label{nonlin_ID_pf}
\end{align}
{Thanks to \eqref{pssv_app} and the relation \eqref{Rnrm_sln}, we have that}   
\begin{align}
\|\wt H-\zeta\|_{L^2}(t)&\leq C(\delta^{-1},\uu)\nu^{1/4}\|\wt H_0\|_{L^2}=C(\delta^{-1},\uu)\nu^{1/4}\|\wt h_0\|_{L^2},\quad t\in[0,\delta^{-1}\nu^{-1/2}].
\end{align}
As a consequence, we have that the first term in \eqref{nonlin_ID_pf} can be estimated as follows for all $t\in[0,\delta^{-1}\nu^{-1/2}]$,
\begin{align}\label{nonlin_ID_pf1}
\quad T_1=e^{-\int_0^t G(s)ds}\lf|\iint(\wt{H}(t)-\zeta(t))F \dx \mathrm{d}\bp\rg|\leq C(\delta^{-1},\uu)\nu^{1/4}e^{-\int_0^t G(s)ds}\|\wt h_0\|_{L_{\bx,\bp}^2}\|F\|_{L^2_{\bx,\bp}}.
\end{align}

\noindent 
{\bf Step \# 2: Mixing of the passive scalar. }
Next, we observe that the $\bx$-Fourier transform of $\zeta$ solves the following equation:
\begin{align}
&\pa_t \wh{\zeta}(t,\bk,\bp)+i\bv(\bp)\cdot \bk\wh{\zeta}(t,\bk,\bp)=\nu\de_\bp \wh{\zeta}(t,\bk,\bp)-\kappa|\bk|^2 \wh{\zeta}(t,\bk,\bp),\\
&\quad \wh{\zeta}(t=0,\bk,\bp)=\wh H_0(\bk,\bp),\quad |\bk|\neq 0. 
\end{align}
One can reformulate the problem as follows:
\begin{align}
&\pa_t \lf(\wh{\zeta}(t,\bk,\bp)e^{\kappa|\bk|^2t}\rg)+i\bv(\bp)\cdot \bk\lf(\wh{\zeta}(t,\bk,\bp)e^{\kappa|\bk|^2t}\rg)=\nu\de_\bp \lf(\wh{\zeta}(t,\bk,\bp)e^{\kappa|\bk|^2t}\rg),\\
&\quad \wh{\zeta}(t=0,\bk,\bp)=\wh h_0(\bk,\bp),\quad |\bk|\neq 0. 
\end{align} We observe that the equation is identical to the equation \eqref{eq:hyp_x} with slightly adjusted notation ($\bp$ is the same as $\by$ and $\bv(\bp)$ is the same as $u_1(\by)$). Hence, the linear inviscid damping estimate in Theorem \ref{thm:mix} and the relation \eqref{linear_ID_pf} ensure the following for all $t\geq0$,
\begin{equation}\begin{split}
\lf|\iint \zeta(t)  {F} \dx\mathrm{d}\bp\rg|=&\lf|\sum_{|\bk|\neq 0} \int e^{\kappa|\bk|^2 t}\wh{\zeta}_{\bk}(t,\bp)\overline{e^{-\kappa|\bk|^2 t}\wh F_\bk(\bp)}d\bp\rg|\\
\leq& C_{d,\varsigma} \min\lf\{\frac{\nu^{1/4}}{\min\{1,\nu^{1/4}t^{1/2}\}},e^{-\delta_{0}\nu^{1/2}t}\rg\}\|\wt{h}_0\|_{H_{\bx,\bp}^{\sigma}}\|e^{\kappa t\de_\bx}{F}\|_{L^\infty _t H_{\bx,\bp}^{\sigma}}\\
\leq& C_{d,\varsigma} \min\lf\{\frac{\nu^{1/4}}{\min\{1,\nu^{1/4}t^{1/2}\}},e^{-\delta_{0}\nu^{1/2}t}\rg\}\|\wt{h}_0\|_{H_{\bx,\bp}^{\sigma}}\|{F}\|_{ H_{\bx,\bp}^{\sigma}}.\end{split}\label{nonlin_ID_pf2}
\end{equation}
Here, in the last line, we have used the fact that $e^{\kappa t\de_\bx}$ is a contraction in $H^{\sigma}_{\bx,\bp}$. 

\noindent{\bf Step \# 3: Conclusion.} Finally, we observe that for $t\leq \delta^{-1}\nu^{-1/2}$, the relation $\nu^{\frac14}\leq C(\delta^{-1})t^{-\frac12}$ holds.  
Combining the decomposition \eqref{nonlin_ID_pf} with the estimates \eqref{nonlin_ID_pf1} and \eqref{nonlin_ID_pf2}, we obtain, for all $t\in(0, \delta^{-1}\nu^{-1/2} ]$:{\small
\begin{align}
    \lf|\iint\wt{h}(t)F \dx \mathrm{d}\bp\rg|
    &\leq C(\delta^{-1},\uu,d,\varsigma)e^{-\int_0^t G(s)ds}\Big(\nu^{1/4} +  \min\lf\{\frac{\nu^{1/4}}{\min\{1,\nu^{1/4}t^{1/2}\}},e^{-\delta_{0}\nu^{1/2}t}\rg\}\Big)\|\wt{h}_0\|_{H_{\bx,\bp}^{\sigma}}\|{F}\|_{ H_{\bx,\bp}^{\sigma}}\\
    &\leq C(\delta^{-1},\uu,d,\varsigma)\frac{e^{-\int_0^t G(s)ds}}{t^{1/2}} \|\wt h_{0}\|_{H^{\sigma}}\|{F}\|_{ H_{\bx,\bp}^{\sigma}}
    \leq \frac{C(\delta^{-1},f_0,\uu,d,\varsigma)}{t^{1/2}} \|\wt h_{0}\|_{H^{\sigma}}\|{F}\|_{ H_{\bx,\bp}^{\sigma}}.
\end{align} }Here, in the last line, we used the estimate \eqref{G_est}. 
This concludes the proof of \eqref{nonlin_ID}. 
\end{proof}

\begin{proof}[Proof of Theorem \ref{thm_2nd_order}] Using the lemma above, we are finally ready to derive the last estimate \eqref{Sample_2ndMix} in Theorem \ref{thm_2nd_order}. 
We assume that $\varphi\in H_\bx^{\sigma}$ and recall 
$[f](t,\bx):=\int_{\mathbb{T}^{d}} f(t,\bx,\bp)\dpp$ \eqref{hydrodnsty}, and consider the difference,
\begin{align}
{\mathscr T}:=&\lf|\int \varphi ( \bx )  \Pi ( \bx )  \dx - \int \varphi(\bx) [f](t,\bx) \dx\rg|
=\lf|\int \varphi ( \bx )  \Pi ( \bx )  \dx - \iint  f(t,\bx,\bp)   \varphi(\bx)\dx \textnormal{d}\bp\rg|. 
\end{align}
 We observe that by the relations $h\Pi =f$ and $\|\br h\|_{L^1_\bp}=|\Torus|^{-d}\|h\|_{L^1_{\bx,\bp}}$, the following decomposition holds 
\begin{equation}\label{observable}
\begin{split}
\lf|\iint f(t,\bx,\bp)\varphi(\bx) \textnormal{d}\bp \dx\right.&\left.-\int\varphi(\bx) \Pi(\bx) \dx\rg|  \leq\lf|\iint f(t,\bx ,\bp)\varphi(\bx ) \textnormal{d}\bp\dx \lf( 1-\frac{1}{\|\br h\|_{L^1_\bp} }\rg)\rg|\\
 & \ \ +\lf|\int\frac{\int h(t,\bx ,\bp)\Pi (\bx )\varphi (\bx )\mathrm{d}\bp}{\| \lan h\ran \|_{L^1_\bp} }-\Pi \varphi \dx \rg |=:T_1+T_2.
 \end{split}
 \end{equation}
We start by  considering the term $T_1$ with a focus on estimating the factor $1-\frac{1}{|\Torus|^{-d}\| h\|_{L^1} }$. This quantity can be estimated by a suitable combination of $\wt h$ and $\Pi$. Hence, we first invoke the nonlinear mixing estimate  \eqref{nonlin_ID} to derive the following bound 
\begin{align}
\lf|\iint \wt{h} (t,\bx ,\bp) \Pi(\bx) \textnormal{d}\bp \dx \rg|&=\lf|\iint \wt{h} (t,\bx ,\bp) (\Pi(\bx)-\lan\Pi\ran) \textnormal{d}\bp \dx \rg|\\
&\leq \frac{C(\varsigma,f_0,\uu,\delta^{-1},d)}{t^{1/2}}\|\wt h_{0}\|_{H^{\sigma}_{\bx,\bp}}\|\Pi-|\Torus|^{-{ d}}\|_{H^{\sigma}_{\bx  }}.
\end{align}
Here, we have used that $\lan\Pi\ran_\bx(\bp)\equiv|\Torus|^{-{  d}}$ and the fact that the function $\Pi-|\Torus|^{-d}$ does not depend on $\bp$.  Now the relations  $\|\br h_\bx\|_{L^1_\bp}=|\Torus|^{-d}\|h\|_{L^1_{\bx,\bp}}$  and $0=\iint (\lan h\ran_\bx+\wt{h}-\frac{1}{|\mathbb{T}|^d} ) \Pi$  \eqref{ED_mass}, together with the estimate above, yield the following: 
\begin{align}
||\Torus|^{-d}\|h&\|_{L^1_{\bx,\bp}}- 1 |=|\|\lan h\ran_\bx\|_{L^1_\bp}- 1 |=\lf|\int \lf(\lan h\ran_\bx-\frac{1}{|\mathbb T|^d}\rg) \textnormal{d}\bp\int \Pi(\bx)\dx \rg|\\
&=\lf|\iint \wt{h} \Pi \dx \textnormal{d}\bp \rg|=\lf|\iint \wt{h} (\Pi-|\Torus|^{-d}) \dx \textnormal{d}\bp \rg|
\leq   \frac{C(\varsigma,f_0,\uu,\delta^{-1},d)}{t^{1/2}}\|\wt h_{0}\|_{H^{\sigma}_{\bx,\bp}}\|\Pi-|\Torus|^{-d}\|_{H^{\sigma}_{\bx,\bp }}.
\end{align}
Hence, thanks to the estimate $\|h\|_{L^1_{\bx ,\bp}}\geq \|{\Pi}\|_{L_{\bx}^\infty}^{-1}$, the fact that $\|f(t)\|_{L^1}=1$ and the Sobolev embedding,
\begin{align*}
T_1&\leq \lf|\iint f  \varphi \dx \textnormal{d}\bp\rg| \lf| \frac{\|h\|_{L^1_{\bx ,\bp}}-|\Torus|^{{ d}}}{\|h\|_{L^1_{\bx ,\bp}}}\rg|\\
&\leq  C(d)\|f\|_{L_{\bx,\bp}^1}\|\varphi\|_{L_\bx^\infty}\|{\Pi}\|_{L_{\bx}^\infty}\frac{C(\varsigma,f_0,\uu,\delta^{-1},d)  }{t^{1/2}}\|\wt h_{0}\|_{H^{\sigma}_{\bx,\bp}}\|\Pi-|\Torus|^{-d}\|_{H^{\sigma}_{\bx }}
\leq \frac{C(d,\varsigma, f_0,\uu,\Pi)}{t^{1/2}}\|\varphi\|_{H^{\sigma}_\bx} . 
\end{align*}
For the $T_2$-term in \eqref{observable},  we observe a cancellation of the $\br{h}_\bx$-information,  
 \begin{align*}
T_2=&\frac{1}{\|\br h_\bx\|_{L^1_\bp}}\lf |\iint \lf(\lan h\ran_\bx (t,\bp)+\wt{h}(t,\bx ,\bp)\rg)  \varphi(\bx )\Pi(\bx )\textnormal{d}\bp\dx  -{\int \lan h\ran_\bx(t,\bp) \textnormal{d}\bp }\int\varphi(\bx )\Pi(\bx )\dx \rg |\\
=&\frac{1}{\|\br h_\bx\|_{L^1_\bp}}\lf |\iint \wt{h}(t,\bx ,\bp) \varphi(\bx )\Pi(\bx )\textnormal{d}\bp\dx  \rg |.
\end{align*}
Now we invoke duality, the lower bound $\|\br h_\bx\|_{L^1_\bp}\ge |\Torus|^{-d}\|\Pi\|_{L^\infty_\bx}^{-1}$ and the mixing estimate \eqref{nonlin_ID}, to obtain the bound
\begin{align*}
T_2
\leq&\frac{C(d,\varsigma, f_0,\uu)}{t^{1/2}}\|{\Pi}\|_{L_{\bx }^\infty}\| \Pi \varphi-\lan{\Pi \varphi}\ran_\bx\|_{H_\bx^{\sigma}}\|\wt h_{0}\|_{H^{\sigma}_{\bx,\bp}}
\leq \frac{C(d,\varsigma, f_0,\uu,\Pi)}{t^{1/2}} \|  \varphi \|_{H_\bx^{\sigma}}.
\end{align*} 
In the last line, we have used the product estimate for Sobolev functions and the observation that $\br{\Pi \varphi}_\bx$ is a constant.  
Hence, we obtain $
\mathscr{T}\leq \frac{C(d,\varsigma,\Pi,\uu, f_0)}{t^{1/2}}\|\varphi\|_{H_\bx^{\sigma}},\ \forall t\in(0, \delta^{-1}\nu^{-1/2}].$
This concludes the proof of \eqref{Sample_2ndMix} and therefore  the proof of Theorem \ref{thm_2nd_order}. 
\end{proof}

\section{Numerical Simulations}\label{sec:Numerical_Result}
\subsection{Numerical Setup}
In this section, we present a numerical experiment for the flow-accelerated Fokker--Planck equation \eqref{EQ:2nd_Ord} to illustrate the accelerated convergence from arbitrary initial data to the target invariant measure $\Pi$, as proved in Theorem \ref{thm_2nd_order}. The design is based on an unconditionally stable finite-volume scheme.

For a cleaner presentation, we choose the Vicsek-type equation in \eqref{sphere_model} with $d=2$ and $\kappa = \nu$:
\begin{align}\label{eq:recall3d}
\begin{cases}&\dss\pa_t f=\nu\de_\te f+ \nu\na_\bx\cdot \lf(f\na_\bx\log\lf(\frac{f}{\Pi}\rg)\rg)-\Pi\ \na_\bx \cdot\lf(\frac{\bp f}{\Pi}\rg)- {f}\iint f\, \bp\cdot \na_\bx \log\Pi\; \dx d \te\\
&\dss f(t=0,\bx,\te)= f_{0} (\bx,\te),   \quad  \bx\in \Torus^2,\, \te\in \mathbb{S}, \end{cases}
\end{align} 
where $\bp(\te)=(\cos\te, \sin\te)^T$. All the enhanced estimates developed in the previous sections can be adapted to this setting. 
Similar to the renormalization procedure in \eqref{Rnrm_sln}, we observe that  if $f$ solves \eqref{eq:recall3d}, then $\tilde{f}=f \exp\big\{\int_0^t \iint f\, \bp\cdot \na_\bx \log\Pi\; \dx d \te d s\big\}$ solves the equation (after dropping $\tilde{(\cdot)}$)
\begin{align}\label{eq:3d_birth}
\pa_t f =&\nu\de_\te f+\nu\na_\bx\cdot\lf(f\na_\bx\log\lf(\frac{f}{\Pi}\rg)\rg)-\Pi\na_\bx\cdot\lf(\frac{\bp f}{\Pi}\rg),     \quad  \bx=(x_1,x_2)\in \Torus^2,\, \theta\in \mathbb{S}. 
\end{align}
Thus, we only need to design schemes to compute \eqref{eq:3d_birth}; the renormalized solution $\frac{f(t,\bx,\te)}{\int f(t,\bx,\te) \mathrm d\bx\,\mathrm d\te}$ then solves the original equation \eqref{eq:recall3d}.
The corresponding $h$-equation (after dropping tilde) in terms of $h=\frac{f}{\Pi}$ is
$
\pa_t h\myh{=&\nu\de_\te \frac{f}{\Pi}+\frac{\nu}{\Pi}\na_\bx\cdot\lf(f\na_\bx\log\lf(\frac{f}{\Pi}\rg)\rg)-\na_\bx\cdot(\bp \frac{f}{\Pi})\\}
=\nu\de h+\nu\na_\bx(\log \Pi)\cdot \na_\bx h-\binom{\cos\te}{\sin\te}\cdot\na_\bx h. $

 We first recast the Fokker-Planck equation \eqref{eq:3d_birth} as follows
\begin{equation}\label{FP-N}
\begin{aligned}
\pa_t f &=\nu\nabla_\te  \cdot\bbs{\Pi \nabla_{\te} \frac{f}{\Pi}} +\nu \nabla_\bx \cdot \lf( \Pi \nabla_\bx \frac{f}{\Pi} \rg)-\Pi\na_\bx\cdot\lf(\bp \frac{f}{\Pi}\rg)   \\ 
&= \nu \nabla \cdot \lf( \Pi \nabla \frac{f}{\Pi} \rg)-\Pi\na_\bx\cdot\lf(\bp \frac{f}{\Pi}\rg), \quad  \bx=(x_1,x_2)\in \Torus^2,\, \theta\in \Torus. 
\end{aligned} 
\end{equation}
Here $\bp = (\cos \theta, \sin \theta)^\top$ so we have $\nabla_\bx  \cdot \bp = 0$. Expressed in terms of the drift $\bb=(\cos \theta, \sin \theta, 0)^\top$, the last term on the right of \eqref{FP-N} reads
$-\Pi\na_\bx\cdot\Big(\bp \frac{f}{\Pi}\Big)= -\Pi\na \cdot\Big(\bb \frac{f}{\Pi}\Big).$

To design a finite volume scheme based on a general Voronoi tessellation in $\Torus^d$ , we let $x_i, \ i=1,\ldots, n$ denote the cell (or volume) centers
and define the Voronoi cells
\begin{equation}
C_i:= \{x\in \Torus^d; d( x,x_i)\leq d(x,x_j) \text{ for all }x_j\} \quad  \text{ with volume } |C_i|=\hs^d(C_i).
\end{equation}
Here $\hs^d(C_i)$ is the $d$-dimensional Hausdorff measure of cell $C_i$.
Then $\Torus^d=\bigcup_{i=1}^n C_i$ is a Voronoi tessellation of $\Torus^d$. Let $V\!F(i):=\{j; ~\Gamma_{ij}\neq \emptyset\}$ denote the non-empty faces adjacent to  cell $C_i$, 
\begin{equation}
\Gamma_{ij}:= C_i\cap C_j   \text{ with its area  } |\Gamma_{ij}|=\hs^{d-1}(\Gamma_{ij}), \qquad j=1, \cdots, n.
\end{equation}
In each cell $C_i$, the discrete densities $f_i|C_i|$, $\Pi_i|C_i|$  are computed through flux interface, using the following  finite-volume scheme.
Let $\vec{n}\in \bR^d$
denote the outward unit normal vector field on $\pt C_i$ (pointing from $i$ to its adjacent $j$), and let 
$(\bb \cdot \vec{n})^\pm_{ij}\geq0$
denote the positive and respectively negative parts of $\vec{b} \cdot \vec{n}$ at the intersection point of the  geodesic from $x_i$ to $x_j$. 
Integrating \eqref{FP-N} over $C_i$,
\begin{equation}\label{tt310}
\begin{aligned}
\frac{\ud }{\ud t} f_i|C_i|\approx \frac{\ud }{\ud t} \int_{C_i} f\,\mathrm d\hs^d 
\approx& \sum_{j\in V\!F(i) } \int_{\Gamma_{ij}}  \mathbf n \cdot \bbs{ \nu \Pi\nabla \frac{f}{\Pi} }\,\mathrm d\hs^{d-1} + \Pi_i \sum_{j\in V\!F(i) } \int_{\Gamma_{ij}}  \mathbf n \cdot \bbs{ -\bb  \frac{f}{\Pi} }\,\mathrm d\hs^{d-1},
\end{aligned}
\end{equation}
where we use the constant cell probabilities
  $f_i|C_i|$ and $\Pi_i|C_i|$ to approximate the exact solution   on each cell $C_i$.
We end up with  the following finite volume scheme for cells $i=1, \cdots, n$,
\begin{equation}\label{mp}
\frac{\ud}{\ud t}f_i |C_i|=  \sum_{j\in V\!F(i)} |\Gamma_{ij}| \bbs{ \nu \frac{\Pi_i+\Pi_j}{2|x_j-x_i|} \bbs{\frac{f_j}{\Pi_j} - \frac{f_i}{\Pi_i}}   + \big[(\vec{b}\cdot \vec{n})_{ij}^- \frac{f_j}{\Pi_j} -  (\vec{b}\cdot \vec{n})_{ij}^+ \frac{f_i}{\Pi_i}]\Pi_i }.
\end{equation}
A detailed 3D time-space discretization and  stability analysis of \eqref{mp} are given in Appendix \ref{sec_app2}.

\subsection{Numerical Results}\label{sec:numerical-results}

Based on the upwind finite-volume scheme \eqref{mp}, we compute the evolution of the dynamic density $f$ from the prescribed initial data toward a smile-shaped target invariant measure. 

Take a smile-shaped target invariant measure
 \begin{equation}\label{smile}
 \begin{aligned}
 \Pi(x,y) \propto &e^{ -20\big[\bbs{x-\frac65}^2+\bbs{y-\frac{6}{5}}^2 - \frac12\big]^2  + \log\bbs{e^{-10(y-2)^2}} } +  e^{-20 \big[\bbs{x+\frac65}^2+\bbs{y-\frac{6}{5}}^2 - \frac12\big]^2 + \log\bbs{e^{-10(y-2)^2}}} \\
 &+ e^{-20 \bbs{x^2+y^2 - 2}^2 + \log\bbs{e^{-10(y+1)^2}}}+0.1;
 \end{aligned}
 \end{equation}
 see the sixth subfigure in the bottom right corner of Fig.\ref{fig_evl}.
 
Take a  Gaussian mixture
\begin{equation}\label{f_initial}
f_0(x,y)\propto e^{-16(x+3)^2-4y^2}+ e^{-16(x-3)^2-4y^2}+ e^{-4x^2-16(y+3)^2}+ e^{-4x^2-16(y-3)^2}+0.1
\end{equation}
as an initial density $f_0(x,y,\theta)=f_0(x,y)$; see the first subfigure in the top left corner of Fig.\ref{fig_evl}. 
\begin{figure}\includegraphics[scale=0.34]{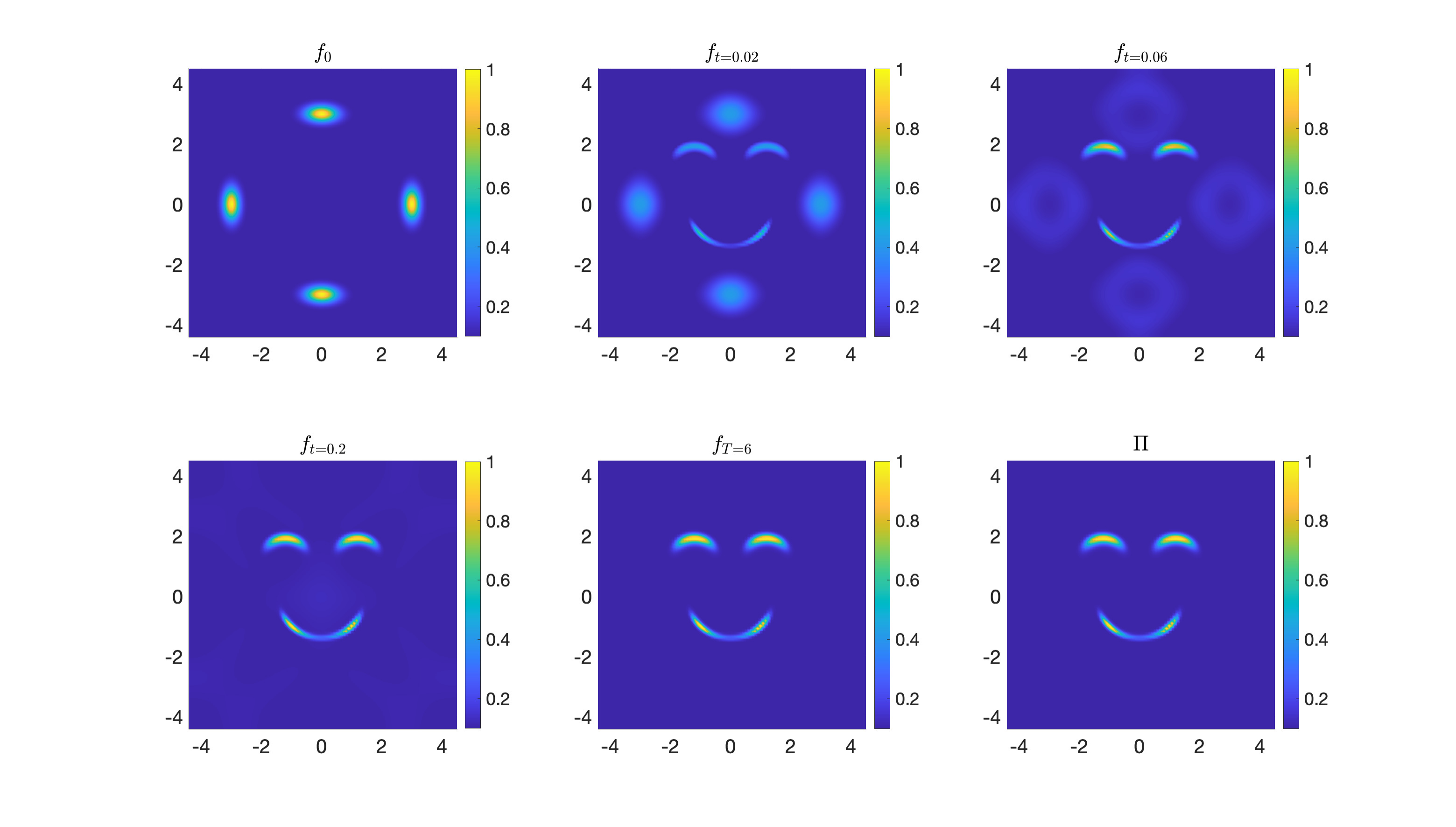} 
\caption{The time evolution of density $f$ computed from scheme \eqref{mp}. The initial density $f_0$ is taken as \eqref{f_initial} and the target invariant measure is taken as \eqref{smile}. Snapshots of $f_t$ are shown at $t=0.02, 0.05, 0.2, 6.$}\label{fig_evl}
\end{figure} 

The time evolution of the hydrodynamic density $\int f(t,x,y,\theta)\,\mathrm d\theta$ is shown in Fig. \ref{fig_evl}
with the time step   $\Delta t=0.002$, the diffusion constant   $\kappa=\nu=0.25$ and the total computational time   $T=6$. The grid sizes are listed in detail in Appendix \ref{sec_app2}. One can see the time-dependent density $f$ converges rapidly to the invariant measure $\Pi$. In particular, the density $f$ at $t=0.2$ and $T=6$ (shown at the bottom of Fig.\ref{fig_evl}) already looks very similar to the invariant measure $\Pi.$   

The $L^1$ and relative $L^2$ errors are computed via $\|\int_0^{2\pi} f(t,x,y,\theta) \ud \theta - \Pi \|_{L^1_{x,y}}$ and $\left\|\frac{2\pi f}{\Pi}-1\right\|_{L^2_{x,y,\theta}}$, respectively. The decay of these errors with respect to time is shown in Fig. \ref{fig_error} with a red dashed line and a black dashed line, respectively. 
\begin{figure}[ht]
\includegraphics[scale=0.3]{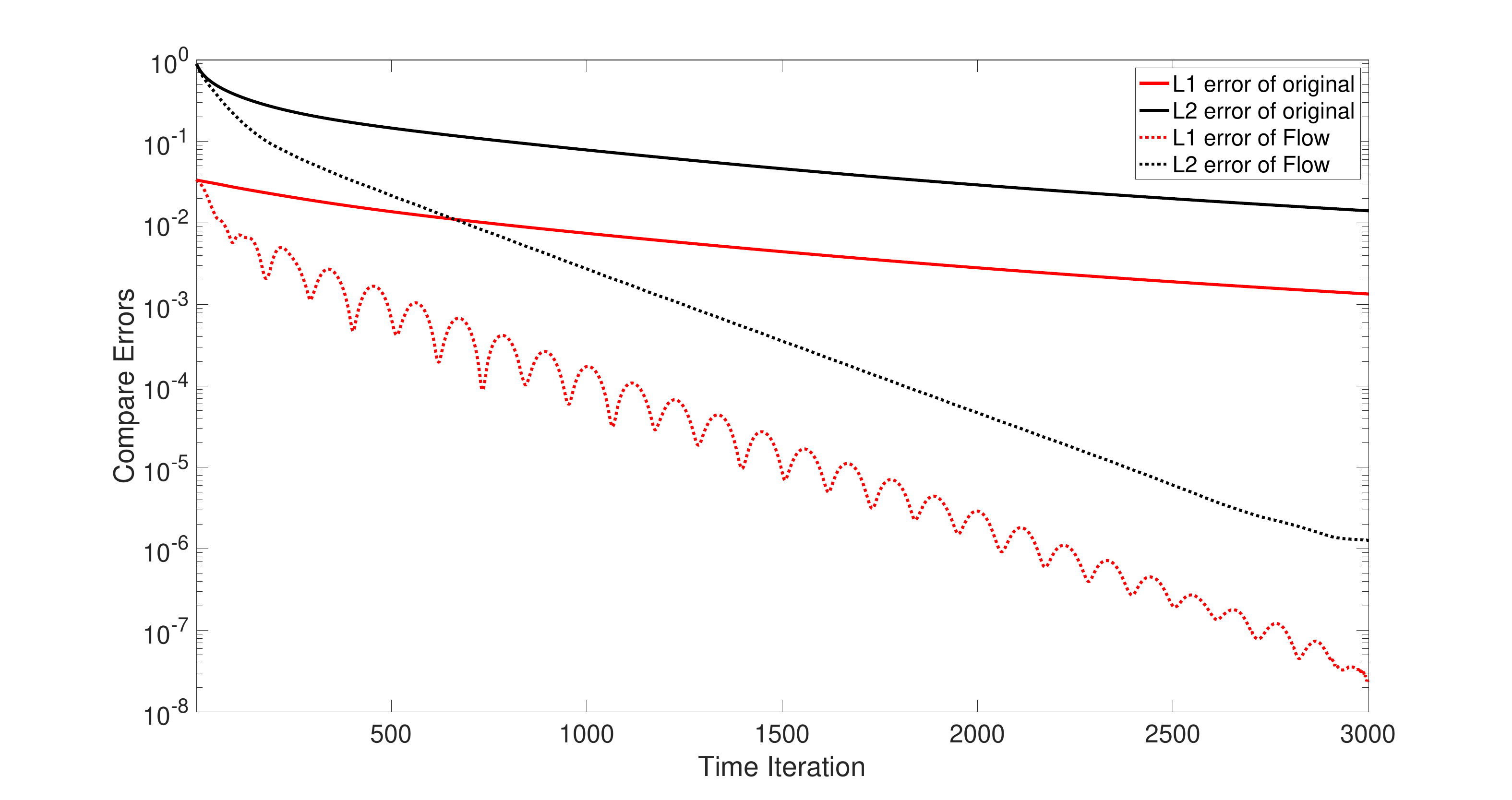}
\caption{The decays of $L^1$ and relative $L^2$ errors between the dynamic density and the invariant measure are shown for both the flow-accelerated Fokker-Planck equation \eqref{eq:recall3d} and the original   Fokker-Planck equation in the gradient flow form \eqref{reducedFP}. The solid red and black lines represent the decay of $L^1$ and $L^2$ errors for \eqref{reducedFP}. The dashed red and black lines represent the decay of $L^1$ and $L^2$ errors for \eqref{eq:recall3d} after flow acceleration. }\label{fig_error}
\end{figure} 
  
For comparison, we also show the convergence rate of the dynamic solutions to the invariant measure for both the original linear Fokker-Planck equation in the gradient flow form and the flow-accelerated Fokker-Planck equation \eqref{eq:recall3d}.  For the original linear Fokker-Planck equation in the gradient flow form, i.e., the no-flow case $\bb=0$, the scheme \eqref{mp} reduces to a classical finite volume scheme for
\begin{equation}\label{reducedFP}
\pa_t f =  \nu\na_\bx\cdot\lf(f\na_\bx\log\lf(\frac{f}{\Pi}\rg)\rg) = \nu \nabla_\bx \cdot \lf( \Pi \nabla_\bx \frac{f}{\Pi} \rg).
\end{equation}
We refer to \cite{gao2023data} for the convergence analysis of the finite volume scheme for the reduced Fokker-Planck equation \eqref{reducedFP} on a compact manifold.
 The decay rates of the $L^1$ and relative $L^2$ errors for the original linear Fokker-Planck equation in the gradient flow form are shown in Fig. \ref{fig_error} with a red solid line and a black solid line, respectively.

\appendix
\section{Connections between the Models}

\begin{lem}[{\bf Normalization}] \label{lem:Connections}
There is a one-to-one correspondence between the regular solution $\widehat{\rho}$ to the \emph{simplified} dynamics \eqref{EQ:1st_Ord_sim} and the \emph{mass-preserving} dynamics \eqref{EQ:1st_Ord}. Specifically, given a solution $\widehat{\rho}$ to \eqref{EQ:1st_Ord_sim}, the mass-normalized solution $\widehat{\rho}/\|\widehat {\rho}\|_{L_\bx^1}$ solves \eqref{EQ:1st_Ord}. Similarly, there is a one-to-one correspondence between the regular solution $\widehat{f}$ to the \emph{simplified} dynamics \eqref{EQ:2nd_Ord_sim} and the \emph{mass-preserving} dynamics \eqref{EQ:2nd_Ord}. 
\end{lem}
\begin{proof} 
We focus on the first-order model because the argument for the second-order model is similar. 
\myh{Recall\begin{align}\label{eq:app_1}&\begin{cases}&\displaystyle\pa_t \rho+ {\vphantom{\int}\vv \cdot\na_\bx \rho} = {\vphantom{\int}\nu\de_\bx \rho+\nu\na_\bx\cdot(\rho \na_\bx{\uu})}   + {\rho\Big(-\vv\cdot \na_\bx {\uu} +\int \rho\, \vv\cdot \na {\uu} \dy\Big)} ,\\
&\na_\bx\cdot\vv(t,\bx)=0,\quad\rho(t=0,\bx)=\rho_{0}(\bx),\quad (t,\bx)\in \rr_+\times \Torus^d.  \end{cases}
\end{align}} 
As in the renormalization procedure in \eqref{Rnrm_sln}, denote 
$g(t):=\int\rho\, \vv\cdot \na_\bx \uu\;\dx=\int \wt \rho \na_\bx \cdot(\vv \uu)\; \dx.
$ 
One can see that if $\rho$ solves the \emph{mass-preserving} dynamics \eqref{EQ:1st_Ord}
, then $\widehat{\rho}=\rho \exp\{\int_0^t g(s) d s\}$ solves the following equation
\begin{align}\label{eq:App_2}
\pa_t \widehat \rho+\vv\cdot\na_\bx\widehat \rho =&\nu\de_\bx \widehat \rho+\nu\na_\bx\cdot\lf(\widehat \rho\na_\bx\uu\rg)-\widehat\rho\vv\cdot\na\uu. 
\end{align}
On the other hand, if $\widehat\rho\geq 0$ solves the equation 
\eqref{eq:App_2}, its mass solves the equation:
$
\ddt\|\wh \rho\|_{L_\bx^1}=-\int \wh \rho \vv\cdot \nabla\uu\dx.$ Combining this relation and \eqref{eq:App_2} yields \eqref{EQ:1st_Ord}
. This establishes the relation between regular solutions to \eqref{EQ:1st_Ord} 
 and \eqref{eq:App_2}. 
\end{proof}
\section{Space-time Discretization and Unconditional Stability}\label{sec_app2}
We outline the numerical scheme used for the computations in section \ref{sec:numerical-results}, based in the 3D domain box $\Omega:= [a,b]\times [c,d]\times [e,f]$ with  $[a,b]=[-4.5,4.5], [c,d]=[-4.5,4.5]$ and $[e,f]=[0,2\pi]$. We use an equi-partition of $\Omega$  into small cells of size $\Delta x = \frac{b-a}{I}, \Delta y= \frac{d-c}{J}$ and $\Delta z=\frac{f-e}{K}$  with $I=100, J=100, K=50.$ 
Let $(x_i,y_j,z_k)$ denote the cell centers
\begin{align*}
x_i = a+ (i-\tfrac{1}{2})\Delta x,  \ \  y_j= c+ (j-\tfrac{1}{2}) \Delta y, \ \  z_k = e + (k-\tfrac{1}{2})\Delta z, \quad  1\leq i\leq I, \ 1\leq j\leq J, \  1\leq k \leq K.
\end{align*}
We use $f_\bl(t)$ to approximate the value of $f(t,x_i, y_j,z_k)$ and set, $\displaystyle g_\bl:=\frac{f_\bl}{\pi_\bl}$ with $\pi_\bl=\pi(x_i, y_j,z_k)$,  where $\bl$ stand for the multi-index $(i,j,k)$. The finite volume discretization \eqref{mp} reads (we use the usual notation of $\mu^\pm$ and $\delta^\pm$ to denote forward and backward  averaging and differencing\footnote{Thus, for example, $\mu_y^+f_{i,j,k}=\tfrac{1}{2}(f_{i,j+1,k}+f_{i,j,k})$ and $\delta_x^-f_{i,j,k}=f_{i,j,k}-f_{i-1,j,k}$.}, and abbreviate $c^\pm_k=\cos^\pm(z_k), s_k^\pm=\sin^\pm(z_k)$)
\begin{align}\label{fp_n}
\dot{f}_{\bl}&=\frac{1}{\Delta x^2}\Big((\nu \mu_x^+\pi_{\bl}+\Delta x c_k^-\pi_\bl)\delta_x^+g_\bl -(\nu \mu_x^-\pi_\bl + \Delta xc_k^+\pi_{\bl})\delta_x^- g_\bl\Big)\\
+& \frac{1}{\Delta y^2}\Big((\nu \mu_y^+\pi_{\bl}+\Delta y s_k^-\pi_\bl)\delta_y^+g_\bl -\left(\nu \mu_y^-\pi_\bl + \Delta y s_k^+\pi_{\bl}\right)\delta_y^- g_\bl\Big) + \frac{1}{\Delta z^2}\Big((\nu \mu_z^+\pi_{\bl})\delta_z^+g_\bl -(\nu \mu_z^-\pi_\bl )\delta_z^- g_\bl\Big).
\end{align}
This is augmented with periodic boundary conditions for $f_\bl$ and $\pi_\bl$.
The expression on the right involves a seven point stencil: the value at the cell center, $g_{i,j,k}$, and its six neighboring values $\{g_\bl :  |\bl-(i,j,k)|=1\}$; we use the abbreviated writing $\dot{f}_\bl=RHS(g_\bl)$. We then use semi-implicit time discretization. Denoting $f^n_\bl=f(t^n,x_i,y_j,z_k)$  as the value of $f$ at $t^n=n\Delta t$, we compute
 \begin{equation}\label{fp_t}
\frac{f^{n+1}_{i,j,k}-f^{n}_{i,j,k}}{\Delta t} = RHS\big(g^{n+1}_{i,j,k}, g^n_{\bl :  |\bl-(i,j,k)|=1}\big)
\end{equation}
This enables us to recover an explicit form of $f^{n+1}_\bl$.
Define
\begin{equation}
\begin{aligned}
\lambda_\bl:= &  \frac{1}{\Delta x ^2}\left(\frac{\nu}{\pi_\bl} (\mu_x^+\pi_\bl
+\mu_x^-\pi_\bl)+\Delta x(c_k^++c_k^-)\right) \\
 & \ \ + 
\frac{1}{\Delta y ^2}\left(\frac{\nu}{\pi_\bl}(\mu_y^+\pi_\bl+ \mu_y^-\pi_\bl) +\Delta y(s_k^++s_k^-)\right)
+ \frac{1}{\Delta z^2}\left(\frac{\nu}{\pi_\bk}(\mu_z^+\pi_\bl + \mu_z^-\pi_\bl)\right),
\end{aligned}
\end{equation}
then \eqref{fp_t} reads
\begin{equation}\label{tm310}
\begin{aligned}
(1+\Delta t \lambda_\bl)g^{n+1}_\bl = 
 g^{n}_\bl
& +  \frac{\Delta t}{\Delta x ^2} \Big(\frac{\nu}{\pi_\bl}\mu_x^+\pi_\bl+ \Delta x c_k^-\Big)g^n_{i+1,j,k}  + \frac{\Delta t}{\Delta x ^2}\Big(\frac{\nu}{\pi_\bl} \mu_x^-\pi_\bl+\Delta x c_k^+\Big)g^{n}_{i-1,j,k} \\
& + \frac{\Delta t}{\Delta y ^2} \Big(\frac{\nu}{\pi_\bl}\mu_y^+\pi_\bl+ \Delta y s_k^-\Big)g^n_{i,j+1,k}  + \frac{\Delta t}{\Delta y ^2}\Big(\frac{\nu}{\pi_\bl} \mu_y^-\pi_\bl+\Delta y s_k^+\Big)g^{n}_{i,j-1,k} \\
&+ \frac{\Delta t}{\Delta z ^2} \Big(\frac{\nu}{\pi_\bl}\mu_z^+\pi_\bl\Big)g^n_{i,j,k+1}  + \frac{\Delta t}{\Delta z ^2}\Big(\frac{\nu}{\pi_\bl} \mu_z^-\pi_\bl\Big)g^{n}_{i,j,k-1}.
\end{aligned}
\end{equation}
Observe that the coefficients of the seven $g^n$-terms corresponding to seven-point stencil on the right of \eqref{tm310} are non-negative and sum up to $1+\Delta t\lambda_\bl$.
Thus, $g^{n+1}_{i,j,k}$ is given as a convex combination of $g^n_{i,j,k}$ and its six 
neighboring values that participate in the stencil, $g^n_{\bl: |\bl-(i,j,k)|=1}$.
In particular, we conclude that $\displaystyle \frac{f}{\pi}$ and likewise, $\displaystyle \frac{f}{\pi}-1$, satisfy the maximum principle.
\begin{lem}
If $\pi>0$ and initial data $f_0>0$, then the scheme \eqref{tm310} satisfies
\begin{enumerate}
\item positivity-preserving property $f^n_{i,j,k}>0$ for any $i,j,k,n;$
\item unconditional maximum principle: 
$\displaystyle 
\max_{i,j,k} \frac{f^{n+1}_{i,j,k}}{\pi_{i,j,k}} \leq  \max_{i,j,k} \frac{f^{n}_{i,j,k}}{\pi_{i,j,k}};
$ 
\item $\ell^\8$ stability:
$\displaystyle 
\max_{i,j,k} \Big|\frac{f^{n+1}_{i,j,k}}{\pi_{i,j,k}}-1\Big| \leq  \max_{i,j,k} \Big|\frac{f^{n}_{i,j,k}}{\pi_{i,j,k}}-1\Big|.
$
\end{enumerate}
\end{lem}

\section{The Particle System of the Sampling Dynamics}\label{sec:App-D}
In this section, we propose two particle models which resemble the dynamics of \eqref{EQ:1st_Ord} and \eqref{EQ:2nd_Ord}. As a consequence, the effective simulation of these particle dynamics also provides a fast sampling algorithm.

Motivated by \cite{BenPoratCarrilloJabin25,Nastassia22}, we consider the empirical measure
$
\mu_N(t):=\frac{1}{N}\sum_{i=1}^N m_i(t)\delta_{\bx_i(t)}.
$ 
Let ${\bf m}_N:=(m_1,\cdots, m_N)$, where $m_i(t=0)= {1}$ with time-invariant total mass  $\sum_{i=1}^Nm_i=N$. Then the dynamics is 
\begin{align}
\begin{cases}
\mathrm{d}\bx_i(t)=\vv(t,\bx_i(t))\mathrm {d}t-\nu\na_\bx\uu(\bx_i(t))\mathrm{d}t+\sqrt{2\nu}\mathrm{d}B_i(t),\\[0.3em]
\displaystyle \mathrm d m_i(t)=\frac{m_i(t)}{N}\sum_{j=1}^N m_j(t)\Bigl(\A(t,\bx_j(t))-\A(t,\bx_i(t))
\Bigr)\mathrm dt,\vphantom{\int}
\\[0.3em]
\A(t,\bx):=\vv(t,\bx)\cdot\na \uu(\bx), \vphantom{\int}
\\[0.3em]
\bx_i(t=0)=\bx_{i;0},\quad m_i(t=0)=m_{i;0}=1,\quad i\in\{1,\cdots,N\}.
\end{cases}\label{1st_Particle}
\end{align}

Next we consider the second-order sampling dynamics \eqref{EQ:2nd_Ord}. Let the empirical measure be 
$
\mu_N:=\frac{1}{N}\sum_{i=1}^N m_i\delta_{(\bx_i(t),\bp_i(t))},
$ 
Let ${\bf m}_N:=(m_1,\cdots, m_N)$,with time-invariant total mass $\sum_{i=1}^Nm_i=N$. Then the dynamics is 
\begin{align}\begin{cases}
&\dx_i(t)=\bv(\bp_i(t))dt{-}\kappa\na_\bx\uu(\bx_i(t))\dt+\sqrt{2\kappa}\mr dB_t^{i},\\
& \mr d  \bp_i(t)=\sqrt{2\nu}\mr d\mathbb{B}_t^i,\\
 &\mr d m_i(t)=\frac{1}{N}\sum_{j=1}^N m_i(t)m_j(t)\Big(\bv( \bp_j(t))\cdot \na {\uu}(\bx_j(t))-\bv(\bp_i(t))\cdot \na {\uu}(\bx_i(t))\Big)\mr dt,\\
&\bx_i(t=0)=\bx_{i;0},\quad \bp_i(t=0)=\bp_{i;0},\quad m_i(t=0)=m_{i;0}.\end{cases}\label{2nd_Particle}
\end{align}
Here, the Brownian motions $\{ B_t^i\}_{i=1}^N, \{\mathbb B_t^i\}_{i=1}^N$ are independent and identically distributed.

\section{Enhanced Dissipation for Alternating Shear Flows}\label{sec:ED_alt_sf}
\subsection{General Setup}
In this section, we establish the enhanced-dissipation result in Theorem \ref{thm:lnrED_AltS}. 
 We rephrase the problem \eqref{eq:PS} on $\Torus^{2\dd}$ as follows 
\begin{align}\begin{split}
    &\pa_t \eta+\mathbf{u}\cdot \na_{\bx,\by} \eta=\nu \Delta \eta,\quad
    \eta(t=0,\bx,\by)=\eta_0(\bx,\by).
   \end{split}
\end{align}
Here, $(\bx,\by)\in \Torus_\bx^\dd\times\Torus_\by^\dd$ where $\Torus_\bx=\Torus_\by=\rr/2\pi \mathbb{Z}$. Without loss of generality, we assume that $\overline{\eta_0}=\frac{1}{|\Torus|^{2\dd}}\iint \eta_0\,\mathrm d\bx\dy=0$. Let $\{\mathbf{e}_j^\bx\}_{j=1}^\dd$ be the canonical basis vectors for $\Torus_\bx^\dd$ and $\{\mathbf{e}_{j}^\by\}_{j=1}^\dd$ be the canonical basis vectors for $\Torus^{\dd}_\by$. The time-dependent velocity field is defined as follows:
\begin{align}\label{alt_shear}
   \begin{split} \mathbf{u}(t,\bx,\by)&=\Phi_1(t)\lf[\sum_{j=1}^{\dd}\sin(y_j){\bf e}_j^\bx\rg]+\Phi_2(t)\lf[\sum_{j=1}^{\dd}\sin(x_j){\bf e}_{j}^\by\rg]
    =\Phi_1(t)u_1(\by)+\Phi_2(t)u_2(\bx).\end{split}
\end{align}
Here, $\Phi_1$ and $\Phi_2$ are two smooth temporal cutoff functions. {The cutoff $\Phi_1(t)$ is supported in the open set $\cup_{k\in \mathbb{N}}\delta_*^{-1}\nu^{-\frac12}(6k,3+6k)$. It is identically one for $t\in \delta_*^{-1}\nu^{-\frac12}[1+6k,2+6k]$ and is smooth and monotone in $ \delta_*^{-1}\nu^{-\frac12}[6k,1+6k]$ and $ \delta_*^{-1}\nu^{-\frac12}[2+6k,3+6k]$. Similarly, the temporal cutoff $\Phi_2(t)$ is supported in the open set $\cup_{k\in \mathbb{N}}\delta_*^{-1}\nu^{-\frac12}(3+6k,6(k+1))$ and is equal to $1$ on the interval $\delta_*^{-1}\nu^{-1/2}[4+6k,5+6k]$.} Here, the $\delta_*$ is a constant to be determined in Section \ref{sec:alt_sch}. The Fourier variables are 
\[ \bx=(x_1,x_2,\cdots, x_\dd)\Longrightarrow {\bf k}=(k_1,k_2,\cdots, k_\dd);\quad  \by=(y_1,y_2,\cdots, y_\dd)\Longrightarrow {\bf l}=(\ell_1,\ell_2,\cdots, \ell_\dd).\]

First of all, we develop the enhanced-dissipation estimate for the following system
\begin{align}
    &\pa_t\eta+ u_1(\by)\cdot \na_\bx \eta=\nu \Delta_{\bx,\by} \eta,\quad
    \eta(t=0,\bx,\by)=\eta_0(\bx,\by).
\end{align}
We take the Fourier transform in the $\bx$ variables to obtain 
\begin{align}
    \pa_t \hat{\eta}_{\bk}(t,\by)+\underbrace{\lf(\sum_{j=1}^{\dd}\sin(y_j)ik_{j}\rg)}_{=:i u_1\cdot \bk}\hat{\eta}_{\bk}(t,\by)=-\nu|\bk|^2\hat{\eta}_{\bk}(t,\by)+\nu\Delta_\by\hat{\eta}_{\bk}(t,\by).
\end{align}
Now, if we define 
$    f_\bk(t,\by):=\hat\eta_\bk(t,\by)e^{\nu|\bk|^2t}, $
we have the following equation 
\begin{align}\label{eq:hyp_x}
 \pa_t f_{\bk}(t,\by)+i u_1\cdot \bk f_{\bk}(t,\by)=\nu\Delta_\by f_{\bk}(t,\by).    
\end{align}
Next, we would like to show that if $|\bk|\neq0$, the solution experiences the enhanced dissipation at a rate $\mathcal{O}(\nu^{\f12})$. As long as $|\bk|\neq 0$, we have that $\max\{
|k_j|\}_{j=1}^{\dd}\neq 0$ and we define
\begin{align}\label{defn_k_m}
    m\in\operatorname*{argmax}_{1\leq j\leq\dd}|k_j|,\qquad k_m\neq 0,\quad \ep:=\nu|k_m|^{-1}. 
\end{align}
We consider the following Hypocoercivity functional:
\begin{equation}
\begin{split}
\mathcal{G}[f_\bk]:= &\|f_\bk\|_{2}^2+\al \phi\ep^{1/2}\|\pa_{y_m} f_\bk\|_{2}^2\\
 & \quad +\beta \phi^2 \Re\lan i \operatorname{sign}(k_m)\cos(y_m)f_\bk,\pa_{y_m} f_\bk \ran+\gamma\phi^3\ep^{-1/2} \|\cos(y_m) f_\bk\|_2^2.\label{hyp_fnctnl}
\end{split}
\end{equation}
Here, we consider only the maximizing $k_m$-direction and define $\ep:=\f{\nu}{|k_m|}$ and $\phi:=\min\{1,\sqrt{\nu}|k_m|^{1/2}t\}$. 
\begin{theorem}\label{thm:ED}
For a $C^2$ solution of \eqref{eq:hyp_x}, the Hypocoercivity functional has the following estimate for $0<\nu\leq |k_m|$, 
\begin{align}\|f_\bk(t)\|_{L^2}^2\leq \cG[f_\bk](t)\leq e\|f_{0;\bk}\|_{L^2}^2e^{-\mys{2}\delta_0\nu^\f12|k_m|^\f12 t},\quad \forall t\geq 0.
    \label{Hypo_est_ndeg}
    \end{align}
Here, the constant $\delta_0\in(0,1)$ is universal. 
\end{theorem}
As a direct corollary of the above theorem, we have the following estimate.
\begin{cor}\label{cor:ED_full} Consider the  $C^2$-solution to the hypoelliptic passive scalar equation with mean-zero initial data $\lan {f}_0\ran_\bx\equiv 0$, 
\begin{align}
\pa_t {f}+u_1(\by)\cdot\na_\bx  {f}=\nu \de_\by  {f},\quad {f}(t=0)= {f}_0. 
\end{align} 
Then, the following estimate holds for all $0<\nu\leq 1$, 
\begin{align}\| {f}(t)\|_{L_{\bx,\by}^2}^2\leq  e\| {f}_{0}\|_{L_{\bx,\by}^2}^2e^{-\mys{2}\delta_0\nu^\f12 t},\quad \forall t\geq 0.  \label{Hypo_est_ndeg_full}
\end{align}
Here, the constant $\delta_0\in(0,1)$ is universal.
\end{cor}
\begin{proof}
Since the solution preserves the mean-zero constraint, we have that $\displaystyle f(t,\bx,\by)=C\sum_{|\bk|\neq 0} f_{\bk}(t,\by)e^{i\bk\cdot\bx}$. A combination of the estimate \eqref{Hypo_est_ndeg} and the Plancherel equality yields the result\newline
$\displaystyle \|f(t)\|_{L^2_{\bx,\by}}^2=C_\dd\sum_{|\bk|\neq 0}\|f_\bk(t)\|_{L^2_\by}^2\leq C_\dd e\sum_{|\bk|\neq 0}\|f_\bk(0)\|_{L^2_\by}^2e^{-2\delta_0\nu^{\frac 12}t}=e\|f_0\|_{L^2_{\bx,\by}}^2e^{-2\delta_0\nu^{\frac 12}t}$.
\end{proof}

\subsection{Preliminary Lemmas}

The proof relies on a spectral inequality and a comparison lemma. We collect them in this section. 
\begin{lem}\label{lem:spec}
There exists a constant $\mathfrak{C}_{\mr{Spec}}$ such that the following estimate holds for every $j\in\{1,\ldots,\dd\}$, any $k_j\neq 0$, and $\nu\leq|k_j|$: 
    \begin{align}\label{spectral}
       \qquad  \lf(\frac{\nu}{|k_j|}\rg)^{1/2}\|f_\bk\|_{L^2(\Torus^{\dd})}^2\leq \frac{\nu}{|k_j|}\|\pa_{y_j} f_\bk\|_{L^2(\Torus^{\dd})}^2+\mathfrak{C}_{\mr{Spec}}\lf\| \cos({y_j}) f_\bk\rg\|_{L^2(\Torus^{\dd})}^2, \ \  \mathfrak{C}_{\mr{Spec}}=\f{3\pi^2}{2}+\frac{1024}{\pi^2}. 
    \end{align}    
\end{lem}
\begin{remark}
    This estimate is a consequence of the uncertainty principle in $\rr^{\dd}. $ The key point in the estimate is that the bound only depends on the regularity in the $y_j$-direction. 
\end{remark}
\begin{proof} The proof is standard, we refer the readers to \cite{CobleHe23} for details.
\end{proof}
\begin{lem}\label{lem:eqv}
    Assume the relation
  $
 {\beta^2 \leq \al\gamma.}
$ 
    Then, the following equivalence estimate concerning the functional $\mathcal{G}$ \eqref{hyp_fnctnl} holds
    \begin{align}
        \|f\|_2^2& + \frac{1}{2}\lf(\al\phi\ep^{1/2}\|\pa_{y_m}f\|_2^2 + \gamma\phi^3\ep^{-1/2}\|\cos(y_m)f\|_2^2\rg)\\
        & \leq \mathcal{G}[f] \leq \|f\|_2^2 + \frac{3}{2}\lf(\al\phi\ep^{1/2}\|\pa_{y_m}f\|_2^2 + \gamma\phi^3\ep^{-1/2}\|\cos(y_m)f\|_2^2\rg).\label{equiv_ndg}
    \end{align}
\end{lem}
\begin{proof}The proof is standard, we refer the readers to \cite[Lemma A.1]{GuHe24} for details.
\myh{We recall the definition of $\mathcal{G}$ \eqref{hyp_fnctnl}, and estimate $\mathcal{G}[f]$ using H\"older's inequality and Young's inequality,
    \begin{align}
        \mathcal{G}[f] \leq& \|f\|_2^2 + \al\phi\ep^{1/2}\|\pa_{y_m}f\|_2^2 + \beta\phi^2\|\cos(y_m)f\|_2\|\pa_{y_m}f\|_2 + \gamma\phi^3\ep^{-1/2}\|\cos(y_m)f\|_2^2\\
        \leq& \|f\|_2^2 + \frac{3\al}{2}\phi\ep^{1/2}\|\pa_{y_m}f\|_2^2 + \lf(\gamma + \frac{\beta^2}{2\al}\rg)\phi^3\ep^{-1/2}\|\cos(y_m)f\|_2^2.
    \end{align}
    Similarly, we have the lower bound:
    \begin{align}
        \mathcal{G}[f] \geq \|f\|_2^2 + \frac{\al}{2}\phi\ep^{1/2}\|\pa_{y_m}f\|_2^2 + \lf(\gamma - \frac{\beta^2}{2\al}\rg)\phi^3\ep^{-1/2}\|\cos(y_m)f\|_2^2.
    \end{align}
    Since \eqref{ndeg_bnd_req} implies that $\frac{\beta^2}{2\al} \leq \frac{\gamma}{2}$, we obtain that
    \begin{align}
        \|f\|_2^2& + \frac{1}{2}\al\phi\ep^{1/2}\|\pa_{y_m}f\|_2^2 + \frac{1}{2}\gamma\phi^3\ep^{-1/2}\|\cos(y_m)f\|_2^2\\ & \leq \mathcal{G}[f] \leq \|f\|_2^2 + \frac{3}{2}\al\phi\ep^{1/2}\|\pa_{y_m}f\|_2^2 + \frac{3}{2}\gamma\phi^3\ep^{-1/2}\|\cos(y_m)f\|_2^2.
    \end{align}
    This concludes the proof of the lemma.}
\end{proof}

\subsection{Energy Estimates}
By taking the time derivative of the hypocoercivity functional, \eqref{hyp_fnctnl}, we end up with the following decomposition:
\begin{align}\begin{split}
    \frac{d}{dt}\mathcal{G}[f_\bk(t)] =& \frac{d}{dt}\|f_\bk\|_2^2 + \al \ep^{1/2}\frac{d}{dt}\lf(\phi\|\pa_{y_m}f_\bk\|_2^2\rg) + \beta\frac{d}{dt}\lf(\phi^2\Re\langle i\cos(y_m)f_\bk,\pa_{y_m}f_\bk\rangle\rg) \\
    &+ \gamma\ep^{-1/2}\frac{d}{dt}\lf(\phi^3\|\cos(y_m)f_\bk\|_2^2\rg)
    =: T_{L^2} + T_\al+ T_\beta + T_\gamma.
\end{split}\label{ndeg_T_albe_term}
\end{align}
Standard energy estimates yield that 
\begin{align}\label{T_L2}
    T_{L^2}=-2\nu\|\na_\by f_\bk\|_{L^2}^2. 
\end{align}
The estimates for the $T_\al,\ T_\beta$, and $T_\gamma$ terms are collected in the following technical lemma.
\begin{lem}\label{lem:nondegenerate}
 For any constant $B>0$, the following estimates hold:
 \begin{align}
            T_\al \leq &\al\nu\mathbbm{1}_{t\leq (\nu|k_m|)^{-1/2}}\|\pa_{y_m}f\|_2^2 - 2\al\phi \nu^{\f32}|k_m|^{-\f12}\|\na_\by\pa_{y_m}f\|_2^2+ \frac{\beta\phi^2}{B}|k_m|\|\cos(y_m)f_\bk\|_{L^2}^2+\frac{B\al^2}{\beta}\nu\|\pa_{y_m}f\|_{L^2}^2,
 \\
        T_\beta  \leq& 2\beta \nu \|\na_{\by}f_\bk\|_2^2+ \frac{7\al}{4}\phi\ep^{1/2}\nu\|\na_\by\pa_{y_m}f_\bk\|_2^2 + \lf(\frac{3\beta^2}{4\al\gamma}\rg) \gamma\phi^3\ep^{-1/2}\nu \|\cos(y_m)\na_\by f\|_2^2\\
        &+\frac{\beta^2}{\al}\phi^3 \nu^\f12|k_m|^{\f12}\|f_\bk\|_{L^2}^2- \frac{3}{4}\beta|k_m|\phi^2\|\cos(y_m)f_\bk\|_{L^2}^2,\\
           T_\gamma\leq& 
           {3\gamma\phi^2\mathbbm{1}_{t\leq ({\nu}|k_m|)^{-\f12}}|k_m|\|\cos(y_m)f_\bk\|_2^2}+ {4\gamma\phi^3\nu^{\f12}{|k_m|^\f12}\|f_\bk\|_2^2 } - \gamma\phi^3\nu^{\f12}|k_m|^{\f12}\|\cos(y_m) \na_\by f_\bk\|_2^2.
 \end{align}
\end{lem}
\begin{proof}The proof of these lemmas are similar to the arguments in \cite[(A.11)--(A.15)]{GuHe24}. Hence we omit them for the sake of brevity. 
\end{proof}
These estimates allow us to prove Theorem \ref{thm:ED}.
\begin{proof}[Proof of Theorem~\ref{thm:ED}]
    If $T\leq 2\nu^{-1/2}{|k_m|^{-1/2}}$, then a standard $L^2$-energy estimate yields \eqref{Hypo_est_ndeg}. Hence, we assume $T>2\nu^{-1/2}{|k_m|^{-1/2}}$ without loss of generality. We distinguish between two time intervals, i.e., 
  $
        \mathcal{I}_1=[0, \nu^{-1/2}{|k_m|^{-1/2}}], \, \mathcal{I}_2=[\nu^{-1/2}{|k_m|^{-1/2}},T].
$ 
    We organize the proof into three steps. In step \# 1, we choose the $\al, \beta,\gamma$ parameters and derive the energy dissipation relation. In step \# 2, we estimate the functional $\cG$ in the time interval $\mathcal{I}_1$. In step \# 3, we estimate the functional $ \cG$ in the time interval $\mathcal{I}_2$ and conclude the proof.

    \noindent {\bf Step \# 1: Energy bounds.}       
    Combining the estimate \eqref{T_L2} and Lemma \ref{lem:nondegenerate}, we pick $B=4$ and use the fact that $\phi^3\leq \phi^2$ to obtain that
    \begin{align}\begin{split}
    \frac{d}{dt}\mathcal{G}[f(t)]    
\leq&  - \lf(2 -\al- \frac{4\al^2}{\beta}- 2\beta \rg)
{\nu\|\na_\by f_\bk\|_2^2}
- \f\al4\phi\nu^{\f32}|k_m|^{-\f12}\|\na_\by\pa_{y_m}f\|_2^2 \\
&- \lf(\frac{1 }{2} -3\frac{\gamma}{\beta}\rg)\beta \phi^2
{{|k_m|}\|\cos(y_m)f_\bk\|_{L^2}^2 }+\lf( \frac{\beta}{\al}+ 4 \f{\gamma}{\beta}\rg)\beta\nu^{1/2}\phi^3|k_m|^\f12\|f_\bk\|_2^2\\
        &- \lf(1 - \frac{3\beta^2}{4\al\gamma}\rg)\gamma\phi^3\nu^{\f12}|k_m|^\f12\|\cos(y_m)\na_\by f\|_2^2.
        \end{split}  
    \end{align}
    We choose $\al$, $\gamma$ in terms of $\beta$ as follows:  $
        \al = \frac{\beta^{1/2}}{4}, \, \gamma = 4\beta^{3/2} .$ 
Since we assume that $\nu\leq 1$, we can invoke the spectral inequality \eqref{spectral} to obtain that 
      \begin{align*}
    \frac{d}{dt}\cG[f(t)]\leq &  - \Bigl(1 -\sqrt{\beta}- 2\beta -20\beta^{\f32}\Bigr)\nu\|\na_\by f\|_2^2      - \lf(\frac{1}{2} -12\sqrt{\beta}-\mathfrak{C}_{\text{Spec}} 20\sqrt{\beta}\rg)\beta |k_m|\phi^2\|\cos(y_m)f_\bk\|_2^2 \\
        &  - \frac{\gamma}{4}\phi^3\nu^{1/2}|k_m|^\f12\|\cos(y_m)\na_{\by}f_\bk\|_2^2.
    \end{align*}
    Hence, we can choose $
    \beta=\beta(\mathfrak{C}_{\mathrm{Spec}})<1$, 
sufficiently small that 
    \begin{equation}  \begin{split} \ \  &\frac{d}{dt}\mathcal{G}[f(t)]\leq- \frac{1}{2}\nu\|\na_\by f_\bk\|_2^2-\frac{\beta}{4}\phi^2|k_m|\|\cos(y_m)f_\bk\|_2^2 \\
   & \ \ \ \ \leq-\frac{\beta\phi^2}{
    {4}\mathfrak{C}_{\mathrm{Spec}}}\nu^{1/2}|k_m|^{1/2}\|f_\bk\|_2^2-\frac{1}{4}\nu^{1/2}|k_m|^{1/2}\ep^{1/2}\phi\|\pa_{y_m} f_\bk\|_2^2-\frac{\beta}{8}\nu^\f12|k_m|^{\f12}\phi^2\ep^{-\f12}\|\cos(y_m)f_\bk\|_2^2\\
    &\ \ \ \ \leq -2\delta_0(\beta,\mathfrak{C}_{\text{Spec}}^{-1})\nu^{1/2}|k_m|^{1/2}\phi^2\mathcal{G}[f],\quad0<\delta_0(\beta,\mathfrak{C}_{\text{Spec}}^{-1})\leq\frac{1}{8}.\end{split}  \label{dfn_del_0} 
    \end{equation} Here, in the last line, we have invoked the spectral inequality \eqref{spectral} and the comparison lemma \ref{lem:eqv}. Finally, we observe that the parameter $\delta_0$ is universal.
    
    \noindent
    {\bf Step \# 2: Initial time layer estimate.}
    Thanks to the energy dissipation relation \eqref{dfn_del_0}, we obtain that
  $
        \cG[f_\bk(t)]\leq \|f_{0;\bk}\|_{L^2}^2\leq \sqrt{e}\|f_{0;\bk}\|_{L^2}^2e^{-2\delta_0 \nu^{1/2}|k_m|^{1/2}t},\,\forall t\in[0, \nu^{-1/2}|k_m|^{-1/2}].
$
    
    \noindent
    {\bf Step \# 3: Long-time estimate.}
    Assume $t\geq \nu^{-1/2}|k_m|^{-1/2}$. Thanks to the energy dissipation relation \eqref{dfn_del_0}, we obtain    $
        \frac{d}{dt}\cG[f]\leq -2\delta_0\nu^{1/2}|k_m|^{1/2}\cG[f].
 $ Hence, we obtain that 
    \begin{align}
        \cG[f(t)] \leq& \cG[f(\nu^{-1/2}|k_m|^{-1/2})]e^{-2\delta_0\nu^{1/2} {|k_m|^{1/2}}\lf(t-\nu^{-1/2} {|k_m|^{-1/2}}\rg)}
        \leq \sqrt{e}\,\cG[f(\nu^{-1/2}|k_m|^{-1/2})] e^{-2\delta_0\nu^{1/2} {|k_m|^{1/2}}t}. 
    \end{align}
    Now, the results from Steps 2 and 3 yield exponential decay with a universal coefficient. This concludes the proof of  \eqref{Hypo_est_ndeg}.
\end{proof}
\subsection{The Mixing Estimates}
In this section, we derive the mixing estimates. The key is to establish bounds for the following vector fields:
\begin{align}
    &\Gamma^+_m f_\bk:= \pa_{y_m} f_\bk+\A^+(t,\nu,k_m)\cos(y_m)ik_m   f_\bk,\quad \A^+(t,\nu,k_m):=\frac{\tanh((1-i)\sqrt{\nu|k_m|}t)}{(1-i)\sqrt{\nu|k_m|}},\\ &\Gamma_m^-f_\bk:=\pa_{y_m} f_\bk+ \A^-(t,\nu,k_m) \cos(y_m) ik_m f_\bk,\quad \A^-(t,\nu,k_m):=\frac{\tanh((1+i)\sqrt{\nu|k_m|}t)}{(1+i)\sqrt{\nu|k_m|}}.
\end{align}
These vector fields mimic the behavior of the vector field $\pa_y +t\cos(y_m)ik_m  $, which commutes with the inviscid dynamics $\eqref{eq:hyp_x}_{\nu=0}$.  For the sake of simplicity, we further define the quantities
\begin{align}\label{msc_H}\begin{split}
   &{\msc H_+(t,\nu,k_m)}:=-2i\nu |k_m|\mathcal{A}^+(t,\nu,k_m)=(1-i)\sqrt{\nu|k_m|}\tanh((1-i)\sqrt{\nu|k_m|} t),\\
   &{ \mathscr H_-(t,\nu,k_m):=2i \nu |k_m|\mathcal{A}^-(t,\nu,k_m)=(1+i)\sqrt{\nu|k_m|}
\tanh((1+i)\sqrt{\nu|k_m|}t).}
\end{split}
\end{align}
Since in the sequel, we will mostly work with $\mathscr{H}_+$, we use the simplified notation $ \msc H= \msc H_+.$ Here, we observe that 
\begin{align}\label{AHpm_est}
   |\A^\pm(t)|=\frac{|\mathscr{H}_\pm(t)|}{2\nu|k_m|}\approx \min\{t,\frac{1}{\nu^\f12|k_m|^\f12}\}=\nu^{-\f12}|k_m|^{-\f12}\phi(t).
\end{align}
Here, the $\phi(t)$ function is defined in \eqref{hyp_fnctnl}.  
Moreover, we can check that 
\begin{align}
    &\mys{ 
 \Re \mathscr{H}_{ \pm}(t)=\sqrt{\nu|k_m|} \frac{\sinh (2\sqrt{\nu\left|k_m\right|}t)-\sin (2 \sqrt{\nu\left|k_m\right|}t)}{\cosh (2 \sqrt{\nu\left|k_m\right|}t)+\cos (2 \sqrt{\nu\left|k_m\right|}t)}\geq 0}. 
\end{align}  
Computation yields that 
\begin{align}
\label{eq_Gamma}      &\lf(\pa_t+i u_1(\by)\cdot \bk-\nu\de_\by+\msc H_\pm\rg)\Gamma_m^\pm f_\bk=\msc H_\pm\pa_{y_m}((1\mp\sin(y_m)\mathrm{sign}(k_m))f_\bk)\pm\frac{1}2\mathrm{sign}(k_m)\mathscr{H}_\pm\cos(y_m)f_\bk.
\end{align}

Now we define two smooth cutoff functions, 
\begin{align}
    \chi_+:=\pw \chi0+\pw \chi1,\quad \chi_-:=\pw \chi0+\pw \chi 2.
\end{align}
Thus, $\chi_+$ (respectively, $\chi_-$) vanishes in an $\mathcal{O}(1)$-neighborhood of $-\pi/2$ (respectively, $\pi/2$). 

First  we have the following energy estimate for the weighted $L^2$-norm $\|\chi_+\Gamma_m^+\|_{L^2}^2$. The proof is similar to the arguments in \cite[Lemma A.3]{GuHe24}. and is omitted  for the sake of brevity.
\begin{lem}There exists a constant $C$ such that, for $k_m>0$, the following estimate holds:
    \begin{align}
    \|\chi_+ \Gamma_m^+f_\bk(t)\|_{L^2}^2+\|\chi_- \Gamma_m^-f_\bk(t)\|_{L^2}^2\leq C\|f_{0,\bk}\|_{H_\by^1}^2. \label{Ga_est}
\end{align}
For $k_m<0$, we have the corresponding estimate, $\displaystyle     \|\chi_+ \Gamma_m^-f_\bk(t)\|_{L^2}^2+\|\chi_- \Gamma_m^+f_\bk(t)\|_{L^2}^2\leq C\|f_{0,\bk}\|_{H_\by^1}^2$. 
\end{lem}

With the vector-field estimates established, we are ready to present the following mixing theorem. 
\begin{theorem}[Linear Mixing]
\label{thm:mix} Consider (regular) solutions $\eta$ to the following equation associated with \eqref{eq:hyp_x},
\begin{align} \label{eq:hypo}
\begin{split}& \pa_t \eta +u_1(\by)\cdot \na_\bx \eta =\nu\Delta_\by \eta, \quad \eta(t=0)=\eta_0,\quad \nu\in(0,1],\\ 
&u_1(\by)=\sum_{j=1}^\dd \sin(y_j)\mathbf{e}_j^{\bx},\quad \br{\eta_0}_\bx\equiv 0.\end{split}   
\end{align} 
Then the following mixing estimate holds for all $t>0$ and all $F\in L_t^\infty H^{\sigma}_{\bx,\by}, \ \sigma:=\max\{1,\dd/2+\varsigma\}$:
\begin{equation} \label{linear_ID} 
   \lf|\iint  \eta(t) \overline{  F(t) } \dx\mathrm{d}\by\rg|\leq  C_{\dd,\varsigma}\min\lf\{\frac{\nu^{1/4}}{\min\{1,\nu^{1/4}t^{1/2}\}},e^{-\delta_{0}\nu^{1/2}t}\rg\}\| \eta_0\|_{H_{\bx,\by}^{\sigma}}\| F\|_{L_t^\infty H_{\bx,\by}^{\sigma}}. 
\end{equation}
Here, $C_{\dd,\varsigma}\geq1$ depends only on $\dd$ and $\varsigma$, $\varsigma>0$ is an arbitrarily small positive number, $f\vee g:=\max\{f,g\}$, and $\delta_{0}$ is defined in Theorem \ref{thm:ED}.  
\end{theorem}
\begin{proof}
we begin by  observing that the solutions to \eqref{eq:hyp_x} and \eqref{eq:hypo} are related by $\eta(t,\bx,\by)=\sum_{|\bk|\neq 0}\eta_\bk(t,\by)e^{i\bk\cdot \bx}$. Hence, by the Plancherel equality, we have the following relation
\begin{align}
 \lf|\iint   \eta(t,\bx,\by) \overline{  F(t,\bx,\by) } \dx\mathrm{d}\by\rg|=C\lf|\sum_{|\bk|\neq0}\int \eta_\bk (t,\by)\overline{F_\bk(t,\by)} \dy\rg|.\label{linear_ID_pf} 
\end{align} Therefore, it is sufficient to establish the estimate for the right-hand side. 
Furthermore, we note that for $t\leq 1$, the estimate \eqref{linear_ID} is a direct consequence of the dissipative nature of the $L^2$-norm of $ \eta$: 
\begin{align*}
\  \lf|\iint \eta(t) F \dx\mathrm{d}\by\rg|\leq \|  \eta(t)\|_{L^2}\|F\|_{L_t^\infty L^2}\leq \| \eta_0\|_{L^2}\|F\|_{L_t^\infty L^2}\leq e^{\delta_0}e^{-\delta_0\nu^{1/2}t} \|\eta_0\|_{L^2}\|F\|_{L_t^\infty  H^{\sigma}},\quad \forall t\leq 1.
\end{align*}
Hence, without loss of generality, we assume that $t> 1$ in the remainder of the proof.  We follow the idea of the proof of Proposition 1.7 in the paper \cite{CotiZelatiDietertGerardVaret22}. For a general test function $F\in H^{\sigma}_{\bx,\bp}(\Torus^{\dd}),$ we decompose the right-hand side of \eqref{linear_ID_pf}  as 
\begin{align}
\sum_{|\bk|\neq0}\int \eta_\bk \overline{F_\bk} \dy=&\sum_{|\bk|\neq0}\lf(\int \eta_\bk \overline{F_\bk}\chi_+  d\by+\int \eta_\bk \overline{F_\bk}(1-\chi_+) \dy\rg)
=:\sum_{|\bk|\neq  0 }\lf(I_{\bk;+}+I_{\bk;-}\rg).
\end{align}
 As explained in the paper \cite{CotiZelatiDietertGerardVaret22}, a symmetry consideration yields that it is enough to consider the first part of the expression \mys{under the assumption that $k_m>0$}. For the second part, one can use the vector field $\Gamma_m^-$ and associated cutoffs to derive similar estimates. \mys{The $k_m<0$ case can be treated in a similar manner, and we omit it for the sake of brevity.} One can introduce another cutoff function $\chi_{\zeta}(y_m)$ such that it is $1$ in a $\zeta$-neighborhood of $\pi/2$ (with $\zeta\in(0,1)$ to be chosen later). Moreover, $\|\pa_{y_m} \chi_{\zeta}\|_{L_\by^\infty}\leq C\zeta^{-1}
$. With this cutoff, we further decompose the $I_{\bk;+}$ as follows
\begin{align}
   I_{\bk;+}=& \int_{\Torus^\dd} \eta_\bk \overline{F_\bk} \chi_{\zeta}\chi_{+} \dy+\int_{\Torus^\dd} \eta_\bk \overline{F_\bk} (1-\chi_{\zeta})\chi_{+} \dy
   =:I_{\bk;+}^{(1)}+I_{\bk;+}^{(2)}.
\end{align}
For the $I_{\bk;+}^{(1)}$-term, we estimate it using the length of the interval, the observation that $\|\eta_\bk(t)\|_{L^\infty}\leq \|\eta_\bk(0)\|_{L^\infty}$ and the {Sobolev embedding for $\Torus^{\dd}$}:
\begin{align*}
    |I_{\bk;+}^{(1)}|\leq C\zeta \|F_\bk(t)\|_{ L_\by^\infty}\|\eta_\bk(0)\|_{L^\infty}\leq C_\varsigma\zeta\|F_\bk(t)\|_{H_{\by}^{\sigma}}\|\eta_\bk(0)\|_{H_{\by}^{\sigma}}.
\end{align*}
Next we estimate the $I_{\bk;+}^{(2)}$ term using the observation that $|\cos(y_m)|>0$ on this interval and $(\nu|k_m|^{-1})^{\f12}(\Gamma_m^+\eta_\bk-\pa_{y_m} \eta_\bk)=\frac{i-1}{2}\tanh((1-i)\sqrt{\nu|k_m|}t)\cos(y_m)\eta_\bk$:
\begin{align}
    |I_{\bk;+}^{(2)}|
   & \leq\nu^{1/2}|k_m|^{-1/2}\lf|\int \cos(y_m)\frac{\Gamma_m^+\eta_\bk-\pa_{y_m}\eta_\bk}{\frac{i-1}{2}\tanh((1-i)\sqrt{\nu|k_m|}t)}\frac{\overline{F_\bk}}{\cos^2(y_m)}(1-\chi_{\zeta})\chi_{+} \dy\rg| \\
   &\leq \int \nu^{1/2}|k_m|^{-1/2} |\cos(y_m)|\lf|\frac{\Gamma_m^+\eta_\bk}{\frac{i-1}{2}\tanh((1-i)\sqrt{\nu|k_m|}t)}\rg|\lf|\frac{\overline{F_\bk}}{ \cos^2(y_m)} \rg| (1-\chi_{\zeta})\chi_{+} \dy\\
    &\quad+\lf|\frac{\nu^{1/2}|k_m|^{-1/2}}{\frac{i-1}{2}\tanh((1-i)\sqrt{\nu|k_m|}t)}\rg| \lf|\int\cos(y_m)\pa_{y_m}\eta_\bk \frac{\overline{F_\bk}}{\cos^2(y_m)} (1-\chi_{\zeta})\chi_{+} \dy\rg|=:T_4+T_5.
\end{align}
For the $T_4$-term, we estimate it with the Sobolev embedding, \mys{the equivalence $\phi(t)\approx|\tanh((1\pm i)\sqrt{\nu|k_m|}t)|$,} and \eqref{AHpm_est} 
as follows:\begin{align}
    T_4\leq& \frac{C\nu^{1/2}}{|k_m|^{1/2}\phi(t)}\|F_\bk(t)\|_{ L_\by^\infty} \int_{|y_m-\pi/2|\geq \zeta} \mys{\chi_+}\frac{|\Gamma_m^+ \eta_\bk| }{|\cos(y_m)|} \dy \\
    \leq&  \frac{C\nu^{1/2}\|F_\bk(t)\|_{ H_\by^{\sigma}}\|\mys{\chi_+}\Gamma_m^+\eta_\bk\|_{L^2}\zeta^{-\f12}}{|k_m|^{1/2}\min\{\nu^{1/2}|k_m|^{1/2}t,1\}} 
    \leq  \frac{C\nu^{1/2}\|F_\bk(t)\|_{ H_\by^{\sigma}}\|\mys{\chi_+}\Gamma_m^+\eta_\bk\|_{L^2}\zeta^{-\f12}}{ |k_m|^{1/2}\min\{\nu^{1/2} t,1\}},\quad \forall t>0.
\end{align}
To estimate the second term $T_5$, we recall the relation \eqref{AHpm_est}, integrate by parts, and estimate each resulting term as follows:
\begin{align}
&T_5\leq C\lf|\frac{ \nu^{1/2} |k_m|^{-1/2}}{\frac{i-1}{2}\tanh((1-i)\sqrt{\nu|k_m|}t)}\rg| \lf|\int\eta_\bk\pa_{y_m}\lf( \frac{\overline{F_\bk}}{\cos (y_m)} (1-\chi_{\zeta})\chi_{+}\rg) \dy\rg|\\
&\leq  \frac{C\nu^{\f12}|k_m|^{-\f12}}{\min\{\nu^{1/2}t,1\} }\|\eta_\bk\|_{L^\infty}\lf(\|F_\bk(t)\|_{ H_\by^1}\Bigl(\int \frac{(1-\chi_{\zeta})^2\chi_{+}^2}{|\cos(y_m)|^2} \dy\rg)^{1/2}+\|F_\bk(t)\|_{L_{\by}^\infty}\int \frac{ |\sin(y_m)|(1-\chi_{\zeta})\chi_+}{|\cos(y_m)|^2} \dy\\
&\hspace{4.5cm}+\|F_\bk(t)\|_{L_{\by}^\infty}\int \frac{ |\pa_{y_m}\chi_{\zeta}|\chi_+}{|\cos(y_m)|} \dy+\|F_\bk(t)\|_{L_{\by}^\infty}\int \frac{ |\pa_{y_m}\chi_+|(1-\chi_{\zeta})}{\lf|\cos(y_m)\rg|}\dy\Bigr)\\
&\leq  \frac{C\nu^{\f12}|k_m|^{-\f12}}{\min\{\nu^{1/2}t,1\} }\|\eta_{0;\bk}\|_{L^\infty}\|F_\bk(t)\|_{ H_\by^{\sigma}}\lf(\zeta^{-\frac12}+\zeta^{-1}+|\log\zeta|\rg)
\leq  \frac{C\nu^{1/2}\|\eta_{0;\bk}\|_{H^{\sigma}}\|F_\bk(t)\|_{ H_\by^{\sigma }}}{|k_m|^{1/2}\min\{\nu^{1/2} t,1\}}\frac{1}{\zeta}.
\end{align}
Hence, we observe that if we set $\zeta=\left(\frac{\nu^{1/2}}{\min\{\nu^{1/2} t,1\}}\right)^{1/2}$ and invoke the vector-field bound \eqref{Ga_est}, the following estimate holds
\begin{align}
    |I_{\bk;+}|\leq\frac{C\nu^{1/4}}{ \min\{\nu^{1/4} t^{1/2},1\}}\|\eta_\bk(0)\|_{H^{\sigma}}\|F_\bk(t)\|_{H_\by^{\sigma}},\quad \forall t\geq0.
\end{align}
Summing all the $\bk$-components, we obtain that 
\begin{align}\label{mix}
&\lf|\sum_{|\bk|\neq0}\int \eta_\bk (t,\by)\overline{F_\bk(t,\by)} \dy\rg|
\leq \sum_{|\bk|\neq 0}\frac{C \nu^{1/4}}{ \min\{\nu^{1/4}t^{1/2},1\}}\|\eta_\bk(0)\|_{H_\by^{\sigma}}\|F_\bk(t)\|_{ H_\by^{\sigma}}
\\
&\leq \frac{C\nu^{1/4}}{\min\{\nu^{1/4}t^{1/2},1\}}\|\eta_0\|_{H_{\bx,\by}^{\sigma}}\|F(t)\|_{ H_{\bx,\by}^{\sigma}}\leq \frac{C\nu^{1/4}}{\min\{\nu^{1/4}t^{1/2},1\}}\|\eta_0\|_{H_{\bx,\by}^{\sigma}}\|F\|_{L_t^\infty H_{\bx,\by}^{\sigma}}.
\end{align}
Moreover, thanks to the enhanced dissipation estimate in Theorem \ref{thm:ED}, we have the following estimate for all times
\begin{align}
\lf|\sum_{|\bk|\neq 0}\int\eta_\bk(t,\by)\overline{F_\bk(t,\by)}\dy\rg|\leq C\|\eta(0)\|_{L^2_{\bx,\by}}e^{-\delta_{0}\nu^{1/2}t}\|F\|_{L_t^\infty L_{\bx,\by}^2}.\label{ed}
\end{align}
Combining the estimates \eqref{mix} and \eqref{ed}, we obtain the result.
\end{proof}

\subsection{The Alternating Scheme}\label{sec:alt_sch}
In this section, we consider the alternating shear flow scenario. We recall that $d=2\dd$. Recall the smooth cutoff functions $\Phi_1(t), \Phi_2(t)$ and the alternating shear 
$\bf u$ \eqref{alt_shear}. We set $\delta_*:=5^{-1}\delta_0$ where $\delta_0$ is defined in \eqref{Hypo_est_ndeg} and  \eqref{dfn_del_0}. Without loss of generality, we assume that $\overline{\eta_0}=0$, which is a property that is preserved by the dynamics.

The goal is to show that at time $t=2\mathscr T:=30\delta_0^{-1}\nu^{-\frac12}$, the $L^2$-norm $\eta-\overline{\eta}$ decays to $e^{-1}$ times its original value. On the support of the first temporal cutoff $\Phi_1(t)$, there are two distinct regions: a) the effective region $t\in \pw IE_1:=[1,2]\times 5\delta_0^{-1}\nu^{-\frac12}$; b) the warm-up/cool-down region $\mathrm{supp}\ \Phi_1\backslash\pw IE_1$. For $t\in\mathrm{supp}\ \Phi_1(t)$, we observe that the streamline average 
 $   \lan \eta\ran_\bx(t,\by):=\frac{1}{|\Torus|^\dd}\int \eta(t,\bx,\by) \dx$ 
solves the heat equation
\begin{align}
    \pa_t \lan \eta\ran_\bx(t,\by)+\Phi_1(t)\lan u_1(\by)\cdot \na_\bx \eta\ran_\bx=\nu \Delta_\by \lan \eta\ran_\bx(t,\by)\;\Longrightarrow \;  \pa_t \lan \eta\ran_\bx(t,\by)=\nu \Delta_\by \lan \eta\ran_\bx(t,\by).
\end{align} 
Hence, 
\begin{align}
\|\br \eta_\bx(t)\|_{L^2}\leq \|\br \eta_\bx(0)\|_{L^2},\quad \forall t\in\mathrm{supp}\ \Phi_1.\label{br_eta_x}
\end{align}
Next we consider the fluctuation $\eta-\br\eta_\bx$. The fluctuation solves the passive scalar equation, which dissipates the $L^2$-norm. Moreover, on the effective region $\pw IE_1$, the fluctuation experiences strong enhanced dissipation \eqref{Hypo_est_ndeg}. Hence, it is not surprising to derive the following estimate at the transition time $\mathscr{T}$: 
$
\|\eta-\br \eta_\bx\|_{L^2}^2(\msc T)\leq\|\eta-\br \eta_\bx\|_{L^2}^2(10\delta_0^{-1}\nu^{-\frac12})\underbrace{\leq}_{\rm E.D.} e^{-4} \|\eta-\br \eta_\bx\|_{L^2}^2(5\delta_0^{-1}\nu^{-\frac12})\leq \frac{1}{e^4}\|\eta_0-\br {\eta_0}_\bx\|_{L^2}^2.
$ 
This concludes the argument in the first region. 

For $t\in \mathrm{supp}\ \Phi_2$, an argument similar to that in \eqref{br_eta_x} shows that the $L^2$-norm of the $\lan \eta\ran_{\by}$ is non-expansive on  the support of $\Phi_2(t)$. 
Moreover, we observe that at the transition time $\msc T$
\begin{align}
    \|\lan \eta\ran_{\by}\|_{L^2_{\bx,\by}}(\msc T)\leq \|\eta-\lan \eta\ran_\bx\|_{L^2_{\bx,\by}}(\msc T)\leq \frac{1}{e^2}\|\eta_0-\lan \eta_0\ran_\bx\|_{L^2_{\bx,\by}}\leq \frac{1}{e^2}\|\eta_0\|_{L^2_{\bx,\by}}.
\end{align}
As a consequence,
$
    \|\lan \eta\ran_{\by}(t)\|_{L^2_{\bx,\by}}\leq \frac{1}{e^2}\|\eta_0\|_{L^2_{\bx,\by}},\,\forall t\in \operatorname{supp}\Phi_2. 
$ 
Now by Theorem \ref{thm:ED}, we have that at time $2\msc T$, 
$
    \|\eta-\lan \eta\ran_{\by}\|_{L^2_{\bx,\by}}(2\msc T)\leq \frac{1}{e^2} \|\eta-\lan \eta\ran_{\by}\|_{L^2_{\bx,\by}}(\msc T)\leq \frac{1}{e^2} \|\eta_{0}\|_{L^2_{\bx,\by}}.
$ 
Hence, $
    \|\eta(2\msc T)\|_{L^2_{\bx,\by}}\leq  \|\lan \eta\ran_{\by} \|_{L^2_{\bx,\by}}(2\msc T)+  \|\eta-\lan \eta\ran_{\by}\|_{L^2_{\bx,\by}}(2\msc T)\leq \frac{2}{e^2} \|\eta_{0}-\overline{\eta_{0}}\|_{L^2_{\bx,\by}}.
$ 
By the same argument, we see that for $s,t\in 2\msc T \mathbb N$, the following estimate holds:
\begin{align}\label{discrete_ED}
 \|\eta(s+t)-\overline{\eta}(s+t)\|_{L^2_{\bx,\by}}\leq \lf(\frac{2}{e^2}\rg)^{\frac{t}{2\msc T}} \|\eta(s)-\overline{\eta}(s)\|_{L^2_{\bx,\by}}.
\end{align}
{
For general $s,t \geq 0 $, we find the smallest integer $N$ and largest integer $M$ so that $(2\msc T)N\geq s,\, (2\msc T)M\leq s+t,\, M, N\in \mathbb{N}.$  
Note that if $t\leq 2\msc T $, then the estimate \eqref{ED_V} is direct:
\begin{align}
\|\eta(s+t)-\overline{\eta}(s+t)\|_2\leq \|\eta(s)-\overline{\eta}(s)\|_2\leq e^2\|\eta(s)-\overline{\eta}(s)\|_2 e^{-\frac{1}{2\msc T } t},\quad 0\leq t\leq 2\msc T .
\end{align}
Hence we assume $t>2\msc T $ and observe that $2\msc T (M-N)\geq t-4\msc T$. Now we apply the estimate  \eqref{discrete_ED} with $s,t\in 2\msc T\mathbb{N}$, and the non-increasing nature of $L^2$-norm of the solutions to derive that
\begin{align*}
\|\eta-\overline{\eta}\|_2(s+t)\leq&\|\eta-\overline\eta\|_2(M2\msc T)\leq \|\eta-\overline\eta\|_2( N2\msc T)\lf(\frac{2}{e^2}\rg)^{M-N}\leq \|\eta-\overline\eta\|_2(s) e^{-\frac{1}{2\msc T}(2\msc T)(M-N)}\\
\leq&\|\eta-\overline\eta\|_2(s) e^{-\frac{1}{2\msc T} (t-4\msc T)}=e^2\|\eta-\overline\eta\|_2(s)e^{-\frac{1}{2\msc T} t},\quad \forall s,t\geq 0.
\end{align*}This concludes the proof of (\ref{ED_V}) in the general case.}

\section{Mass-Searching Dynamics}\label{sec:mass_search}
\begin{proof}[Proof of Theorem \ref{thm:mass_1}] Consider the ratio $v:=\omega e^{-\ww}$, which solves the equation $
\pa_t v+  \vv \cdot\na_\bx v  =   \nu\de_\bx v .
$ 
Direct computation yields that the spatial average of $v$ satisfies the relation $\overline{v}(t)=\overline{v}(t=0)=M$ for all $t\geq 0.$ As a consequence of Theorem \ref{thm:lnrED_AltS}, we obtain that 
 $
\|\omega e^{-\ww}- M\|_{L^2}(t)=\|v- \overline{v}\|_{L^2}(t)\leq C \|v(t=0)- \overline{v}\|_{L^2}e^{-\delta_{\rm ED}\nu^{1/2}t}
\leq C \||\Torus|^de^{-\ww}-M\|_{L^2}e^{-\delta_{\rm ED}\nu^{1/2}t}. 
$ 
Hence, the solution converges to the normalization constant $M$ in $L^2$ at the enhanced exponential rate $\mathcal{O}(\nu^{1/2})$. 
\end{proof}

\myh{
\begin{proof}[Proof of Theorem \ref{thm:mass_2}] To derive the convergence \eqref{mass_converge_2}, we observe that the ratio $g:=fe^{-\ww}$ solves
\begin{align*}
\pa_t g+\na_\bx\cdot(\bv(\bp) g)=\nu\de_\bp g,\quad
g(t=0,\bx,\bp)=|\Torus|^de^{-\ww}.
\end{align*}
Furthermore, the $\bx$-average $\lan g\ran$ solves the equation 
\begin{align}
&\pa_t\lan g\ran=\nu\de_\bp \lan g\ran,\quad
\lan g\ran(t=0,\bp)\equiv\int_{\Torus^d}e^{-\ww(\bx)}\dx=M.
\end{align}
Hence, the $\bx$-average satisfies $\lan g\ran(t,\bp)\equiv M$, and the fluctuation $\wt g:=g-\lan g\ran$ solves the PDE
\begin{align}&\pa_t \wt g+\na_\bx\cdot(\bv \wt g)=\nu\de_\bp \wt g,\quad
\wt g(t=0,\bx,\bp)=|\Torus|^de^{-\ww}-M.
\end{align} 
The enhanced dissipation \eqref{Hypo_est_ndeg_full} implies that there is a universal constant $\delta\in(0,1)$ so that
\begin{align}
\|\wt g(t)\|_{L_{\bx,\bp}^2}\leq C\||\Torus|^de^{-\ww}-M\|_{L_{\bx,\bp}^2}e^{-\delta \nu^{1/2}t}. 
\end{align}Hence, we have that
$
\|f(t)e^{-\ww}-M\|_{L_{\bx,\bp}^2}\leq C\||\Torus|^de^{-\ww}-M\|_{L_{\bx,\bp}^2}e^{-\delta \nu^{1/2}t}.
$ This concludes the proof. 
\end{proof}
}
\bibliographystyle{abbrv}
\bibliography{nonlocal_eqns}

\end{document}